\documentclass[]{interact}

\usepackage[ruled,vlined,linesnumbered,nosemicolon]{algorithm2e} 
\usepackage[misc,geometry]{ifsym} % Envelope symbol \Letter in title page

\usepackage{amsmath}
\usepackage{amssymb}
\usepackage{bm}
\usepackage{enumerate}
\mathchardef\mhyphen="2D 
\usepackage[titletoc,toc,title]{appendix}
\usepackage{footnote}
\makesavenoteenv{tabular}
\makesavenoteenv{table}
\usepackage{graphicx}

\usepackage[font=small,labelfont=bf]{caption}
\usepackage[numbers,sort&compress]{natbib}
\usepackage{float}
\usepackage{url}
\usepackage[bottom]{footmisc} % Keep the footnote at the bottom of pages
\newtheorem{theorem}{Theorem}[section]
\newtheorem{lemma}[theorem]{Lemma}
\newtheorem{corollary}[theorem]{Corollary}
\theoremstyle{plain}

\usepackage{multirow}% http://ctan.org/pkg/multirow

\usepackage{mleftright}

\usepackage[svgnames]{xcolor, colortbl}

\DeclareMathOperator*{\argmin}{argmin}

\makeatletter
\newcounter{manualsubequation}
\renewcommand{\themanualsubequation}{\alph{manualsubequation}}
\newcommand{\startsubequation}{%
  \setcounter{manualsubequation}{0}%
  \refstepcounter{equation}\ltx@label{manualsubeq\theequation}%
  \xdef\labelfor@subeq{manualsubeq\theequation}%
}
\newcommand{\tagsubequation}{%
  \stepcounter{manualsubequation}%
  \tag{\ref{\labelfor@subeq}\themanualsubequation}%
}
\let\subequationlabel\ltx@label
\makeatother

\usepackage{etoolbox}

\makeatletter
\patchcmd{\endalign}{\restorealignstate@}{\global\let\df@label\@empty\restorealignstate@}{}{}
\makeatother

\definecolor{rev}{RGB}{0,0,0}
\begin{document}
%%%%%%%%%%%%%%%%%%%%%%%%%%%%%%%%%%%%%%%%%%%%%%%%%%%%%%%%%%%%%%%%%%%%%%%%%%%%%%%%%%%%%%%%%%%%%%
% 		TITLE PAGE
%%%%%%%%%%%%%%%%%%%%%%%%%%%%%%%%%%%%%%%%%%%%%%%%%%%%%%%%%%%%%%%%%%%%%%%%%%%%%%%%%%%%%%%%%%%%%%

\title{Data-driven distributionally robust risk parity portfolio optimization}

\author{
\name{Giorgio Costa\textsuperscript{a}\thanks{Giorgio Costa. Email: gcosta@mie.utoronto.ca} and Roy H. Kwon\textsuperscript{a}\thanks{Roy H. Kwon. Email: rkwon@mie.utoronto.ca}}
\affil{\textsuperscript{a}Department of Mechanical and Industrial Engineering, University of Toronto, 5 King's College Road, Toronto, Ontario M5S 3G8, Canada}
}

\maketitle

\begin{abstract}
We propose a distributionally robust formulation of the traditional risk parity portfolio optimization problem. Distributional robustness is introduced by targeting the discrete probabilities attached to each observation used during parameter estimation. Instead of assuming that all observations are equally likely, we consider an ambiguity set that provides us with the flexibility to find the most adversarial probability distribution based on the investor's {\color{rev}desired degree of robustness}. This allows us to derive robust estimates to parametrize the distribution of asset returns without having to impose any particular structure on the data. The resulting distributionally robust optimization problem is a constrained convex--concave minimax problem. Our approach is financially meaningful and attempts to attain full risk diversification with respect to the worst-case instance of the portfolio risk measure. We propose a novel algorithmic approach to solve this minimax problem, which blends projected gradient ascent with sequential convex programming. By design, this algorithm is highly flexible and allows the user to choose among alternative statistical distance measures to define the ambiguity set. Moreover, the algorithm is highly tractable and scalable. Our numerical experiments suggest that a distributionally robust risk parity portfolio can yield a higher risk-adjusted rate of return when compared against the nominal portfolio.
\end{abstract}


%%%%%%%%%%%%%%%%%%%%%%%%%%%%%%%%%%%%%%%%%%%%%%%%%%%%%%%%%%%%%%%%%%%%%%%%%%%%%%%%%%%%%%%%%%%%%%
% 		INTRODUCTION
%%%%%%%%%%%%%%%%%%%%%%%%%%%%%%%%%%%%%%%%%%%%%%%%%%%%%%%%%%%%%%%%%%%%%%%%%%%%%%%%%%%%%%%%%%%%%%
\section{Introduction}\label{sec:intro}

Portfolio selection can be aptly presented as an optimal decision-making problem. Such problems have become prevalent in computational finance since the introduction of modern portfolio theory (MPT) by \citeauthor{markowitz1952portfolio} \cite{markowitz1952portfolio}. MPT posits that a portfolio's financial reward is quantified by its rate of return, while financial risk is quantified by the portfolio's variance. However, these two parameters are typically unknown to an investor and must be estimated from observable data, leading to estimation errors. In turn, these errors may have a profound impact on the portfolio's ex post financial performance. In the context of computational finance, the sensitivity of portfolio optimization to errors in estimated parameters has been widely explored in the literature \cite{best1991sensitivity, merton1980estimating, broadie1993computing}, leading to what is sometimes referred to as `error maximization' given the poor out-of-sample performance of these (ex ante) optimal portfolios.

Accounting for uncertainty during optimization has become paramount in any decision-making problem where parameters are non-deterministic. If we have knowledge of the underlying probability distribution that governs these parameters, then we can formulate this optimization problem as a stochastic program \cite{birge2011introduction, shapiro2014lectures}. On the other hand, when we have no distributional knowledge (or if we do not have confidence in our estimates) of the uncertain parameters, we can ignore any distributional estimates and instead solve the worst-case instance of the problem to ensure we retain feasibility in our solution. This is the basis of robust optimization, which frames the problem deterministically by taking the most extreme estimates of our uncertain parameters within some confidence level \cite{ben1998robust, bertsimas2004price, ben2009robust}. Some examples of robust optimization in the context of portfolio selection are presented in \cite{tutuncu2004robust, lobo2000worst, goldfarb2003robust, guastaroba2011investigating, fabozzi2007robust}.

This manuscript is based on a class of problems that sits somewhere in-between stochastic programming and robust optimization. Such problems attempt to use distributional information during optimization, but accept that the underlying probability distribution is unknown. Instead, the distribution is said to lie within an ambiguity set of probability distributions. Similar to robust optimization, a worst-case approach is taken, but with the distinction that we do this at the distributional level. Such a robust formulation for stochastic programs was proposed by \citeauthor{scarf1958min} \cite{scarf1958min}. Since then, this class of problems has often been referred to as \textit{minimax} problems or, more recently, as distributionally robust optimization (DRO) problems \cite{delage2010distributionally}. A detailed survey paper on DRO is presented in \cite{rahimian2019distributionally}.

The minimax problem has its roots in game theory \cite{neumann1928theorie}. In the context of this manuscript, we seek to minimize our cost function with respect to our decision variable, while the secondary player, i.e., `nature', is adversarial and seeks to maximize our cost with respect to our uncertain parameters. Thus, our true goal is to minimize our cost within the decision space against the most adversarial instance of the underlying distribution of the uncertain parameters. Minimax problems have been widely studied in literature in both theory and applications \cite{vzavckova1966minimax, dupavcova1987minimax, breton1995algorithms, shapiro2002minimax, shapiro2004class}. We note that minimax problems are sometimes referred to as saddle-point problems \cite{kim2008minimax, rustem2009algorithms} due to the `saddle' shape of the cost function when viewed in the higher-dimensional space created by the decision variable and the uncertain parameters. In particular, our manuscript focuses on the well-behaved subset of convex--concave minimax problems. 

The main objective of this manuscript is to introduce a distributionally robust portfolio selection problem. Specifically, we address a portfolio selection strategy known as \textit{risk parity}, which has gained popularity over the last decade among academics and practitioners. Risk parity seeks to construct a portfolio where the risk contribution of its constituent assets is equalized. In other words, each asset in the portfolio contributes the same level of risk towards the portfolio. Thus, by design, the risk parity problem is solely concerned with the portfolio risk measure, and does not necessitate the estimation of a reward measure. \citeauthor{maillard2008properties} \cite{maillard2008properties} carefully explain how to partition a portfolio's variance to find the risk contribution per asset. Directly optimizing the problem with respect to the asset risk contributions leads to a non-convex optimization problem \cite{bai2016least, costa2020generalized}, but some convex reformulations exist. For example, \citeauthor{mausser2014computing} \cite{mausser2014computing} cast the risk parity problem as a second order cone program (SOCP), while \citeauthor{bai2016least} \cite{bai2016least} casts it as  an unconstrained convex optimization problem. To address uncertainty in the estimated risk measure, \citeauthor{costa2020robust} \cite{costa2020robust} propose a robust risk parity framework built on the SOCP formulation, which takes the worst-case estimate of the risk measure but ignores any distributional information. For the purpose of this manuscript, we will use the convex risk parity problem from \cite{bai2016least}.

{\color{rev}It is possible to introduce distributional robustness into a portfolio selection problem by targeting the scenarios from which we derive the estimated parameters. When parameters are estimated from data, it is typically assumed that each scenario in the dataset is equally likely (i.e., we implicitly assume a uniform discrete probability distribution to describe the probability of each scenario). However, as shown in \cite{calafiore2007ambiguous, ben2013robust}, this assumption can be broken to allow the individual scenarios to have differing probabilities.} In turn, this discrete probability distribution can be modelled as a set of decision variables, allowing us to design a maximization problem to find the most adversarial discrete probability distribution such that we attain the worst-case instance of the estimated parameters. Given that only the portfolio risk measure is pertinent for risk parity, this manuscript focuses solely on the derivation of the asset covariance matrix from data. 

Addressing distributional robustness through a discrete probability distribution aligns naturally with a data-driven parameter estimation process. This assumes that market efficiency holds and that raw market data suffices to accurately represent the set of possible future returns. More importantly, this avoids making any assumptions about the underlying probability distribution of the asset returns, as well as avoiding assigning a structured process (such as a factor model) to model the returns. Thus, we are not required to impose a structure on the raw market data, which fully exempts us from the risk of model misspecification. This follows a similar rationale to another popular scenario-based portfolio risk measure known as historical value-at-risk, which assumes that raw market data suffices to represent the set of possible future outcomes. For the purpose of robustness in this manuscript, the application of a discrete probability distribution avoids the biases that could arise from assuming the returns have a specific structure, and provides us with the flexibility to derive a robust estimate of the asset covariance matrix implied by the raw market data themselves. 

The distributional ambiguity can be modelled as a convex set. This convex set is defined by constraints corresponding to the axioms of probability and, in particular, by a constraint that bounds the statistical distance between a nominal (i.e., assumed) probability distribution and its adversarial counterpart. The nominal distribution can be defined as any reasonable discrete probability distribution, but this amounts to a uniform distribution when we assume that all scenarios are equally likely. Thus, the adversarial distribution is allowed to deviate from the `equally likely' nominal distribution by a maximum permissible limit defined by our choice of statistical distance measure and our {\color{rev}desired degree of robustness}.

Given the conditions of our problem, we are limited to statistical distance measures for discrete probability distributions. The statistical distance measure used in \cite{calafiore2007ambiguous} was the Kullback--Leibler (KL) divergence. However, the KL divergence is not a proper distance metric,\footnote{A metric or `distance function' must be non-negative and satisfy the following axioms: symmetry, identity of indiscernibles, and the triangle inequality.} making it difficult for an investor to appropriately define this distance based on a {\color{rev}desired degree of robustness}. Thus, our manuscript focuses on statistical distance measures that satisfy the following two conditions: the measure must be a proper distance metric with finite bounds, and we must be able to formulate it as a computationally-tractable convex function. Therefore, the distributional ambiguity set is predominantly defined by our choice of distance measure and {\color{rev}degree of robustness}, which in turn defines our distributional robustness. Specifically, this manuscript discusses the following three statistical distance measures: the Jensen--Shannon (JS) divergence, Hellinger distance, and total variation (TV) distance. However, we note that our framework extends naturally to any finite statistical distance measure that can be represented as a convex function. 

% Algorithm 1: projected gradient descent-ascent algorithm tailored to this problem
{\color{rev}The nominal risk parity problem is modelled} as a convex minimization problem. In turn, the corresponding distributionally robust risk parity (DRRP) problem is a convex--concave minimax problem, where we seek to maximize our objective by finding the most adversarial instance of a discrete probability distribution. Our asset allocation variable is pragmatically constrained by the set of admissible portfolios, while the adversarial distribution is fundamentally constrained both by the axioms of probability and by the measure of statistical distance from the nominal distribution. Our modelling framework gives the user the flexibility to choose their preferred measure of statistical distance, provided this can be modelled as a convex function. {\color{rev}Thus, we have a} constrained minimax problem. 

{\color{rev}Such a problem can be solved by exploiting the dual of the adversarial maximization problem. In particular, \citeauthor{ben2013robust} \cite{ben2013robust} propose a tractable framework to reformulate minimax problems where statistical distance measures are part of the constraints. However, this framework is not able to handle variance-based risk measures due to the non-linearity in the adversarial probability variable. Nevertheless, as shown by \citeauthor{gotoh2018robust} \cite{gotoh2018robust}, a further transformation of the portfolio variance enables us to restate the DRRP problem as a straightforward convex minimization problem rather than a minimax problem. However, this reformulation brings about a secondary problem: numerical performance. Restating the DRRP problem as a convex minimization problem increases the complexity of the original problem, meaning that a non-linear optimization software package may be unable to solve the problem within reasonable time, in particular for large-scale problems.}

% Algorithm 2: Gradient ascent with sequential convex programming
{\color{rev}Accordingly, our objective is not only to present the DRRP problem, but also to introduce a numerically efficient method to solve it.} A standard approach to solve constrained minimax problems {\color{rev}is to use} projection-type methods \cite{bertsekas1976goldstein, nedic2009subgradient, xiu2003some}. {\color{rev}Thus, our initial attempt to improve numerical performance results in a} projected gradient descent--ascent (PGDA) algorithm that alternates between the descent and ascent steps to reach the saddle point {\color{rev}(i.e., the optimal solution). However,} a standard PGDA algorithm requires that we take alternating descent and ascent steps as we move towards the saddle point of our problem. Such an approach typically requires double the number of design parameters when compared to algorithms that move in a single direction. Moreover, these design parameters must be defined by the user a priori (e.g., initial point, step sizes). Finally, iterating in two directions increases the possibility of numerical divergence. {\color{rev}Although the PGDA algorithm suffers from the aforementioned drawbacks, it serves to motivate a novel gradient-based algorithm designed to efficiently solve the DRRP minimax problem}.

{\color{rev}Our proposed algorithm} is grounded in projected gradient {\color{rev}ascent} and sequential convex programming, {\color{rev}and it works by exploiting} the existence of a unique optimal risk parity portfolio for any given {\color{rev}instance of the adversarial probability} distribution. Thus, our proposed algorithm operates iteratively through gradient ascent in the probability space while solving a risk parity minimization problem in the asset weight space after every iteration. We can interpret our proposed algorithm as an implementation of sequential convex programming (SCP), where we ascend in the probability space towards the most adversarial instance of the portfolio risk measure after every iteration using a projected gradient ascent (PGA) method. Thus, we refer to our proposed algorithm as SCP--PGA. Compared to the PGDA algorithm, each iteration of the SCP--PGA algorithm is computationally more expensive. However, the exactness of each step translates to significantly fewer iterations until convergence. Additionally, we will see that the structure of the problem, combined with modern optimization software packages, allows for a computationally tractable and scalable implementation. Exploiting the convex minimization step at each iteration simplifies the algorithmic development since we are only required to iteratively ascend in the probability space while maintaining the risk parity condition in the asset weight space.  

Numerical experiments show that our SCP--PGA algorithm is computationally efficient and scales well for problems with a large number of assets and scenarios. Moreover, the in-sample experiments show that the DRRP problem behaves as expected, while the out-of-sample experiments demonstrate good ex post performance. Specifically, the DRRP portfolio is able to attain a higher risk-adjusted rate of return when compared to the nominal risk parity portfolio. 

%-------------------------------------------------------------------------------------------
% Contribution
%-------------------------------------------------------------------------------------------
\subsection{Contribution}\label{sec:contribution}

This manuscript presents a DRRP portfolio optimization problem with a discrete probability ambiguity set on the portfolio risk measure. {\color{rev}Introducing distributional robustness through a discrete probability distribution} allows us to design a minimax risk parity problem. Specifically, this minimax problem seeks to equalize the asset risk contributions against the worst-case instance of the portfolio variance. 

{\color{rev}The contributions of this manuscript} are the following. First, we introduce the DRRP problem, which seeks risk parity with respect to the most adversarial estimate of the portfolio risk measure through a purely data-driven process (i.e, the probabilistic ambiguity is implied by the data themselves). Second, we explicitly define how to construct this ambiguity set using different statistical distance metrics and we show how to use an investor's {\color{rev}desired degree of robustness} to size the ambiguity set. {\color{rev}Third, we recast the original minimax problem as a convex minimization problem by exploiting the dual of the maximization step as shown in \cite{ben2013robust}. Although results in a convex minimization problem, it is highly non-linear and becomes numerically inefficient to solve as the problem increases in size. Therefore, our final contribution addresses this shortcoming by significantly enhancing numerical performance. To do so, we introduce} the SCP--PGA algorithm to iteratively solve the DRRP minimax problem. We note that the flexible structure of the SCP--PGA algorithm means that it may be applied to solve other portfolio selection problems, as well as other types of constrained convex--concave minimax problems from other disciplines. 

%-------------------------------------------------------------------------------------------
% Outline
%-------------------------------------------------------------------------------------------
\subsection{Outline}\label{sec:outline}

The outline of this paper is the following. Section \ref{sec:prelim} introduces the preliminaries that serve as a foundation for the development of this paper. Our main contribution is presented in Section \ref{sec:DRORP}, where we propose the DRRP problem and {\color{rev}develop the SCP--PGA algorithm} to find the optimal risk parity portfolio. The corresponding numerical experiments are shown in Section \ref{sec:Exp}, which evaluate the proposed problem's computational tractability, as well as its in-sample and out-of-sample financial performance. Finally, Section \ref{sec:conclusion} summarizes the findings and contribution of this paper.

%-------------------------------------------------------------------------------------------
% Notation
%-------------------------------------------------------------------------------------------
\subsection{Notation}\label{sec:notation}

We denote a real space of dimension \(n\) by \(\mathbb{R}^n\) and the corresponding non-negative orthant by \(\mathbb{R}_+^n\). Moreover, symmetric matrices of dimension \(n\) with real-valued elements are denoted by \(\mathbb{S}^n\), while the subset of positive semi-definite (PSD) matrices are denoted by \(\mathbb{S}_+^n\). If we need to reference some specific element \(i\) within a vector \(\bm{z}\), we denote this as \(z_i\). If we define a vector as the product between a matrix and a vector, \(\bm{A}\bm{z}\in\mathbb{R}^m\), then we reference its \(i^{\text{th}}\) element as \([\bm{A}\bm{z}]_i\). The \(\ell_p\)-norm of an arbitrary vector \(\bm{z}\in\mathbb{R}^n\) is denoted by \(\|\cdot\|_p\), where \(\|\bm{z}\|_p \triangleq (\sum_{i=1}^n |z_i|^p)^{1/p}\). {\color{rev}Finally, theorems that are well-known and established in literature are identified by the theorem's name in brackets and are presented without proof.}

%%%%%%%%%%%%%%%%%%%%%%%%%%%%%%%%%%%%%%%%%%%%%%%%%%%%%%%%%%%%%%%%%%%%%%%%%%%%%%%%%%%%%%%%%%%%%%
% 		Preliminaries
%%%%%%%%%%%%%%%%%%%%%%%%%%%%%%%%%%%%%%%%%%%%%%%%%%%%%%%%%%%%%%%%%%%%%%%%%%%%%%%%%%%%%%%%%%%%%%
\section{Preliminaries}\label{sec:prelim}

%-------------------------------------------------------------------------------------------
% Estimation of parameters
%-------------------------------------------------------------------------------------------
\subsection{Estimation of parameters}\label{sec:param}

We begin by discussing the measures of financial reward and risk that will be used in this manuscript. As defined in MPT \cite{markowitz1952portfolio}, the reward is measured by the portfolio rate of return (or simply the `return'). The portfolio return is a weighted linear combination of the returns of the \(n\) assets that constitute the portfolio. The asset returns are random variables which we define as the vector \(\bm{\xi}\in\mathbb{R}^n\). These random returns are governed by some probability distribution with first and second moments defined as the expected returns \(\bm{\mu}\in\mathbb{R}^n\) and covariance matrix \(\bm{\Sigma}\in\mathbb{S}_+^n\), respectively. It follows that, at the asset-level, the financial reward is measured by \(\bm{\mu}\) and the financial risk by \(\bm{\Sigma}\). In particular, the true moments are assumed to be latent and are typically estimated from data, meaning they are prone to suffer from estimation error \cite{chopra1993, best1991sensitivity, merton1980estimating}. 

We define a portfolio as a vector of asset weights \(\bm{x}\in\mathbb{R}^n\), where \(x_i\) represents the proportion of wealth invested in asset \(i\). From an asset management perspective, \(\bm{x}\) is our vector of decision variables that represents our asset allocation strategy. Thus, at the portfolio-level, the portfolio random return is defined as \(\pi(\bm{x}) \triangleq \bm{\xi}^\top \bm{x}\). The corresponding measures of financial reward and risk are
%-^-^-^-^-^-^-^-^-^-^-^-^-^-^-^-^-^-^-^-^-^-^-^-^-^-^-^-^-^-^-
\begin{align}
	\mu_\pi(\bm{x}) &\triangleq \bm{\mu}^\top \bm{x}, \label{eq:PortRet}\\
    \sigma_\pi^2(\bm{x}) &\triangleq \bm{x}^\top \bm{\Sigma} \bm{x}, \label{eq:PortVar}
\end{align}
%-^-^-^-^-^-^-^-^-^-^-^-^-^-^-^-^-^-^-^-^-^-^-^-^-^-^-^-^-^-^-
where the portfolio expected return is \(\mu_\pi\in\mathbb{R}\) while the portfolio variance is \(\sigma_\pi^2\in\mathbb{R}_+\). The portfolio \(\bm{x}\) is generally constrained by the set of admissible portfolios \(\mathcal{X}\), which, in our case, disallows short sales and imposes a unit budget constraint. Short sales are prohibited due to a fundamental limitation of risk parity, which we discuss in greater detail in Section \ref{sec:RP}. It follows that the set of admissible portfolios is the following simplex
%-^-^-^-^-^-^-^-^-^-^-^-^-^-^-^-^-^-^-^-^-^-^-^-^-^-^-^-^-^-^-
\begin{equation}
	\mathcal{X} \triangleq \big\{ \bm{x}\in\mathbb{R}^n_+ : \bm{1}^\top \bm{x} = 1\big\}.
\label{eq:set}
\end{equation}
%-^-^-^-^-^-^-^-^-^-^-^-^-^-^-^-^-^-^-^-^-^-^-^-^-^-^-^-^-^-^-
Restricting \(\bm{x}\) to the non-negative orthant of the real \(n\)-dimensional space serves to disallow short sales. The equality constraint in \(\mathcal{X}\) is necessary to ensure that the entirety of our available budget is invested in the assets.

The first two moments of the joint probability distribution of asset returns, \(\bm{\mu}\) and \(\bm{\Sigma}\), are typically estimated from data (e.g., historical scenarios of asset returns). Assume our asset return dataset \(\hat{\bm{\xi}} \in\mathbb{R}^{n\times T}\)   consists of \(T\) discrete scenarios for \(n\) assets (i.e., we have \(\hat{\bm{\xi}}\) scenarios and these scenarios suffice to appropriately represent the possible outcomes of the random variable \(\bm{\xi}\)). In a similar fashion to \cite{calafiore2007ambiguous}, we assume there exists some probability \(p_t\) associated with each scenario \(t\). In vector notation, this is the probability mass function \(\bm{p}\in\mathcal{P}\), where
%-^-^-^-^-^-^-^-^-^-^-^-^-^-^-^-^-^-^-^-^-^-^-^-^-^-^-^-^-^-^-
\begin{equation}
	\mathcal{P} \triangleq \big\{\bm{p} \in\mathbb{R}_+^T : \bm{1}^\top \bm{p} = 1\big\}
\label{eq:axioms}
\end{equation}
%-^-^-^-^-^-^-^-^-^-^-^-^-^-^-^-^-^-^-^-^-^-^-^-^-^-^-^-^-^-^-
is the simplex defined by the axioms of probability. 

If we have knowledge of \(\bm{p}\), then we can statistically estimate \(\bm{\mu}\) and \(\bm{\Sigma}\). Let \(\hat{\bm{\xi}}_t\in\mathbb{R}^n\) be the \(t^{\text{th}}\) scenario of the dataset \(\hat{\bm{\xi}}\). The first two moments are
%-^-^-^-^-^-^-^-^-^-^-^-^-^-^-^-^-^-^-^-^-^-^-^-^-^-^-^-^-^-^-
\begin{align}
    \hat{\bm{\mu}}(\bm{p}) &\triangleq \mathbb{E}[\bm{\xi}] = \sum_{t=1}^T p_t \cdot \hat{\bm{\xi}}_t,\label{eq:mu}\\
    \hat{\bm{\Sigma}}(\bm{p}) &\triangleq \mathbb{E}\big[\big(\bm{\xi} - \hat{\bm{\mu}}(\bm{p})\big)^2\big] = \sum_{t=1}^T p_t\cdot \big(\hat{\bm{\xi}}_t - \hat{\bm{\mu}}(\bm{p})\big)\big(\hat{\bm{\xi}}_t - \hat{\bm{\mu}}(\bm{p})\big)^\top,\label{eq:Sigma}
\end{align}  
%-^-^-^-^-^-^-^-^-^-^-^-^-^-^-^-^-^-^-^-^-^-^-^-^-^-^-^-^-^-^-
where \(\hat{\bm{\mu}} \in\mathbb{R}^n\) and \(\hat{\bm{\Sigma}}\in\mathbb{S}_+^n\) are the data-driven estimates of the latent parameters \(\bm{\mu}\) and \(\bm{\Sigma}\), respectively. Our estimates are shown as functions of some discrete probability distribution \(\bm{p}\). If we assume each scenario is equally likely, then \eqref{eq:mu} and \eqref{eq:Sigma} are simply the standard sample arithmetic mean and sample covariance matrix\footnote{We note that \(\hat{\bm{\Sigma}}(\bm{q})\), where \(q_t=1/T\) for \(t=1,\dots,T\), yields the standard scenario-based estimate of the covariance matrix. To recover the \textit{unbiased} estimate of the covariance matrix, we should multiply \(\hat{\bm{\Sigma}}(\bm{q})\) by \(T/(T-1)\). However, this distinction has no effect for the purpose of this paper.} typically derived from data. Finally, we note that \(\hat{\bm{\Sigma}}(\bm{p})\) in \eqref{eq:Sigma} is the result of the weighted sum of \(T\) rank-1 symmetric matrices, which means \(\hat{\bm{\Sigma}}(\bm{p})\) is guaranteed to be a PSD matrix. 

Estimating \(\bm{\mu}\) and \(\bm{\Sigma}\) in this fashion means we are not required to impose any particular structure to model the latent asset returns distribution, avoiding the any biases and errors arising from model misspecification. Our only assumption is that market efficiency holds, meaning we can derive adequate estimates of \(\bm{\mu}\) and \(\bm{\Sigma}\) directly from a given raw market dataset \(\hat{\bm{\xi}}\).

The estimated portfolio expected return and variance follow the same logic as \eqref{eq:PortRet} and \eqref{eq:PortVar}, except we replace the latent parameters \(\bm{\mu}\) and \(\bm{\Sigma}\) with the estimates \(\hat{\bm{\mu}}(\bm{p})\) and \(\hat{\bm{\Sigma}}(\bm{p})\) from \eqref{eq:mu} and \eqref{eq:Sigma}. We break down the derivation as follows. Assume we have a portfolio \(\bm{x}\). For a given dataset \(\hat{\bm{\xi}}\), the corresponding vector of portfolio return scenarios is \(\hat{\bm{\pi}}(\bm{x}) \triangleq \hat{\bm{\xi}}^\top \bm{x}\). Thus, the estimated portfolio expected return \(\hat{\mu}_\pi\in\mathbb{R}\) and variance \(\hat{\sigma}_\pi^2\in\mathbb{R}_+\) are
%-^-^-^-^-^-^-^-^-^-^-^-^-^-^-^-^-^-^-^-^-^-^-^-^-^-^-^-^-^-^-
\begin{align}
	\startsubequation\subequationlabel{eq:PortRetEst}\tagsubequation\label{eq:PortRetEst1}
    \hat{\mu}_\pi(\bm{x},\bm{p}) &\triangleq \bm{x}^\top \hat{\bm{\mu}}(\bm{p})\\
    							&= \bm{p}^\top \hat{\bm{\pi}}(\bm{x}),\tagsubequation\label{eq:PortRetEst2}\\[1.5ex]
    \startsubequation\subequationlabel{eq:PortVarEst}\tagsubequation\label{eq:PortVarEst1}
    \hat{\sigma}_\pi^2(\bm{x},\bm{p}) &\triangleq \bm{x}^\top \hat{\bm{\Sigma}}(\bm{p}) \bm{x}\\
    								&= \mathbb{E}\Big[ \big(\pi(\bm{x}) - \mathbb{E}[\pi(\bm{x})]\big)^2 \Big] = \mathbb{E}[\pi^2(\bm{x})] - \big(\mathbb{E}[\pi(\bm{x})]\big)^2\nonumber\\ 
    \tagsubequation\label{eq:PortVarEst2}
    &= \bm{p}^\top \hat{\bm{\pi}}^2(\bm{x}) - \bm{p}^\top \hat{\bm{\Theta}}(\bm{x}) \bm{p},
\end{align}
%-^-^-^-^-^-^-^-^-^-^-^-^-^-^-^-^-^-^-^-^-^-^-^-^-^-^-^-^-^-^-
where \(\hat{\bm{\pi}}^2(\bm{x})\in\mathbb{R}^T_+\) denotes the element-wise square of the vector of portfolio return scenarios, and \(\hat{\bm{\Theta}}(\bm{x}) \triangleq \hat{\bm{\pi}}(\bm{x}) \hat{\bm{\pi}}(\bm{x})^\top\). By definition, we have that \(\hat{\bm{\Theta}}(\bm{y})\in\mathbb{S}_+^T\) for any vector \(\bm{y}\in\mathbb{R}^n\), meaning the portfolio variance in \eqref{eq:PortVarEst2} is concave over \(\bm{p}\in\mathcal{P}\). Moreover, since \(\hat{\bm{\Sigma}}(\bm{p})\in\mathbb{S}_+^n\) for any probability distribution \(\bm{p}\in\mathcal{P}\), the portfolio variance in \eqref{eq:PortVarEst1} is convex over \(\bm{x}\in\mathcal{X}\). As we will see in Section \ref{sec:DRORP}, the convexity over \(\bm{x}\in\mathcal{X}\) and concavity over \(\bm{p}\in\mathcal{P}\) of the portfolio variance \(\hat{\sigma}_\pi^2(\bm{x},\bm{p})\) will allow us to formulate a convex--concave minimax problem. 

%-------------------------------------------------------------------------------------------
% Risk parity
%-------------------------------------------------------------------------------------------
\subsection{Risk parity}\label{sec:RP}

Risk parity is a modern asset allocation strategy that aims to construct a portfolio where every asset contributes the same amount of risk. Thus, risk parity is fully diversified from a risk perspective. In turn, the risk parity problem is only concerned with financial risk, and does not require a measure of financial reward during optimization.

{\color{rev} The portfolio standard deviation, often referred to as the portfolio `volatility,' is the square root of the portfolio variance. Assume we have perfect knowledge of the distribution of the asset random returns (i.e., assume we have knowledge of the true covariance matrix \(\bm{\Sigma}\)). Then, the portfolio standard deviation can be found by taking the square root of \eqref{eq:PortVar}, i.e., \(\sigma_\pi = \sqrt{\bm{x}^\top \bm{\Sigma} \bm{x}}\). By definition, the portfolio standard deviation is a homogeneous function of degree one, which allows us to partition the standard deviation into its asset-level components by applying Theorem \ref{thm:Euler}.

\begin{theorem}[Euler's homogeneous function theorem]
\label{thm:Euler}
A positive homogeneous function of degree \(k\) is a function where \(f(c\cdot \bm{z}) = c^k\cdot f(\bm{z})\ \forall\ c\in\mathbb{R}_+\).  Let \(f:\mathbb{R}^n \rightarrow \mathbb{R}\) be a continuous and differentiable homogeneous function of degree one. Then,
%-^-^-^-^-^-^-^-^-^-^-^-^-^-^-^-^-^-^-^-^-^-^-^-^-^-^-^-^-^-^-
\[
	k\cdot f(\bm{z}) = \bm{z}^\top \nabla f(\bm{z}).
\]
%-^-^-^-^-^-^-^-^-^-^-^-^-^-^-^-^-^-^-^-^-^-^-^-^-^-^-^-^-^-^-
\end{theorem}

The portfolio standard deviation, \(\sigma_\pi = \sqrt{\bm{x}^\top \bm{\Sigma} \bm{x}}\), is a positive homogeneous function of degree one. Then, as shown in \cite{maillard2008properties},  we can apply Theorem \ref{thm:Euler} to partition it into a sum of asset-level components as follows,}
%-^-^-^-^-^-^-^-^-^-^-^-^-^-^-^-^-^-^-^-^-^-^-^-^-^-^-^-^-^-^-
\begin{equation}
    \sigma_\pi = \sqrt{\bm{x}^\top \bm{\Sigma} \bm{x}} = \sum_{i=1}^{n} x_i \frac{\partial \sigma_\pi}{\partial x_i} = \sum_{i=1}^{n} x_i \frac{[\bm{\Sigma} \bm{x}]_i}{\sqrt{\bm{x}^\top \bm{\Sigma} \bm{x}}}.
\label{eq:riskDecomp}
\end{equation}
%-^-^-^-^-^-^-^-^-^-^-^-^-^-^-^-^-^-^-^-^-^-^-^-^-^-^-^-^-^-^-
The latter part of \eqref{eq:riskDecomp} shows the partitions of the portfolio standard deviation for each asset \(i = 1, \dots, n\). Note that the denominator in this expression is consistent for all partitions, and it is equal to the portfolio standard deviation. As shown in \cite{costa2020generalized}, we can rearrange this expression such that we partition the portfolio variance instead. Thus, we can express the portfolio variance as the sum of \(n\) parts,\footnote{{\color{rev}Since the portfolio variance is a positively homogeneous function of degree two, we can also reach the same conclusion by decomposing the portfolio variance directly using Theorem \ref{thm:Euler}.}} 
%-^-^-^-^-^-^-^-^-^-^-^-^-^-^-^-^-^-^-^-^-^-^-^-^-^-^-^-^-^-^-
\begin{equation}
    \sigma_\pi^2 = \bm{x}^\top \bm{\Sigma} \bm{x} = \sum_{i=1}^{n} x_i [\bm{\Sigma} \bm{x}]_i = \sum_{i=1}^{n} R_i,
\label{eq:varDecomp}
\end{equation}
%-^-^-^-^-^-^-^-^-^-^-^-^-^-^-^-^-^-^-^-^-^-^-^-^-^-^-^-^-^-^-
where \(R_i \triangleq x_i [\bm{\Sigma} \bm{x}]_i\) is the individual risk contribution of asset \(i\). Now that we are able to measure the individual risk contributions, we can formulate an optimization problem to construct {\color{rev}a portfolio that satisfies the risk parity condition,} \(R_i = R_j\ \forall\ i,j\).

As prescribed by \citeauthor{bai2016least} \cite{bai2016least}, we can design an unconstrained convex optimization problem that, at optimality, attains the desired risk parity condition. The problem is the following
%-^-^-^-^-^-^-^-^-^-^-^-^-^-^-^-^-^-^-^-^-^-^-^-^-^-^-^-^-^-^-
\begin{equation}
	\min_{\bm{y}\in\mathbb{R}_+^n}\quad \frac{1}{2}\bm{y}^\top\bm{\Sigma} \bm{y} - \kappa \sum_{i=1}^n \ln(y_i),
\label{eq:NomRP}
\end{equation}
%-^-^-^-^-^-^-^-^-^-^-^-^-^-^-^-^-^-^-^-^-^-^-^-^-^-^-^-^-^-^-
where \(\kappa > 0\) is some arbitrary constant\footnote{In theory, we can assign any positive value to \(\kappa\) without loss of generality. In practice, we should avoid assigning extremely large or small values to \(\kappa\) to avoid numerical instability.} and the auxiliary variable \(\bm{y}\in\mathbb{R}^n_+\) serves as a placeholder for our asset weights. The auxiliary variable \(\bm{y}\) will most likely violate the set of admissible portfolios \(\mathcal{X}\) given that we do not impose a budget equality constraint. 

{\color{rev}The first term in the objective function of \eqref{eq:NomRP}, \(\bm{y}^\top\bm{\Sigma} \bm{y}\), is akin to the portfolio variance with respect to the placeholder variable \(\bm{y}\), which may imply that \eqref{eq:NomRP} seeks to minimize the portfolio variance. However, this is not the case due to the logarithmic barrier term. In fact, this objective function is designed specifically to attain the \textit{risk parity condition} at optimality. A solution \(\bm{y}^{\text{RP}}\in\mathbb{R}_+^n\) satisfies the the risk parity condition when all the variance-based asset risk contributions are equalized, i.e., \(R_i = R_j\ \forall\ i,j\). 

\begin{lemma}
\label{lemma:RP}
Let \(\bm{\Sigma}\in\mathbb{S}_+^n\) and \(\kappa > 0\). The optimal solution of the optimization problem in \eqref{eq:NomRP} is unique and satisfies the risk parity condition, i.e., \(\bm{y}^{\text{RP}} = \argmin_{\bm{y}\in\mathbb{R}_+^n} \frac{1}{2}\bm{y}^\top\bm{\Sigma} \bm{y} - \kappa \sum_{i=1}^n \ln(y_i)\).
\end{lemma} 

\begin{proof}
Let
%-^-^-^-^-^-^-^-^-^-^-^-^-^-^-^-^-^-^-^-^-^-^-^-^-^-^-^-^-^-^-
\(
	f(\bm{y}) = \frac{1}{2}\bm{y}^\top\bm{\Sigma} \bm{y} - \kappa \sum_{i=1}^n \ln(y_i).
\)
%-^-^-^-^-^-^-^-^-^-^-^-^-^-^-^-^-^-^-^-^-^-^-^-^-^-^-^-^-^-^-
Since \(\bm{\Sigma}\in\mathbb{S}_+^n\) and given that \(\kappa \sum_{i=1}^{n} \ln{(y_i)}\)  for \(\kappa>0\) is a strictly concave function, then \(f\) is a strictly convex function. By design, the logarithmic barrier term naturally restricts the inputs to the non-negative orthant of the real \(n\)-dimensional space, i.e., \(\bm{y}\in\mathbb{R}_+^n\). 

Since \(f\) is strictly convex function, its unique minimizer \(\bm{y}^{\text{RP}}\) is attained by satisfying the first-order condition, \(\nabla f(\bm{y}^{\text{RP}}) = \bm{0}\). The gradient of \(f\) is 
%-^-^-^-^-^-^-^-^-^-^-^-^-^-^-^-^-^-^-^-^-^-^-^-^-^-^-^-^-^-^-
\[
    \nabla f(\bm{y}) = \bm{\Sigma} \bm{y} - \kappa \bm{y}^{-1},
\]
%-^-^-^-^-^-^-^-^-^-^-^-^-^-^-^-^-^-^-^-^-^-^-^-^-^-^-^-^-^-^-
	where \(\bm{y}^{-1} = [1/y_1, 1/y_2, ..., 1/y_n]^\top\). Since \(\nabla f(\bm{y}^{\text{RP}}) = \bm{0}\), we must have that 
%-^-^-^-^-^-^-^-^-^-^-^-^-^-^-^-^-^-^-^-^-^-^-^-^-^-^-^-^-^-^-
\[
\begin{aligned}
	\ [\bm{\Sigma} \bm{y}^{\text{RP}}]_i &= \frac{\kappa}{y_i^{\text{RP}}} && \forall\ i, \\[1ex]
    y_i^{\text{RP}} [\bm{\Sigma} \bm{y}^{\text{RP}}]_i &= \kappa && \forall\ i, \\[1ex]
    y_i^{\text{RP}} [\bm{\Sigma} \bm{y}^{\text{RP}}]_i &= y_j^{\text{RP}} [\bm{\Sigma} \bm{y}^{\text{RP}}]_j && \forall\ i,\ j,
\end{aligned}
\]
%-^-^-^-^-^-^-^-^-^-^-^-^-^-^-^-^-^-^-^-^-^-^-^-^-^-^-^-^-^-^-
which shows \(\bm{y}^{\text{RP}}\) satisfies the risk parity condition for any \(\kappa>0\).
\end{proof}
}

Since the risk parity problem in \eqref{eq:NomRP} does not impose {\color{rev}a budget equality} constraint, we cannot claim its optimal solution \(\bm{y}^{\text{RP}}\) is an admissible portfolio. {\color{rev}Nevertheless, we can use \(\bm{y}^{\text{RP}}\) to recover the unique optimal risk parity portfolio \(\bm{x}^{\text{RP}}\). For any \(\bm{y}\in\mathbb{R}_+^n\), the projection onto the set of admissible portfolios \(\mathcal{X}\) is
%-^-^-^-^-^-^-^-^-^-^-^-^-^-^-^-^-^-^-^-^-^-^-^-^-^-^-^-^-^-^-
\[
	\text{\large $\Pi$}_{\mathcal{X}}(\bm{y}) \triangleq \frac{\bm{y}}{\sum_{i=1}^n y_i}.
\]
%-^-^-^-^-^-^-^-^-^-^-^-^-^-^-^-^-^-^-^-^-^-^-^-^-^-^-^-^-^-^-

\begin{theorem}
\label{thm:RP}
	For a given \(\bm{\Sigma}\in\mathbb{S}_+^n\), there exists a unique optimal risk parity portfolio \(\bm{x}^{\text{RP}}\in\mathcal{X}\), independent of the choice of \(\kappa>0\), which can be recovered by projecting the solution to the risk parity problem in \eqref{eq:NomRP} onto \(\mathcal{X}\), i.e.,  \(\bm{x}^{\text{RP}} \triangleq \text{\large $\Pi$}_{\mathcal{X}}\big(\bm{y}^{\text{RP}}\big)\).
\end{theorem}

\begin{proof}
	This proof is similar to Lemma 2.2 in \cite{bai2016least}. Let  \(\bm{y}^a\) be the optimal solution to the optimization problem in \eqref{eq:NomRP} for some \(\kappa = \kappa_a > 0\). Applying our previous Lemma \ref{lemma:RP}, we must have that
	%-^-^-^-^-^-^-^-^-^-^-^-^-^-^-^-^-^-^-^-^-^-^-^-^-^-^-^-^-^-^-
	\[
		y_i^a [\bm{\Sigma} \bm{y}^a]_i = \kappa_a\ \forall\ i.
	\]
	%-^-^-^-^-^-^-^-^-^-^-^-^-^-^-^-^-^-^-^-^-^-^-^-^-^-^-^-^-^-^-
	Now, let \(\kappa_b = \kappa_a / \big(\sum_{i=1}^n y_i^a \big)^2\). If we solve \eqref{eq:NomRP} a second time with \(\kappa = \kappa_b\), the optimal solution is
	%-^-^-^-^-^-^-^-^-^-^-^-^-^-^-^-^-^-^-^-^-^-^-^-^-^-^-^-^-^-^-
	\[
		y_i^b [\bm{\Sigma} \bm{y}^b]_i = \kappa_b = \frac{\kappa_a}{\big(\sum_{i=1}^n y_i^a\big)^2} = \frac{y_i^a [\bm{\Sigma} \bm{y}^a]_i}{\big(\sum_{i=1}^n y_i^a\big)^2}\quad \forall\ i,
	\]
	%-^-^-^-^-^-^-^-^-^-^-^-^-^-^-^-^-^-^-^-^-^-^-^-^-^-^-^-^-^-^-
	which also means that \(\bm{y}^b\) satisfies the budget equality constraint, i.e., \(\bm{1}^\top \bm{y}^b = 1\). Now, since \(\bm{y}^b\) satisfies the risk parity condition and \(\bm{y}^b\in\mathcal{X}\), then, by definition, \(\bm{y}^b = \bm{x}^{\text{RP}}\). Note that this is equivalent to projecting the initial optimal solution \(\bm{y}^a\) onto the set of admissible portfolios \(\mathcal{X}\), i.e., \(\text{\large $\Pi$}_{\mathcal{X}}(\bm{y}^a) = \bm{y}^b = \bm{x}^{\text{RP}}\). In other words, we can scale any arbitrary value of \(\kappa>0\) such that we recover the unique optimal solution \(\bm{x}^{\text{RP}}\in\mathcal{X}\) for a given covariance matrix \(\bm{\Sigma}\in\mathbb{S}_+^n\).
\end{proof}

As shown in Theorem \ref{thm:RP}, there exists a unique risk parity solution for a given \(\bm{\Sigma}\in\mathbb{S}_+^n\). However, the converse is not always true. Two different covariance matrices may lead to the same risk parity portfolio if these two covariance matrices are linearly dependent. Therefore, one-to-one correspondence between a covariance matrix \(\bm{\Sigma}\) and the risk parity solution \(\bm{x}^{\text{RP}}\) is only guaranteed  when \(\bm{\Sigma}\) is linearly independent from other matrices. This leads to the following corollary.

\begin{corollary}
\label{coll:RP}
	Assume we have two covariance matrices, \(\bm{\Sigma}^a, \bm{\Sigma}^b \in \mathbb{S}_+^n\). Let \(\bm{x}^a, \bm{x}^b\in\mathcal{X}\) be the risk parity solutions corresponding to \(\bm{\Sigma}^a\) and \(\bm{\Sigma}^b\), respectively. Since both \(\bm{x}^a\) and \(\bm{x}^b\) are optimal risk parity solutions, then \(\bm{x}^a = \bm{x}^b\) if and only if \(\bm{\Sigma}^a = c\cdot \bm{\Sigma}^b\) for any \(c>0\).	
\end{corollary}

\begin{proof}
	The proof arises naturally from Theorem \ref{thm:RP}, which established a unique correspondence between a given covariance matrix \(\bm{\Sigma}\in\mathbb{S}_+^n\) and its  risk parity portfolio \(\bm{x}^{\text{RP}}\in\mathcal{X}\). From Lemma \ref{lemma:RP}, the risk parity solution \(\bm{x}^a\) satisfies
	%-^-^-^-^-^-^-^-^-^-^-^-^-^-^-^-^-^-^-^-^-^-^-^-^-^-^-^-^-^-^-
	\[
		x_i^a [\bm{\Sigma}^a \bm{x}^a]_i = x_j^a [\bm{\Sigma}^a \bm{x}^a]_j\quad \forall\ i,j.
	\]
	%-^-^-^-^-^-^-^-^-^-^-^-^-^-^-^-^-^-^-^-^-^-^-^-^-^-^-^-^-^-^-
	Let \(\bm{\Sigma}^a = c\cdot \bm{\Sigma}^b\). Then 
	%-^-^-^-^-^-^-^-^-^-^-^-^-^-^-^-^-^-^-^-^-^-^-^-^-^-^-^-^-^-^-
	\[
		x_i^a [(c \cdot \bm{\Sigma}^b) \bm{x}^a]_i = x_j^a [(c \cdot \bm{\Sigma}^b) \bm{x}^a]_j \ \iff\
		x_i^a [\bm{\Sigma}^b \bm{x}^a]_i = x_j^a [\bm{\Sigma}^b \bm{x}^a]_j\quad \forall\ i,j,
	\]
	%-^-^-^-^-^-^-^-^-^-^-^-^-^-^-^-^-^-^-^-^-^-^-^-^-^-^-^-^-^-^-
	Recall that both \(\bm{x}^a, \bm{x}^b\in\mathcal{X}\). Since \(\bm{x}^a\) satisfies the risk parity condition \(R_i = R_j\ \forall i,j\) with respect to \(\bm{\Sigma}^b\), then, by definition, \(\bm{x}^b = \bm{x}^a\). It follows that if \(\bm{\Sigma}^a \neq c\cdot\bm{\Sigma}^b\) for any \(c>0\), then \(\bm{x}^a\) is not a risk parity solution with respect to \(\bm{\Sigma}^b\), and \(\bm{x}^b\) is not a risk parity solution with respect to \(\bm{\Sigma}^a\). 
\end{proof}
}

Traditionally, the risk parity asset allocation strategy restricts itself to `long-only' portfolios where short sales are disallowed. This aligns well with restrictions typically observed in the asset management industry. However, this restriction stems from a fundamental limitation of risk parity portfolio optimization. As shown in \eqref{eq:NomRP}, risk parity can be formulated as a strictly convex optimization problem with a unique global solution. Other equivalent convex formulations can be found in \cite{maillard2008properties} and \cite{mausser2014computing}. However, once short sales are allowed the problem becomes non-convex and the uniqueness of our solution is no longer guaranteed \cite{bai2016least, costa2020generalized}. For the sake of computational tractability, this paper restricts itself to the long-only condition imposed by traditional risk parity asset allocation strategies. 

Thus far, we have assumed we have knowledge of the true (but latent) covariance matrix \(\bm{\Sigma}\). In practice, we can use the estimated covariance matrix \(\hat{\bm{\Sigma}}(\bm{p})\) from \eqref{eq:Sigma}, which corresponds to some discrete probability distribution \(\bm{p}\). We can find the risk parity portfolio corresponding to an instance of \(\bm{p}\in\mathcal{P}\) through the following system of equations,
%-^-^-^-^-^-^-^-^-^-^-^-^-^-^-^-^-^-^-^-^-^-^-^-^-^-^-^-^-^-^-
\begin{subequations}
\label{eq:minRP}
\begin{align}
	f_{\text{RP}}(\bm{y},\bm{p}) &\triangleq \frac{1}{2}\bm{y}^\top \hat{\bm{\Sigma}}(\bm{p}) \bm{y} - \kappa \sum_{i=1}^n \ln{(y_i)},\label{eq:minRPobj}\\[1ex]
	\bm{y}^{\text{RP}}(\bm{p}) &\triangleq \argmin_{\bm{y}\in\mathbb{R}_+^n}\ f_{\text{RP}}(\bm{y},\bm{p}),\label{eq:minRPopt}\\[1ex]
	\bm{x}^{\text{RP}}(\bm{p}) &\triangleq \text{\large $\Pi$}_{\mathcal{X}}\big(\bm{y}^{\text{RP}}(\bm{p})\big)\label{eq:minRPproj},
\end{align}
\end{subequations}
%-^-^-^-^-^-^-^-^-^-^-^-^-^-^-^-^-^-^-^-^-^-^-^-^-^-^-^-^-^-^-
where \(f_{\text{RP}}: \mathbb{R}_+^n \times \mathcal{P} \rightarrow \mathbb{R}\) is the risk parity objective function. 

{\color{rev}Therefore,} we can use the optimization problem in \eqref{eq:minRP} to find an optimal risk parity portfolio for any estimate of the covariance matrix \(\hat{\bm{\Sigma}}(\bm{p})\) with respect to any instance of \(\bm{p}\in\mathcal{P}\).
 
%%%%%%%%%%%%%%%%%%%%%%%%%%%%%%%%%%%%%%%%%%%%%%%%%%%%%%%%%%%%%%%%%%%%%%%%%%%%%%%%%%%%%%%%%%%%%%
% 		DRO
%%%%%%%%%%%%%%%%%%%%%%%%%%%%%%%%%%%%%%%%%%%%%%%%%%%%%%%%%%%%%%%%%%%%%%%%%%%%%%%%%%%%%%%%%%%%%%
\section{Distributionally robust risk parity}\label{sec:DRORP}

This section presents our two main contributions: a data-driven DRRP portfolio optimization problem and the SCP--PGA algorithm to solve it. Our immediate goal is twofold: to design an appropriate ambiguity set \(\mathcal{U}_{\bm{p}}\) for our adversarial probability distribution \(\bm{p}\), and to formulate the DRRP minimax problem. We address these two issues in the following two subsections, before proceeding into the algorithmic development. Finally, we will conclude this section by discussing a variant of the risk parity problem where an investor can incorporate estimated expected returns into the optimization problem. 

%-------------------------------------------------------------------------------------------
%       Probability distribution ambiguity set
%-------------------------------------------------------------------------------------------
\subsection{Probability distribution ambiguity set}\label{sec:ambiguity}

Our adversarial probability distribution \(\bm{p}\) belongs to the ambiguity set \(\mathcal{U}_{\bm{p}}\), which we proceed to formally define. A probability distribution must adhere to the simplex \(\mathcal{P}\) defined by the axioms of probability. Moreover, our goal is to define an ambiguity set where the adversarial distribution \(\bm{p}\) must lie within a maximum permissible distance \(d\) from the nominal distribution \(\bm{q}\). Thus, the ambiguity set is 
%-^-^-^-^-^-^-^-^-^-^-^-^-^-^-^-^-^-^-^-^-^-^-^-^-^-^-^-^-^-^-
\begin{equation}
\label{eq:ambigSet}	
	\mathcal{U}_{\bm{p}}(\bm{q},d) \triangleq \big\{\bm{p} \in\mathcal{P} : D(\bm{p},\bm{q}) \leq d\big\}
\end{equation}
%-^-^-^-^-^-^-^-^-^-^-^-^-^-^-^-^-^-^-^-^-^-^-^-^-^-^-^-^-^-^-
where \(D(\bm{p},\bm{q})\) is a convex function that models a given statistical distance measure, while \(d\in\mathbb{R}_+\) is a user-defined bound on the maximum permissible distance between \(\bm{p}\) and \(\bm{q}\). We note that, by definition, \(\mathcal{U}_{\bm{p}}\subseteq \mathcal{P}\).

A statistical distance measure can be used to quantify the similarity between two probability distributions. We limit our choice of statistical distance measures to a subset of convex functions that operate on discrete distributions.

The distributionally robust portfolio selection problem in \cite{calafiore2007ambiguous} used the KL divergence to define the ambiguity set. However, the KL divergence is not a proper metric since it is not symmetric and does not respect the triangle inequality. For two discrete probability distributions \(\bm{p},\ \bm{q}\in\mathcal{P}\), the KL divergence is defined as 
%-^-^-^-^-^-^-^-^-^-^-^-^-^-^-^-^-^-^-^-^-^-^-^-^-^-^-^-^-^-^-
\begin{equation}
\label{eq:KL}
    D_{\text{KL}}(\bm{p},\bm{q}) \triangleq \sum_{t=1}^T p_t \ln{\bigg(\frac{p_t}{q_t}\bigg)}
\end{equation}
%-^-^-^-^-^-^-^-^-^-^-^-^-^-^-^-^-^-^-^-^-^-^-^-^-^-^-^-^-^-^-
The asymmetry of the KL divergence becomes apparent if we reverse the order of the arguments \(\bm{p}\) and \(\bm{q}\) (i.e., \(D_{\text{KL}}(\bm{p},\bm{q}) \neq D_{\text{KL}}(\bm{q},\bm{p})\)). Moreover, the upper bound of the KL divergence is not properly defined, making it difficult to define an appropriate maximum permissible distance between \(\bm{p}\) and \(\bm{q}\). 

A measure closely related to the KL divergence is the JS divergence, which was introduced by \citeauthor{lin1991divergence} \cite{lin1991divergence}. Unlike the KL divergence, the JS divergence is symmetric and has finite bounds. The JS divergence is defined as 
%-^-^-^-^-^-^-^-^-^-^-^-^-^-^-^-^-^-^-^-^-^-^-^-^-^-^-^-^-^-^-
\begin{equation}
\label{eq:delJS}
\begin{aligned}
    D_{\text{JS}}(\bm{p},\bm{q}) &\triangleq \frac{1}{2} D_{\text{KL}}(\bm{p},\bm{m}) + \frac{1}{2} D_{\text{KL}}(\bm{q},\bm{m})\\
    &= \frac{1}{2}\sum_{t=1}^T p_t \ln{(p_t)} + q_t \ln{(q_t)} - (p_t + q_t) \ln{\bigg(\frac{p_t + q_t}{2}\bigg)},
\end{aligned}
\end{equation}
%-^-^-^-^-^-^-^-^-^-^-^-^-^-^-^-^-^-^-^-^-^-^-^-^-^-^-^-^-^-^-
where \(\bm{m} = \frac{1}{2}(\bm{p}+\bm{q}) \in \mathcal{P}\). Given that our definition of the KL divergence in \eqref{eq:KL} uses the natural logarithm, our definition of the JS divergence has the useful property of being bounded between zero and \(\ln(2)\), i.e.,
%-^-^-^-^-^-^-^-^-^-^-^-^-^-^-^-^-^-^-^-^-^-^-^-^-^-^-^-^-^-^-
\[
	0 \leq D_{\text{JS}}(\bm{p},\bm{q}) \leq \ln(2).
\]
%-^-^-^-^-^-^-^-^-^-^-^-^-^-^-^-^-^-^-^-^-^-^-^-^-^-^-^-^-^-^-
We can derive a proper metric from the JS divergence by taking its square root, which is known as  the JS distance \cite{endres2003new, fuglede2004jensen} (i.e., the JS distance is \(\sqrt{D_{\text{JS}}(\bm{p},\bm{q})}\)). This distance measure is bounded between zero and \(\sqrt{\ln(2)}\).

Next, we present the square of the Hellinger distance
%-^-^-^-^-^-^-^-^-^-^-^-^-^-^-^-^-^-^-^-^-^-^-^-^-^-^-^-^-^-^-
\begin{equation}
\label{eq:delH}
	D_{\text{H}}(\bm{p},\bm{q}) \triangleq \frac{1}{2}\sum_{t=1}^T \big(\sqrt{p_t} - \sqrt{q_t}\big)^2.
\end{equation}
%-^-^-^-^-^-^-^-^-^-^-^-^-^-^-^-^-^-^-^-^-^-^-^-^-^-^-^-^-^-^-
The Hellinger distance is a proper distance metric and, by its definition, is bounded between zero and one. We define \(D_{\text{H}}(\bm{p},\bm{q})\) as the squared Hellinger distance to improve computational tractability in practice.

The last distance measure we discuss is the TV distance,
%-^-^-^-^-^-^-^-^-^-^-^-^-^-^-^-^-^-^-^-^-^-^-^-^-^-^-^-^-^-^-
\begin{equation}
\label{eq:delTV}
	D_{\text{TV}}(\bm{p},\bm{q}) \triangleq \frac{1}{2} \|\bm{p} - \bm{q}\|_1, 
\end{equation}
%-^-^-^-^-^-^-^-^-^-^-^-^-^-^-^-^-^-^-^-^-^-^-^-^-^-^-^-^-^-^-
which is a proper distance metric. The TV distance is bounded between zero and one. 

The JS, Hellinger and TV distances are proper metrics and have finite bounds, which will allow us to define a maximum permissible distance \(d\) between \(\bm{q}\) and \(\bm{p}\). Moreover, (\ref{eq:delJS}--\ref{eq:delTV}) are convex functions over \(\bm{p}\in\mathcal{P}\) for any \(\bm{q}\in\mathcal{P}\). In turn, this means the ambiguity set \(\mathcal{U}_{\bm{p}}(\bm{q},d)\) is convex. We note that the functions \eqref{eq:delH} and \eqref{eq:delTV} can be implemented computationally by introducing auxiliary variables during optimization, but this does not fundamentally alter the problem. An example of how to computationally implement them is shown in Appendix \ref{app:distance}.

Our modelling framework provides sufficient flexibility for the user to prescribe their own choice of \(\bm{q}\in\mathcal{P}\). However, given the data-driven nature of our manuscript, we formally define the nominal probability distribution as a discrete uniform distribution, i.e., \(\bm{q} \triangleq [1/T\ \cdots\ 1/T]^\top \in \mathbb{R}^T\). This falls in line with our goal to define the most adversarial distribution \(\bm{p}\) relative to the distribution implied by the data. 

To finalize the definition of \(\mathcal{U}_{\bm{p}}\), we must determine the value of the maximum permissible distance \(d\) based on the investor's {\color{rev}desired degree of robustness}. The distance measures in (\ref{eq:delJS}--\ref{eq:delTV}) have theoretical lower and upper bounds. In particular, the upper bounds are only attainable if the nominal distribution \(\bm{q}\) differs the most from our adversarial distribution \(\bm{p}\). For a discrete probability distribution, this happens when both the nominal and adversarial distributions assign a probability \(q_i = p_j = 1\) for scenarios \(i\neq j\), with all other scenarios having a probability of zero. In practice, the theoretical upper bounds are unattainable under the assumption that \(\bm{q}\) is a discrete uniform distribution. Consider the following example of an extreme probability distribution, \(\bm{s} \triangleq [1\ 0\ \cdots\ 0]^\top \in\mathbb{R}^T\), which assigns all of its weight to a single scenario. The distribution \(\bm{s}\) is the most we can differ from the uniform distribution \(\bm{q}\). Thus, in practice, the true upper bound is defined as \(B(T) \triangleq D(\bm{s}, \bm{q})\in\mathbb{R}_+\), where the argument \(T\) corresponds to the dimension of the fixed distributions \(\bm{s}\) and \(\bm{q}\). We define the practical upper bounds of our three distance measures as
%-^-^-^-^-^-^-^-^-^-^-^-^-^-^-^-^-^-^-^-^-^-^-^-^-^-^-^-^-^-^-
\begin{subequations}
\begin{align}
	B_{\text{JS}}(T) &\triangleq D_{\text{JS}}(\bm{s},\bm{q}),\\
	B_{\text{H}}(T) &\triangleq D_{\text{H}}(\bm{s},\bm{q}),\\
	B_{\text{TV}}(T) &\triangleq D_{\text{TV}}(\bm{s},\bm{q}).
\end{align}
\end{subequations}
%-^-^-^-^-^-^-^-^-^-^-^-^-^-^-^-^-^-^-^-^-^-^-^-^-^-^-^-^-^-^-
For example, if our data consist of ten scenarios (\(T=10\)), then the upper bounds of the measures in (\ref{eq:delJS}--\ref{eq:delTV}) are
%-^-^-^-^-^-^-^-^-^-^-^-^-^-^-^-^-^-^-^-^-^-^-^-^-^-^-^-^-^-^-
\[
\begin{aligned}
	B_{\text{JS}}(10) &=\frac{1}{2}\Big((0.1)\ln(0.1) - (1.1)\ln(0.55) + (9)(0.1)\ln(2)\Big) \approx 0.5256,\\
	B_{\text{H}}(10) &= \frac{1}{2}\Big(1 - 2\sqrt{0.1} + (10)(0.1) \Big) \approx 0.6838,\\
	B_{\text{TV}}(10) &= \frac{1}{2}\Big( 0.9 + (9)(0.1) \Big) = 0.9.
\end{aligned}
\]
%-^-^-^-^-^-^-^-^-^-^-^-^-^-^-^-^-^-^-^-^-^-^-^-^-^-^-^-^-^-^-
As \(T\) increases, the upper bounds approach their theoretical values (i.e., as \(T\rightarrow\infty\), we have \(B_{\text{JS}}\rightarrow \ln(2)\), \(B_{\text{H}}\rightarrow 1\) and \(B_{\text{TV}}\rightarrow 1\)). The purpose of this exercise is to avoid defining \(d\) in terms of a theoretical upper bound. Instead, we seek to define it relative to the number of scenarios in our dataset, which is more relevant in practice. 

We define an appropriate maximum permissible distance \(d\) between our nominal and adversarial distributions based on the upper bound \(B\) and the user-defined {\color{rev}degree of robustness} \(0\leq\omega\leq 1\). In turn, we can use this to constrain the statistical distance between \(\bm{p}\) and \(\bm{q}\) (i.e., \(D(\bm{p},\bm{q})\leq d_{\text{JS}}\)). 

Recall that, in the case of the JS distance, we must square the {\color{rev}degree of robustness} since the JS divergence is the square of the JS distance. Thus, for a given {\color{rev}degree of robustness} \(\omega\) and number of scenarios \(T\), the maximum permissible distance is
%-^-^-^-^-^-^-^-^-^-^-^-^-^-^-^-^-^-^-^-^-^-^-^-^-^-^-^-^-^-^-
\begin{equation}
\label{eq:limitJS}
	d_{\text{JS}}(\omega,T) \triangleq \omega^2 B_{\text{JS}}(T).
\end{equation}
%-^-^-^-^-^-^-^-^-^-^-^-^-^-^-^-^-^-^-^-^-^-^-^-^-^-^-^-^-^-^-
Similarly, since \(D_{\text{H}}(\bm{p},\bm{q})\) in \eqref{eq:delH} is defined as the square of the Hellinger distance, the maximum permissible distance is
%-^-^-^-^-^-^-^-^-^-^-^-^-^-^-^-^-^-^-^-^-^-^-^-^-^-^-^-^-^-^-
\begin{equation}
\label{eq:limitH}
	d_{\text{H}}(\omega,T) \triangleq \omega^2 B_{\text{H}}(T).
\end{equation}
%-^-^-^-^-^-^-^-^-^-^-^-^-^-^-^-^-^-^-^-^-^-^-^-^-^-^-^-^-^-^-
Finally, given that the TV distance in \eqref{eq:delTV} is already a proper metric, we define the maximum permissible distance as
%-^-^-^-^-^-^-^-^-^-^-^-^-^-^-^-^-^-^-^-^-^-^-^-^-^-^-^-^-^-^-
\begin{equation}
\label{eq:limitTV}
	d_{\text{TV}}(\omega,T) \triangleq \omega B_{\text{TV}}(T).
\end{equation}
%-^-^-^-^-^-^-^-^-^-^-^-^-^-^-^-^-^-^-^-^-^-^-^-^-^-^-^-^-^-^-
To properly define the ambiguity set \(\mathcal{U}_{\bm{p}}\) in \eqref{eq:ambigSet}, the we must choose a single distance measure and define \(D(\bm{p},\bm{q})\) as one of the options in (\ref{eq:delJS}--\ref{eq:delTV}), with \(d\) defined accordingly.

%-------------------------------------------------------------------------------------------
%       Minimax problem
%-------------------------------------------------------------------------------------------
\subsection{Minimax problem}\label{sec:minimax}

For a given dataset \(\hat{\bm{\xi}}\), our problem is defined by the investor's choice of statistical distance measure and {\color{rev}degree of robustness}. Given this information, we aim to construct an optimal DRRP portfolio \(\bm{x}^*\). The nominal risk parity problem in \eqref{eq:minRP} is strictly convex for any given estimate of the covariance matrix \(\hat{\bm{\Sigma}}(\bm{p})\). Therefore, there exists a unique risk parity portfolio \(\bm{x}^{\text{RP}}(\bm{p})\) for every instance of \(\bm{p}\in\mathcal{P}\). 

The distinction between \(\bm{x}^*\) and \(\bm{x}^{\text{RP}}(\bm{p})\) is the following. The latter is the optimal risk parity portfolio for some arbitrary instance of \(\bm{p}\in\mathcal{P}\), as shown in \eqref{eq:minRP}. On the other hand, we use \(\bm{x}^*\) to denote the portfolio resulting from the most adversarial instance of \(\bm{p}\in\mathcal{U}_{\bm{p}}\) such that it maximizes our risk parity objective function. Thus, our optimal DRRP portfolio \(\bm{x}^*\) can be formulated as a minimax problem where we seek an optimal portfolio against an optimally adversarial discrete probability distribution.

The risk parity problem in \eqref{eq:minRP} requires that we first optimize an unconstrained problem and then project it onto the set of admissible portfolios. However, for simplicity, let us ignore the projection step and treat the unconstrained auxiliary variable \(\bm{y}\) as a proxy\footnote{Projecting the auxiliary variable \(\bm{y}\in\mathbb{R}_+^n\) onto \(\mathcal{X}\) is a trivial step, as shown in  \eqref{eq:RPsoln}.} for our asset weights \(\bm{x}\). Thus, for now, let the variables of our minimax problem be \(\bm{y}\) and \(\bm{p}\).

Recall our original definition of the portfolio variance, which was expressed in two equivalent forms in \eqref{eq:PortVarEst1} and \eqref{eq:PortVarEst2}. Moreover, recall our original definition of the risk parity objective function \(f_{\text{RP}}(\bm{y}, \bm{p})\) in \eqref{eq:minRPobj}. Using both expressions of the portfolio variance, we can restate our risk parity objective function in two equivalent forms
%-^-^-^-^-^-^-^-^-^-^-^-^-^-^-^-^-^-^-^-^-^-^-^-^-^-^-^-^-^-^-
\begin{subequations}
\label{eq:RPFunc}
\begin{align}
	\label{eq:RPFuncY}
	 f_{\text{RP}}(\bm{y},\bm{p}) &\triangleq \frac{1}{2}\bm{y}^\top \hat{\bm{\Sigma}}(\bm{p}) \bm{y} - \kappa \sum_{i=1}^n \ln(y_i)\\[1ex]
	 \label{eq:RPFuncP}
	 							&\triangleq \frac{1}{2}\bigg(\bm{p}^\top \hat{\bm{\pi}}^2(\bm{y}) - \bm{p}^\top \hat{\bm{\Theta}}(\bm{y}) \bm{p}\bigg) - \kappa \sum_{i=1}^n \ln(y_i),
\end{align}
\end{subequations}
%-^-^-^-^-^-^-^-^-^-^-^-^-^-^-^-^-^-^-^-^-^-^-^-^-^-^-^-^-^-^-
where \eqref{eq:RPFuncY} is exactly the same as \eqref{eq:minRPobj} and is restated for clarity, while \eqref{eq:RPFuncP} presents \(f_{\text{RP}}(\bm{y},\bm{p})\) explicitly in terms of \(\bm{p}\). Formulating the objective function in these two equivalent forms allows us to observe how the function acts upon both the decision variable \(\bm{y}\) and the adversarial probability \(\bm{p}\). {\color{rev}It follows that the corresponding DRRP problem can be stated as a minimax problem,}
%-^-^-^-^-^-^-^-^-^-^-^-^-^-^-^-^-^-^-^-^-^-^-^-^-^-^-^-^-^-^-
\begin{equation}
\label{eq:RPminimax}
	\min_{\bm{y}\in\mathbb{R}_+^n}\ \max_{\bm{p}\in\mathcal{U}_{\bm{p}}}\quad f_{\text{RP}}(\bm{y},\bm{p}). 
\end{equation}
%-^-^-^-^-^-^-^-^-^-^-^-^-^-^-^-^-^-^-^-^-^-^-^-^-^-^-^-^-^-^-

{\color{rev}
\begin{theorem}[\citeauthor{neumann1928theorie}'s minimax theorem \cite{neumann1928theorie}]
\label{thm:minimax}

	Let \(\mathcal{Y}\subset\mathbb{R}^n\) and \(\mathcal{P}\subset\mathbb{R}^T\). If the function \(f:\mathcal{Y}\times\mathcal{P} \rightarrow \mathbb{R}\) is continuous convex--concave, where \( f(\cdot, \bm{p}): \mathcal{Y} \rightarrow \mathbb{R}\) is convex for any fixed \(\bm{p}\) and \( f(\bm{y},\cdot): \mathcal{P} \rightarrow \mathbb{R}\) is concave for any fixed \(\bm{y}\), then we have that 
%-^-^-^-^-^-^-^-^-^-^-^-^-^-^-^-^-^-^-^-^-^-^-^-^-^-^-^-^-^-^-
\begin{equation}
\label{eq:minimaxThm}
	\max_{\bm{p}\in\mathcal{P}}\ \min_{\bm{y}\in\mathcal{Y}}\ f(\bm{y},\bm{p}) = \min_{\bm{y}\in\mathcal{Y}}\ \max_{\bm{p}\in\mathcal{P}}\ f(\bm{y},\bm{p}),
\end{equation}
%-^-^-^-^-^-^-^-^-^-^-^-^-^-^-^-^-^-^-^-^-^-^-^-^-^-^-^-^-^-^-
where any local optimum is a global optimum. Assume the global optimal solution of \eqref{eq:minimaxThm} is \((\bm{y}^*,\bm{p}^*)\). It follows that the minimax theorem can be equivalently restated as the saddle-point inequality
%-^-^-^-^-^-^-^-^-^-^-^-^-^-^-^-^-^-^-^-^-^-^-^-^-^-^-^-^-^-^-
\[
	f(\bm{y}^*,\bm{p}) \leq f(\bm{y}^*,\bm{p}^*) \leq f(\bm{y},\bm{p}^*)\ \  \forall\ \bm{y}\in\mathcal{Y},\ \bm{p}\in\mathcal{P}.
\]
%-^-^-^-^-^-^-^-^-^-^-^-^-^-^-^-^-^-^-^-^-^-^-^-^-^-^-^-^-^-^-
\end{theorem}
}

As we saw in Section \ref{sec:param}, both \(\hat{\bm{\Sigma}}(\bm{p})\) and \(\hat{\bm{\Theta}}(\bm{y})\) are PSD for any \(\bm{p}\in\mathcal{P}\) and \(\bm{y}\in\mathbb{R}^n_+\), respectively. Therefore, \(f_{\text{RP}}(\bm{y},\cdot): \mathcal{P}\rightarrow\mathbb{R}\)   is concave for every \(\bm{y}\in\mathbb{R}_+^n\), {\color{rev}while, by Lemma \ref{lemma:RP}, the function \(f_{\text{RP}}(\cdot,\bm{p}): \mathbb{R}_+^n\rightarrow\mathbb{R}\) is strictly convex for every \(\bm{p}\in\mathcal{P}\)}. Moreover, the sets \(\mathcal{X}\) and \(\mathcal{U}_{\bm{p}}\) are convex. {\color{rev}This means that \(f_{\text{RP}}\) is} a convex--concave minimax problem. {\color{rev}Thus, by Theorem \ref{thm:minimax},} the saddle-point inequality holds for the DRRP problem, i.e., 
%-^-^-^-^-^-^-^-^-^-^-^-^-^-^-^-^-^-^-^-^-^-^-^-^-^-^-^-^-^-^-
\begin{equation}
\label{eq:minimaxineq}
	f_{\text{RP}}(\bm{y}^*,\bm{p}) \leq f_{\text{RP}}(\bm{y}^*,\bm{p}^*) \leq f_{\text{RP}}(\bm{y},\bm{p}^*)\ \  \forall\ \bm{y}\in\mathbb{R}_+^n,\ \bm{p}\in\mathcal{U}_{\bm{p}},
\end{equation}
%-^-^-^-^-^-^-^-^-^-^-^-^-^-^-^-^-^-^-^-^-^-^-^-^-^-^-^-^-^-^-
where \((\bm{y}^*,\bm{p}^*)\) is the optimal solution (i.e., the saddle point) of \(f_{\text{RP}}(\bm{y},\bm{p})\).

The maximization step in \eqref{eq:RPminimax} is also meaningful in a financial context. Consider the definition of \(f_{\text{RP}}(\bm{y},\bm{p})\) in \eqref{eq:RPFuncP} where, without loss of generality, we have defined the portfolio variance using the unnormalized proxy variable \(\bm{y}\). The maximization step in \eqref{eq:RPminimax} pertains solely to the portfolio variance given that the logarithmic barrier term only acts on the variable \(\bm{y}\). Thus, intuitively, the maximization step aims to find the most adversarial probability distribution \(\bm{p}\) such that we attain the worst-case instance of the portfolio variance. This leads to the following conclusion: the minimax problem in \eqref{eq:RPminimax} seeks the optimal risk parity portfolio with respect to the worst-case portfolio variance. 

{\color{rev}
%-------------------------------------------------------------------------------------------
%       Robust counterpart of the risk parity problem
%-------------------------------------------------------------------------------------------
\subsection{Robust counterpart of the risk parity problem}\label{sec:robConvex}

The convex--concave nature of the DRRP minimax problem means that we are able to exploit the dual of the maximization step in order to reformulate the problem as a simple convex minimization problem. In particular, we will use the robust framework proposed by \citeauthor{ben2013robust} \cite{ben2013robust}, which prescribes a method to reduce a convex--linear minimax problem with an ambiguity set similar to the one presented in \eqref{eq:ambigSet} into a straightforward convex minimization problem. 

The framework in \cite{ben2013robust} requires that the initial minimax problem has the generic form
%-^-^-^-^-^-^-^-^-^-^-^-^-^-^-^-^-^-^-^-^-^-^-^-^-^-^-^-^-^-^-
\begin{equation}
\label{eq:BenTal}
	\min_{\bm{z}\in\mathcal{Z}}\ \max_{\bm{p}\in\mathcal{U}}\quad (\bm{a} + \bm{C} \bm{p})^\top \bm{f}(\bm{z}),
\end{equation}
%-^-^-^-^-^-^-^-^-^-^-^-^-^-^-^-^-^-^-^-^-^-^-^-^-^-^-^-^-^-^-
where \(\bm{a}\in\mathbb{R}^k\) and \(\bm{C}\in\mathbb{R}^{k\times T}\) are constants, \(\bm{z}\in\mathcal{Z}\subseteq\mathbb{R}^n\) is some generic decision variable, and \(\bm{f}:\mathbb{R}^n \rightarrow\mathbb{R}^k\) is some generic vector-valued function. It is clear that \eqref{eq:BenTal} is linear in \(\bm{p}\). Thus, at first glance, we are unable apply this robust framework to the DRRP problem in \eqref{eq:RPFuncP} since it is quadratic in \(\bm{p}\).

In particular, the quadratic expression of \(\bm{p}\) in \eqref{eq:RPFuncP} stems from the presence of the portfolio variance. However, as shown by \citeauthor{gotoh2018robust} \cite{gotoh2018robust}, the variance can be recast as a linear function in \(\bm{p}\),
\begin{align}
\hat{\sigma}_\pi^2(\bm{y},\bm{p}) &= \mathbb{E}\Big[ \big(\pi(\bm{y}) - \mathbb{E}[\pi(\bm{y})]\big)^2 \Big] = \min_{c}\ \mathbb{E}\big[(\pi(\bm{y}) - c)^2\big]\nonumber\\ 
	\label{eq:PortVarEst3}
    &= \min_{c}\ \sum_{t=1}^T p_t\cdot \big([\hat{\bm{\xi}}^\top \bm{y}]_t - c\big)^2,
\end{align}
where \(c\in\mathbb{R}\) is an auxiliary variable that serves as a placeholder for the expected value of \(\pi(\bm{y})\). Note that we have replaced the asset weight variable \(\bm{x}\) with its unnormalized proxy \(\bm{y}\) to align it with our previous derivation of the minimax problem. 

After linearizing the variance in \(\bm{p}\), we can restate the optimization problem in (\ref{eq:RPFunc}--\ref{eq:RPminimax}) as follows
%-^-^-^-^-^-^-^-^-^-^-^-^-^-^-^-^-^-^-^-^-^-^-^-^-^-^-^-^-^-^-
\[
	\min_{\bm{y}\in\mathbb{R}_+^n, c}\ \ \max_{\bm{p}\in\mathcal{U}}\quad \sum_{t=1}^T p_t \cdot \big([\hat{\bm{\xi}}^\top \bm{y}]_t - c\big)^2 - \kappa \sum_{i=1}^n \ln(y_i).
\]	
%-^-^-^-^-^-^-^-^-^-^-^-^-^-^-^-^-^-^-^-^-^-^-^-^-^-^-^-^-^-^-
Since this version of the minimax problem is linear in \(\bm{p}\), we can use the robust framework in \cite{ben2013robust} to introduce a distributionally \textit{robust counterpart}\footnote{\color{rev}For the remainder of this manuscript, we refer to the optimization problem in \eqref{eq:singleConvex} as the `robust counterpart' to differentiate it from the DRRP minimax problem.} to the risk parity problem as follows
%-^-^-^-^-^-^-^-^-^-^-^-^-^-^-^-^-^-^-^-^-^-^-^-^-^-^-^-^-^-^-
\begin{equation}
\label{eq:singleConvex}
	\min_{\lambda\in\mathbb{R}_+, \bm{y}\in\mathbb{R}_+^n, c, \rho}\ \rho + d\cdot \lambda + \frac{\lambda}{T} \sum_{t=1}^T D^*\bigg( \frac{\big([\hat{\bm{\xi}}^\top \bm{y}]_t - c\big)^2 - \rho}{\lambda}\bigg) - \kappa \sum_{i=1}^n \ln(y_i),
\end{equation}
%-^-^-^-^-^-^-^-^-^-^-^-^-^-^-^-^-^-^-^-^-^-^-^-^-^-^-^-^-^-^-
where \(d = d(\omega,T)\) is the maximum permissible distance between \(\bm{p}\) and \(\bm{q}\), and \(D^*(\cdot)\) is the conjugate function of the statistical distance function \(D(\cdot)\). For example, in the case of the Hellinger distance, \(D_{\text{H}}^*(a) = a/(1-a)\) for \(a<1\). Other examples of statistical distance measures and their conjugates can be found in \cite{ben2013robust}. By design, the optimization problem in \eqref{eq:singleConvex} is convex over \(\bm{y}\in\mathbb{R}_+^n\), \(\lambda\in\mathbb{R}_+\), and \(c,\rho \in\mathbb{R}\) (for reference, see \cite{ben2013robust}). Note that after solving this problem, we must project the optimal solution \(\bm{y}^*\) onto the set of admissible portfolios, i.e., \(\bm{x}^* = \text{\large $\Pi$}_{\mathcal{X}}(\bm{y}^*)\). Thus, after linearizing the problem in \(\bm{p}\), we are able to reformulate the DRRP minimax problem into its robust counterpart, which is a standard convex minimization problem. 

Nevertheless, the robust counterpart suffers from a major drawback: it is highly non-linear. Thus, as the DRRP problem increases in size (both in number of assets \(n\) and scenarios \(T\)), the computational cost of solving the robust counterpart increases significantly. This is compounded by the computational cost associated with the linearization of the portfolio variance, which requires us to model the variance as a minimization problem of its own. The numerical performance of the robust counterpart in \eqref{eq:singleConvex} is evaluated later on in Section \ref{sec:NumPerf}. Although the robust counterpart provides a theoretically sound avenue to solve the DRRP problem, our experimental results show that it becomes numerically intractable as the scale of the problem increases. Therefore, this motivates our introduction of a numerically efficient gradient-based algorithm to solve the DRRP problem. As we will see later on, a gradient-based algorithm that iteratively alternates between the minimization and maximization steps is much more efficient than the robust counterpart in practice. The algorithmic development is presented in Sections \ref{sec:algoGDA} and \ref{sec:algoGA}. 
}

%-------------------------------------------------------------------------------------------
%       Projected gradient descent--ascent 
%-------------------------------------------------------------------------------------------
\subsection{Projected gradient descent--ascent}\label{sec:algoGDA}

{\color{rev}We turn our attention to the enhancement of numerical performance by introducing an algorithm designed to exploit the convex--concave structure of the original DRRP minimax problem. We begin by discussing a gradient-based algorithm, which we will refer to as the PGDA algorithm. The PGDA algorithm works by sequentially alternating between descending in \(\bm{y}\) and ascending in \(\bm{p}\) until convergence.}

To retain feasibility after each iteration, we project each step in \(\bm{y}\) and \(\bm{p}\) onto the sets \(\mathbb{R}_+^n\) and \(\mathcal{U}_{\bm{p}}\), respectively. In particular, the non-linearity of the statistical distance measure means that the projection onto the ambiguity set \(\mathcal{U}_{\bm{p}}\) is non-trivial and cannot be solved in closed form. Instead, the projection must be solved as a constrained optimization problem. A Euclidean projection ensures the problem is strictly convex, guaranteeing the uniqueness of our solution. We define the projection of some arbitrary vector \(\bm{u}\in\mathbb{R}^T\) onto the set \(\mathcal{U}_{\bm{p}}\) as follows,
%-^-^-^-^-^-^-^-^-^-^-^-^-^-^-^-^-^-^-^-^-^-^-^-^-^-^-^-^-^-^-
\begin{subequations}
\label{eq:maxRPproj}
\begin{equation}
	\label{obj:maxRPproj}
	\text{\large $\Pi$}_{\mathcal{U}_{\bm{p}}}(\bm{u}) \triangleq\ \ \argmin_{\bm{p}}\quad \|\bm{u} - \bm{p}\|_2^2\\
\end{equation}
\vspace{\dimexpr-\abovedisplayskip-\belowdisplayskip -0.5em}
\begin{alignat}{4}
		&\qquad\qquad\qquad\ \text{s.t.}\qquad  &\bm{1}^T \bm{p} 	&= 1,\label{const:axiom2}\\
		& &D(\bm{p},\bm{q}) 	&\leq d,\label{const:distance} \\
		& &\bm{p}					&\geq 0,\label{const:axiom1}
\end{alignat}
\end{subequations}
%-^-^-^-^-^-^-^-^-^-^-^-^-^-^-^-^-^-^-^-^-^-^-^-^-^-^-^-^-^-^-
where the constraints (\ref{const:axiom2}--\ref{const:axiom1}) arise from the ambiguity set \(\mathcal{U}_{\bm{p}}(\bm{q},d)\). In particular, constraint \eqref{const:distance} is shown with respect to a generic distance measure \(D(\bm{p},\bm{q})\), which can be defined by the user as any of the measures in (\ref{eq:delJS}--\ref{eq:delTV}) with an appropriate maximum permissible distance \(d\). Thus, for some point \(\bm{u}\), the projection \(\text{\large $\Pi$}_{\mathcal{U}_{\bm{p}}}(\bm{u})\) finds the closest solution within the ambiguity set \(\mathcal{U}_{\bm{p}}(\bm{q},d)\). 

Likewise, we retain feasibility in the descent step by projecting each iteration in the descent direction onto the set \(\mathbb{R}_+^n\). This projection is trivial and, for some arbitrary point \(\bm{z}\in\mathbb{R}^n\), can be computed as follows
%-^-^-^-^-^-^-^-^-^-^-^-^-^-^-^-^-^-^-^-^-^-^-^-^-^-^-^-^-^-^-
\begin{equation}
	\label{eq:minRPprojY}
	\text{\large $\Pi$}_{\mathbb{R}_+^n}(\bm{z}) \triangleq\ \begin{cases} z_i \quad &\text{if } z_i > 0\\
						0^+ &\text{otherwise} \end{cases}\quad \text{for } i = 1, \dots, n,
\end{equation}
%-^-^-^-^-^-^-^-^-^-^-^-^-^-^-^-^-^-^-^-^-^-^-^-^-^-^-^-^-^-^-
where we inspect every element of \(\bm{z}\) and set any non-positive element to an arbitrarily small positive value.\footnote{Any non-positive value must be replaced with a strictly positive value, \(0^+\) due to the logarithm barrier term in our objective function. Thus,\(0^+\) can be set to some small positive value during implementation.}

Like an unconstrained gradient descent--ascent algorithm, we take steps to descend in \(\bm{y}\) and ascend in \(\bm{p}\) in the direction of the respective gradients of \(f_{\text{RP}}(\bm{y},\bm{p})\). The gradients of \(f_{\text{RP}}(\bm{y},\bm{p})\) are
%-^-^-^-^-^-^-^-^-^-^-^-^-^-^-^-^-^-^-^-^-^-^-^-^-^-^-^-^-^-^-
\begin{align}
	\nabla_{\bm{y}} f_{\text{RP}}(\bm{y},\bm{p}) &= \hat{\bm{\Sigma}}(\bm{p}) \bm{y} - \kappa \bm{y}^{-1},\label{eq:gradY}\\[1ex]
	\nabla_{\bm{p}} f_{\text{RP}}(\bm{y},\bm{p}) &= \frac{1}{2}\hat{\bm{\pi}}^2(\bm{y}) - \hat{\bm{\Theta}}(\bm{y}) \bm{p},\label{eq:gradP}
\end{align}
%-^-^-^-^-^-^-^-^-^-^-^-^-^-^-^-^-^-^-^-^-^-^-^-^-^-^-^-^-^-^-
where \(\bm{y}^{-1} = [1/y_1\ \cdots\ 1/y_n]^\top\). 

Given that the feasible sets \(\mathbb{R}_+^n\) and \(\mathcal{U}_{\bm{p}}\) are convex, we can design the search directions in both \(\bm{y}\) and \(\bm{p}\) such that we retain feasibility after each iteration. Assume we have some feasible solutions \(\bm{y}^k \in \mathbb{R}_+^n\) and \(\bm{p}^k\in\mathcal{U}_{\bm{p}}\). The search directions are
%-^-^-^-^-^-^-^-^-^-^-^-^-^-^-^-^-^-^-^-^-^-^-^-^-^-^-^-^-^-^-
\begin{align}
	\bm{g}^k &\triangleq \text{\large $\Pi$}_{\mathbb{R}_+^n}\Big(\bm{y}^{k} - \alpha_k \nabla_{\bm{y}} f_{\text{RP}}(\bm{y}^k,\bm{p}^k) \Big) - \bm{y}^k,\nonumber \\
	\label{eq:projStep}\bm{h}^k &\triangleq \text{\large $\Pi$}_{\mathcal{U}_{\bm{p}}}\Big(\bm{p}^{k} + \gamma_k \nabla_{\bm{p}} f_{\text{RP}}(\bm{y}^{k},\bm{p}^k) \Big) - \bm{p}^k,
\end{align}
%-^-^-^-^-^-^-^-^-^-^-^-^-^-^-^-^-^-^-^-^-^-^-^-^-^-^-^-^-^-^-
where \(\alpha_k\) and \(\gamma_k\) are the step sizes in each direction. To ensure that our next iteration remains within the feasible set, we define the search parameters \(\eta_{\bm{y}}, \eta_{\bm{p}}\in[0,1]\). Thus, our next iterations in each direction are
%-^-^-^-^-^-^-^-^-^-^-^-^-^-^-^-^-^-^-^-^-^-^-^-^-^-^-^-^-^-^-
\begin{align}
	\bm{y}^{k+1} &= \bm{y}^k + \eta_{\bm{y}} \bm{g}^k,\nonumber \\
	\label{eq:gradStep}\bm{p}^{k+1} &= \bm{p}^k + \eta_{\bm{p}} \bm{h}^k.
\end{align}
%-^-^-^-^-^-^-^-^-^-^-^-^-^-^-^-^-^-^-^-^-^-^-^-^-^-^-^-^-^-^-
The points \(\bm{y}^{k+1}\) and \(\bm{p}^{k+1}\) are the result of linear combinations between two feasible points in each set, respectively. Since the sets are convex, the points \(\bm{y}^{k+1}\) and \(\bm{p}^{k+1}\) are feasible by definition.

We defer to the Barzilai--Borwein method \cite{barzilai1988two} to define the step sizes \(\alpha_k\) and \(\gamma_k\). Specifically, we use the following definition of the Barzilai--Borwein method. {\color{rev}For any iteration \(k\geq 1\), we have 
%-^-^-^-^-^-^-^-^-^-^-^-^-^-^-^-^-^-^-^-^-^-^-^-^-^-^-^-^-^-^-
\begin{align}
	\alpha_{k} &= \frac{\|\bm{y}^k - \bm{y}^{k-1}\|_2^2}{\phantom{\Big(} \big| (\bm{y}^k - \bm{y}^{k-1})^\top \big(\nabla_{\bm{y}} f_{\text{RP}}(\bm{y}^k,\bm{p}^k) - \nabla_{\bm{y}} f_{\text{RP}}(\bm{y}^{k-1},\bm{p}^{k-1}) \big) \big|},\label{eq:alpha} \\[1.5ex]
	\gamma_{k} &= \frac{\|\bm{p}^k - \bm{p}^{k-1}\|_2^2}{\phantom{\Big(} \big| (\bm{p}^k - \bm{p}^{k-1})^\top \big(\nabla_{\bm{p}} f_{\text{RP}}(\bm{y}^k,\bm{p}^k) - \nabla_{\bm{p}} f_{\text{RP}}(\bm{y}^{k-1},\bm{p}^{k-1}) \big)\big|}.\label{eq:gamma}
\end{align}
%-^-^-^-^-^-^-^-^-^-^-^-^-^-^-^-^-^-^-^-^-^-^-^-^-^-^-^-^-^-^-
In their seminal work, \citeauthor{barzilai1988two} \cite{barzilai1988two} produce a computationally inexpensive quasi-Newton step size through a scaled product of the identity matrix designed to approximate the Hessian of \(f_{\text{RP}}\) in either direction, i.e., \(\nabla_{\bm{y}}^2 f_{\text{RP}}(\bm{y}^k,\bm{p}^k) \approx (\alpha_k \bm{I}_n)^{-1}\) and \(\nabla_{\bm{p}}^2 f_{\text{RP}}(\bm{y}^k,\bm{p}^k) \approx (\gamma_k \bm{I}_T)^{-1}\), where \(\bm{I}_d\) is the identity matrix of dimension \(d\).} 

The Barzilai--Borwein step size is sometimes referred to as the `spectral step size'. In the case of projected gradient descent, this class of algorithms is sometimes referred to as `spectral projected gradient descent' \cite{birgin2000nonmonotone}. 

Next, we discuss how to determine the search parameters \(\eta_{\bm{y}}, \eta_{\bm{p}}\in(0,1]\). Specifically, we favour the non-monotone Grippo--Lampariello--Lucidi (GLL) line search proposed in \cite{grippo1986nonmonotone}. The GLL line search method has been shown to work well with spectral projected gradient descent and ensures global convergence on closed convex sets \cite{birgin2000nonmonotone, dai2005projected}. 

A brief overview of this line search method follows. Consider the descent step in \(\bm{y}\). For a given integer \(m\geq 1\), we are searching for \(\eta_{\bm{y}}\in(0,1]\) such that 
%-^-^-^-^-^-^-^-^-^-^-^-^-^-^-^-^-^-^-^-^-^-^-^-^-^-^-^-^-^-^-
\begin{equation}
\label{eq:GLL}
	f_{\text{RP}}(\bm{y}^k+\eta_{\bm{y}} \bm{g}^k,\bm{p}^k) \leq \max\limits_{j\in\mathcal{J}} f_{\text{RP}}(\bm{y}^{k-j}, \bm{p}^{k-j}) + \beta \eta_{\bm{y}} (\bm{g}^k)^\top \nabla_{\bm{y}} f_{\text{RP}}(\bm{y}^k,\bm{p}^k)
\end{equation}
%-^-^-^-^-^-^-^-^-^-^-^-^-^-^-^-^-^-^-^-^-^-^-^-^-^-^-^-^-^-^-
where \(\mathcal{J}\triangleq\big\{j\in\mathbb{Z} : 0\leq j\leq \min \{k,m-1\}\big\}\) and \(\beta\in(0,1)\) is some predefined constant. Intuitively, a larger value of \(\eta_{\bm{y}}\) corresponds to a more aggressive descent step. Thus, we can set \(\eta_{\bm{y}} = 1\) and shrink it appropriately by some fixed factor \(\tau\in(0,1)\), resulting in an inexact but fast method to determine an appropriate value for \(\eta_{\bm{y}}\). 

The GLL method stems from an Armijo-type line search, but it allows us to take greedier steps. For example, if we set \(m=1\), then we revert back to a traditional Armijo-type line search method and the condition in \eqref{eq:GLL} causes our objective function to decrease monotonically. Thus, by considering multiple previous iterations of the objective value we allow for a non-monotonic decrease.

For the purpose of the PGDA algorithm, the ascent direction follows the same logic. However, we do not discuss it in detail for the sake of brevity. Instead, the complete PGDA algorithm is presented in Algorithm \ref{algo:DRORP}, which shows how to calculate the steps in the descent and ascent directions. 

Although the global convergence of spectral projected gradient descent with a GLL line search has been established \cite{birgin2000nonmonotone, dai2005projected}, we purposely avoid claiming that this is true for PGDA. However, we note that for appropriate step sizes, the convergence of constrained convex--concave minimax problems has been previously established {\color{rev}(see \cite{nedic2009subgradient} for a complete proof of convergence for a PGDA algorithm).} For the purpose of this manuscript, the PGDA algorithm serves solely as a stepping stone towards the development of the SCP--PGA algorithm.

%-^-^-^-^-^-^-^-^-^-^-^-^-^-^-^-^-^-^-^-^-^-^-^-^-^-^-^-^-^-^-
\begin{algorithm}
	\KwIn{Data \(\hat{\bm{\xi}}\in\mathbb{R}^{n\times T}\); {\color{rev}Degree of robustness} \(\omega \in(0,1)\);  Distance measure \{JS, Hellinger, TV\}; Nominal distribution \(\bm{q}\in\mathcal{P}\); Risk parity constant \(\kappa > 0\); Initial step sizes \(\alpha_0, \gamma_0 > 0\); Initial proxy portfolio \(\bm{y}^0\); Convergence tolerance \(\varepsilon_0\); Search control parameters \(\beta, \tau\in (0,1)\); GLL parameter \(m\geq 1\)}
	Find the distance limit \(d(\omega,T)\) as shown in either of (\ref{eq:limitJS}--\ref{eq:limitTV})\;
	Initialize the adversarial distribution: \(\bm{p}^0 = \bm{q}\)\;
    Initialize the convergence measure: \(\varepsilon = 1\)\;
	Initialize the counter: \(k = 0\)\;
    \While{\(\varepsilon > \varepsilon_0\)}{
    	\If{\(k \geq 1\)}{
            Update \(\alpha_k\) as shown in \eqref{eq:alpha}\;
            Update \(\gamma_k\) as shown in \eqref{eq:gamma}\;
        }

		\(\bm{g}^k = \text{\large $\Pi$}_{\mathbb{R}_+^n}\Big(\bm{y}^{k} - \alpha_k \nabla_{\bm{y}} f_{\text{RP}}(\bm{y}^k,\bm{p}^k)\Big) - \bm{y}^k\)\;      
                \(\eta_{\bm{y}} = 1\)\;
		\(\bar{\bm{y}} = \bm{y}^k + \eta_{\bm{y}} \bm{g}^k\)\;
		\While{\(f_{\text{RP}}(\bar{\bm{y}},\bm{p}^k) > \max\limits_{j\in\mathcal{J}} f_{\text{RP}}(\bm{y}^{k-j}, \bm{p}^{k-j}) + \beta \eta_{\bm{y}} (\bm{g}^k)^\top \nabla_{\bm{y}} f_{\text{RP}}(\bm{y}^k,\bm{p}^k) \)}{
		\(\eta_{\bm{y}} = \eta_{\bm{y}} \tau\)\;
        \(\bar{\bm{y}} = \bm{y}^k + \eta_{\bm{y}} \bm{g}^k\)\;
        }
        \(\bm{y}^{k+1} = \bar{\bm{y}}\)\;
		\(\bm{h}^k = \text{\large $\Pi$}_{\mathcal{U}_{\bm{p}}}\Big(\bm{p}^{k} + \gamma_k \nabla_{\bm{p}} f_{\text{RP}}(\bm{y}^k,\bm{p}^k)\Big) - \bm{p}^k\)\;      
        \(\eta_{\bm{p}} = 1\)\;
		\(\bar{\bm{p}} = \bm{p}^k + \eta_{\bm{p}} \bm{h}^k\)\;
		\While{\(f_{\text{RP}}(\bm{y}^k,\bar{\bm{p}}) < \min\limits_{j\in\mathcal{J}} f_{\text{RP}}(\bm{y}^{k-j}, \bm{p}^{k-j}) + \beta \eta_{\bm{p}} (\bm{h}^k)^\top \nabla_{\bm{p}} f_{\text{RP}}(\bm{y}^k,\bm{p}^k) \)}{
		\(\eta_{\bm{p}} = \tau\eta_{\bm{p}}\)\;
        \(\bar{\bm{p}} = \bm{p}^k + \eta_{\bm{p}} \bm{h}^k\)\;
        }
        \(\bm{p}^{k+1} = \bar{\bm{p}}\)\;
        
        \If{\(k \geq 1\)}{
            \(\color{rev}\displaystyle \varepsilon = \frac{\| \bm{p}^{k+1} - \bm{p}^k \|_2}{\| \bm{p}^k \|_2}\)\;
        }
        \(k = k+1\)\;
    }
    Find the optimal portfolio: \(\bm{x}^* = \bm{x}^{\text{RP}}(\bm{p}^k)\)\;
\KwOut{Optimal DRRP portfolio \(\bm{x}^*\)}
\caption{PGDA for DRRP portfolio optimization}
\label{algo:DRORP}
\end{algorithm}
%-^-^-^-^-^-^-^-^-^-^-^-^-^-^-^-^-^-^-^-^-^-^-^-^-^-^-^-^-^-^-

%-------------------------------------------------------------------------------------------
%       SCP-PGA
%-------------------------------------------------------------------------------------------
\subsection{Sequential convex programming with projected gradient ascent}\label{sec:algoGA}

Our discussion of the PGDA algorithm served two purposes. First, it provided a straightforward approach to solve a convex--concave minimax problem. More importantly, it showed the steps required to navigate such a problem and highlighted some structural weaknesses. In particular, the PGDA algorithm requires that we determine two appropriate step sizes, \(\alpha_k\) and \(\gamma_k\), during each iteration. Moreover, the set \(\mathbb{R}_+^n\) in which \(\bm{y}\) exists is not compact. Therefore, the PGDA algorithm necessitates careful initialization and, in general, may be prone to diverge. 

The PGDA has three structural weaknesses. First, the design of the risk parity problem in \eqref{eq:minRP} means that the descent step in the PGDA algorithm must ignore the budget equality constraint. In other words, the algorithm operates in the unbounded set \(\mathbb{R}_+^n\) instead of the compact set \(\mathcal{X}\). Only after convergence of the PGDA algorithm do we project our solution onto \(\mathcal{X}\). 

Second, the PGDA algorithm is twice as susceptible the problem of vanishing gradients. As we approach a saddle point, the gradient information in both directions starts to vanish, slowing the convergence of the algorithm to an optimal saddle point. 

The third and final weakness is the burden placed on the user to define an initial step size in both \(\bm{y}\) and \(\bm{p}\) directions, as well as the initial guess \(\bm{y}^0\). Since the proxy variable \(\bm{y}\) does not have an upper bound, an improperly sized \(\bm{y}^0\) may slow down convergence. 

These three weaknesses can be remediated by redesigning the algorithm to operate directly on the set of admissible portfolios \(\mathcal{X}\). The strict convexity of the nominal risk parity problem in \eqref{eq:minRP} means that there exists a unique risk parity portfolio \(\bm{x}^{\text{RP}}(\bm{p})\) for every \(\bm{p}\in\mathcal{U}_{\bm{p}}\). Assume we have an ascent algorithm and let \(\bm{p}^k\) be the \(k^{\text{th}}\) iteration of our adversarial probability. Then, for every \(k=0,1,\dots\), there exists a corresponding risk parity portfolio \(\bm{x}^{\text{RP}}(\bm{p}^k)\). {\color{rev}The same holds for the proxy variable \(\bm{y}\), i.e., there exists \(\bm{y}^k = \bm{y}^{\text{RP}}(\bm{p}^k)\) for every iteration \(k=0,1,\dots\).} Thus, we can formulate an algorithm that ascends in \(\bm{p}\in\mathcal{U}_{\bm{p}}\) while enforcing the risk parity condition in {\color{rev}\(\bm{y}\in\mathbb{R}_+^n\)} after every iteration. 

Conversely, we can interpret this algorithm as solving a sequence of convex problems. Specifically, we solve the {\color{rev}proxy risk parity problem \(\bm{y}^k = \bm{y}^{\text{RP}}(\bm{p}^k)\)}, where we update the covariance matrix \(\hat{\bm{\Sigma}}(\bm{p}^k)\) after every iteration \(k\). Thus, the resulting algorithm needs only to ascend in \(\bm{p}\in\mathcal{U}_{\bm{p}}\), meaning it can be solved using PGA. In turn, this means that the user no longer needs to define any of the initial conditions and updates associated with \(\bm{y}\). Given that the proposed PGA algorithm involves iteratively solving a sequence of convex problems, we refer to it as the `SCP--PGA' algorithm. 

{\color{rev}By definition, \(\bm{y}^k = \bm{y}^{\text{RP}}(\bm{p}^k)\) is the minimizer of \(f_{\text{RP}}\) for any fixed \(\bm{p}\in\mathcal{U}_{\bm{p}}\). Thus, if \((\bm{y}^*,\bm{p}^*)\) is the saddle point of \(f_{\text{RP}}\), the saddle-point inequality in \eqref{eq:minimaxineq} can be restated as the following theorem.

\begin{lemma}
\label{lemma:ineq}
	If the function \(f_{\text{RP}}: \mathbb{R}_+^n\times\mathcal{U}_{\bm{p}} \rightarrow \mathbb{R}\) is convex--concave and \(f_{\text{RP}}(\cdot,\bm{p}):\mathbb{R}_+^n\rightarrow\mathbb{R}\) is strictly convex, then the saddle-point inequality in \eqref{eq:minimaxineq} is equivalent to 
%-^-^-^-^-^-^-^-^-^-^-^-^-^-^-^-^-^-^-^-^-^-^-^-^-^-^-^-^-^-^-
\[
	f_{\text{RP}}\big(\bm{y}^{\text{RP}}(\bm{p}),\bm{p}\big) \leq f_{\text{RP}}\big(\bm{y}^{\text{RP}}(\bm{p}^*),\bm{p}^*\big)\ \  \forall\ \bm{p}\in\mathcal{U}_{\bm{p}}.
\]
%-^-^-^-^-^-^-^-^-^-^-^-^-^-^-^-^-^-^-^-^-^-^-^-^-^-^-^-^-^-^-
\end{lemma}

\begin{proof}
	Let \(\phi(\bm{p}) \triangleq \min_{y\in\mathbb{R}_+^n} f_{\text{RP}}(\bm{y},\bm{p})\). By Lemma \ref{lemma:RP}, \(f_{\text{RP}}(\cdot,\bm{p}):\mathbb{R}_+^n\rightarrow\mathbb{R}\) is strictly convex \(\forall\ \bm{p}\in\mathcal{U}_{\bm{p}}\). Thus, for fixed \(\hat{\bm{p}}\), \(\phi(\hat{\bm{p}})\) is equivalent to the risk parity optimization problem in \eqref{eq:minRPopt} with the unique optimal solution \(\bm{y}^{\text{RP}}(\hat{\bm{p}})\). More generally, this means \(\phi(\bm{p}) = f_{\text{RP}}\big(\bm{y}^{\text{RP}}(\bm{p}), \bm{p}\big)\ \forall\ \bm{p}\). 
	
	By definition, we have that \(\phi(\bm{p}) \leq f_{\text{RP}}(\hat{\bm{y}},\bm{p})\) for any fixed \(\hat{\bm{y}}\in\mathbb{R}_+^n\). Let \((\bm{y}^*, \bm{p}^*)\) be the saddle point of \(f_{\text{RP}}\). Again, imposing the risk parity condition in \eqref{eq:minRPopt}, we have that \((\bm{y}^*, \bm{p}^*) = \big(\bm{y}^{\text{RP}}(\bm{p}^*), \bm{p}^*\big)\). It follows that
%-^-^-^-^-^-^-^-^-^-^-^-^-^-^-^-^-^-^-^-^-^-^-^-^-^-^-^-^-^-^-
\[
	\phi(\bm{p}) \leq f_{\text{RP}}(\bm{y}^*,\bm{p}).
\]
%-^-^-^-^-^-^-^-^-^-^-^-^-^-^-^-^-^-^-^-^-^-^-^-^-^-^-^-^-^-^-

	Moreover, the first part of the saddle-point inequality in \eqref{eq:minimaxineq} establishes that \(f_{\text{RP}}(\bm{y}^*,\bm{p}) \leq f_{\text{RP}}(\bm{y}^*,\bm{p}^*)\ \forall\ \bm{y}\in\mathbb{R}_+^n,\ \bm{p}\in\mathcal{U}_{\bm{p}}\). Combining this with the inequality above, we have that
%-^-^-^-^-^-^-^-^-^-^-^-^-^-^-^-^-^-^-^-^-^-^-^-^-^-^-^-^-^-^-
\[
	f_{\text{RP}}\big(\bm{y}^{\text{RP}}(\bm{p}), \bm{p}\big) = \phi(\bm{p}) \leq f_{\text{RP}}(\bm{y}^*,\bm{p}) \leq f_{\text{RP}}(\bm{y}^*,\bm{p}^*) = f_{\text{RP}}\big(\bm{y}^{\text{RP}}(\bm{p}^*),\bm{p}^*\big)
\]
%-^-^-^-^-^-^-^-^-^-^-^-^-^-^-^-^-^-^-^-^-^-^-^-^-^-^-^-^-^-^-
for all \(\bm{p}\in\mathcal{U}_{\bm{p}}\), as desired. 
\end{proof}

The SCP--PGA algorithm exploits the correspondence between \(\bm{p}\) and \(\bm{y}\) by enforcing the risk parity condition during every iteration \(k\), i.e., we have \(\bm{y}^k = \bm{y}^{\text{RP}}(\bm{p}^k)\). By Lemma \ref{lemma:ineq}, we can use PGA to maximize the curve of minima, \(f_{\text{RP}}\big(\bm{y}^{\text{RP}}(\cdot),\cdot\big):\mathcal{U}_{\bm{p}}\rightarrow\mathbb{R}\), while maintaining the risk parity condition during every iteration. Visually, this can be interpreted as iteratively ascending on the `curve of minima' \cite{bertsekas2003convex} shown in Figure \ref{fig:Saddlepoint}. As the SCP--PGA algorithm converges to \(\bm{p}^*\), we find the optimal DRRP portfolio by projecting \(\bm{y}^* = \bm{y}^{\text{RP}}(\bm{p}^*)\) onto \(\mathcal{X}\). In other words, at convergence, the optimal portfolio is \(\bm{x}^* = \text{\large $\Pi$}_{\mathcal{X}}(\bm{y}^*)\).

\begin{figure}[!htbp]
\centering
    \includegraphics[width=0.9\textwidth]{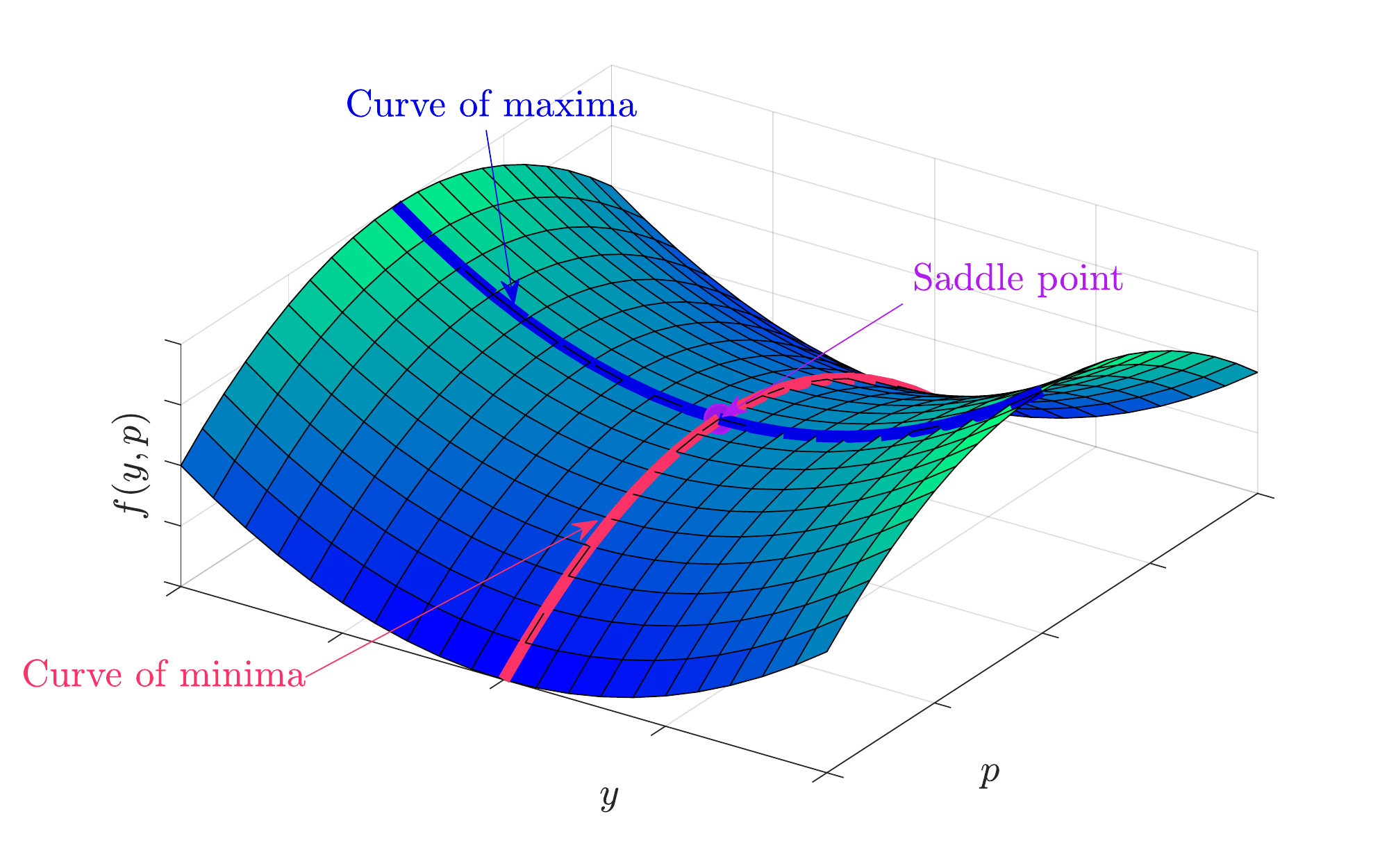}
    \caption{\color{rev} Example of a simple two-dimensional convex--concave function \(f(y,p)\) showing the curve of maxima and the curve of minima. The curve of minima is defined as \(\phi(p) \triangleq \min_{y} f(y,p)\).}
\label{fig:Saddlepoint}
\end{figure}

\begin{theorem}
\label{thm:convergence}
	The function \(f_{\text{RP}}: \mathbb{R}_+^n\times\mathcal{U}_{\bm{p}} \rightarrow \mathbb{R}\) is convex--concave and differentiable over \(\bm{p}\in\mathcal{U}_{\bm{p}}\), and the gradient \(\nabla_{\bm{p}}f_{\text{RP}}(\bm{y},\bm{p})\) is Lipschitz continuous with Lipschitz constant \(L\). Let \((\bm{y}^*,\bm{p}^*)\) be the saddle point of \(f_{\text{RP}}\). 
	
	If we perform a projected gradient ascent in \(\bm{p}\in\mathcal{U}_{\bm{p}}\) with a fixed step size \(\gamma\) from the initial point \(\bm{p}^0\) over \(k\) iterations, then
%-^-^-^-^-^-^-^-^-^-^-^-^-^-^-^-^-^-^-^-^-^-^-^-^-^-^-^-^-^-^-
\[
		0 \leq f_{\text{RP}}\big(\bm{y}^{\text{RP}}(\bm{\bm{p}^*}), \bm{p}^*\big) - \frac{1}{k}\sum_{j=0}^{k-1}f_{\text{RP}}\big(\bm{y}^{\text{RP}}(\bm{p}^j), \bm{p}^j\big) \leq \frac{\| \bm{p}^0 - \bm{p}^* \|_2^2}{2k\gamma} +\frac{\gamma L^2}{2}.
\]
%-^-^-^-^-^-^-^-^-^-^-^-^-^-^-^-^-^-^-^-^-^-^-^-^-^-^-^-^-^-^-
\end{theorem}

\begin{proof}
	This proof is similar to Proposition 3.1 in \cite{nedic2009subgradient}. Recall the projected gradient ascent step introduced in \eqref{eq:projStep} and \eqref{eq:gradStep}. Assume we take a greedy step with \(\eta_{\bm{p}} = 1\) and with a fixed step size \(\gamma\). Let the iteration counter be \(j = 0,1,\dots\). Then, we have that
%-^-^-^-^-^-^-^-^-^-^-^-^-^-^-^-^-^-^-^-^-^-^-^-^-^-^-^-^-^-^-
\begin{align}
	\|\bm{p}^{j+1} - \bm{p}^*\|_2^2 &= \|\bm{p}^j + \eta_{\bm{p}} \bm{h}^j - \bm{p}^*\|_2^2\nonumber \\ 
		&= \big\|\text{\large $\Pi$}_{\mathcal{U}_{\bm{p}}}\big(\bm{p}^{j} + \gamma \nabla_{\bm{p}} f_{\text{RP}}(\bm{y}^j,\bm{p}^j) \big) - \bm{p}^*\big\|_2^2\nonumber\\
	\label{eq:proofineq}	&\leq \| \bm{p}^j + \gamma \nabla_{\bm{p}} f_{\text{RP}}(\bm{y}^j,\bm{p}^j) - \bm{p}^*\|_2^2\\
		&= \| \bm{p}^j - \bm{p}^* \|_2^2 + 2\gamma \nabla_{\bm{p}} f_{\text{RP}}(\bm{y}^j,\bm{p}^j)^\top (\bm{p}^{j} - \bm{p}^*) + \gamma^2 \| \nabla_{\bm{p}} f_{\text{RP}}(\bm{y}^j,\bm{p}^j) \|_2^2.\nonumber
\end{align}
%-^-^-^-^-^-^-^-^-^-^-^-^-^-^-^-^-^-^-^-^-^-^-^-^-^-^-^-^-^-^-
	where the inequality in \eqref{eq:proofineq} arises due to the non-expansive property of the projection operator.\footnote{\color{rev} For this inequality, we have also taken advantage of the fact that \(\bm{p}^* \in\mathcal{U}_{\bm{p}}\), i.e.,  \(\text{\large $\Pi$}_{\mathcal{U}_{\bm{p}}}\big(\bm{p}^*) = \bm{p}^*\).} Since \(f_{\text{RP}}\) is concave for fixed \(\bm{y}^j\), its first-order Taylor expansion between the points \(\bm{p}^*\) and \(\bm{p}^j\) is
%-^-^-^-^-^-^-^-^-^-^-^-^-^-^-^-^-^-^-^-^-^-^-^-^-^-^-^-^-^-^-
\[
	f_{\text{RP}}(\bm{y}^j,\bm{p}^*) \leq f_{\text{RP}}(\bm{y}^j,\bm{p}^j) + \nabla_{\bm{p}} f_{\text{RP}}(\bm{y}^j,\bm{p}^j)^\top (\bm{p}^* - \bm{p}^j)
\]
%-^-^-^-^-^-^-^-^-^-^-^-^-^-^-^-^-^-^-^-^-^-^-^-^-^-^-^-^-^-^-
	or, equivalently,
%-^-^-^-^-^-^-^-^-^-^-^-^-^-^-^-^-^-^-^-^-^-^-^-^-^-^-^-^-^-^-
\[
	\nabla_{\bm{p}}f_{\text{RP}}(\bm{y}^j,\bm{p}^j)^\top (\bm{p}^j - \bm{p}^*) \leq f_{\text{RP}}(\bm{y}^j,\bm{p}^j) -	f_{\text{RP}}(\bm{y}^j,\bm{p}^*)
\]
%-^-^-^-^-^-^-^-^-^-^-^-^-^-^-^-^-^-^-^-^-^-^-^-^-^-^-^-^-^-^-
	which, in turn, allows us to express the initial inequality as
%-^-^-^-^-^-^-^-^-^-^-^-^-^-^-^-^-^-^-^-^-^-^-^-^-^-^-^-^-^-^-
\[
	\|\bm{p}^{j+1} - \bm{p}^*\|_2^2 \leq \| \bm{p}^j - \bm{p}^* \|_2^2 + 2\gamma \big( f_{\text{RP}}(\bm{y}^j,\bm{p}^j) - f_{\text{RP}}(\bm{y}^j,\bm{p}^*) \big) + \gamma^2 \| \nabla_{\bm{p}} f_{\text{RP}}(\bm{y}^j,\bm{p}^j) \|_2^2.
\]
%-^-^-^-^-^-^-^-^-^-^-^-^-^-^-^-^-^-^-^-^-^-^-^-^-^-^-^-^-^-^-
	The objective function \(f_{\text{RP}}\) is twice differentiable in \(\bm{p}\) and has a negative semi-definite Hessian for all \(\bm{p}\in\mathcal{U}_{\bm{p}}\). Thus, it follows that \(\nabla_{\bm{p}}f_{\text{RP}}(\bm{y},\bm{p})\) is Lipschitz continuous with some Lipschitz constant \(L\), which implies \(\|\nabla_{\bm{p}}f_{\text{RP}}(\bm{y}^j,\bm{p}^j)\|_2 \leq L\ \forall\ j\geq 0\). In turn, this means
%-^-^-^-^-^-^-^-^-^-^-^-^-^-^-^-^-^-^-^-^-^-^-^-^-^-^-^-^-^-^-
\[
	\|\bm{p}^{j+1} - \bm{p}^*\|_2^2 \leq \| \bm{p}^j - \bm{p}^* \|_2^2 + 2\gamma \big( f_{\text{RP}}(\bm{y}^j,\bm{p}^j) - f_{\text{RP}}(\bm{y}^j,\bm{p}^*) \big) + \gamma^2 L^2.
\]
%-^-^-^-^-^-^-^-^-^-^-^-^-^-^-^-^-^-^-^-^-^-^-^-^-^-^-^-^-^-^-
	By design of the SCP--PGA algorithm, we impose the condition \(\bm{y}^j \triangleq \bm{y}^{\text{RP}}(\bm{p}^j)\) for every iteration \(j\geq 0\). Thus, by definition, we have that \(\phi(\bm{p}^j) = \min_{y\in\mathbb{R}_+^n} f_{\text{RP}}(\bm{y},\bm{p}^j) = f_{\text{RP}}(\bm{y}^j,\bm{p}^j)\). In turn, by the definition of \(\phi(\bm{p})\), we also have that \(\phi(\bm{p}^*) \leq f_{\text{RP}}(\bm{y}^j,\bm{p}^*)\). Therefore,
%-^-^-^-^-^-^-^-^-^-^-^-^-^-^-^-^-^-^-^-^-^-^-^-^-^-^-^-^-^-^-
\begin{align}
	\frac{1}{2\gamma} \big( \|\bm{p}^{j+1} - \bm{p}^*\|_2^2 - \| \bm{p}^j - \bm{p}^* \|_2^2 \big) - \frac{\gamma L^2}{2} &\leq \big( f_{\text{RP}}(\bm{y}^j,\bm{p}^j) - f_{\text{RP}}(\bm{y}^j,\bm{p}^*) \big)\nonumber\\
		&\leq \phi(\bm{p}^j) - \phi(\bm{p}^*).\nonumber
\end{align}
%-^-^-^-^-^-^-^-^-^-^-^-^-^-^-^-^-^-^-^-^-^-^-^-^-^-^-^-^-^-^-
	Next, we can sum the inequality above over itself for \(j=0,\dots,k-1\), which leads to
%-^-^-^-^-^-^-^-^-^-^-^-^-^-^-^-^-^-^-^-^-^-^-^-^-^-^-^-^-^-^-
\[
	\frac{1}{2\gamma} \big( \|\bm{p}^k - \bm{p}^*\|_2^2 - \| \bm{p}^0 - \bm{p}^* \|_2^2 \big) - \frac{k\gamma L^2}{2} \leq \sum_{j=0}^{k-1}\phi(\bm{p}^j) - k \phi(\bm{p}^*).
\]
%-^-^-^-^-^-^-^-^-^-^-^-^-^-^-^-^-^-^-^-^-^-^-^-^-^-^-^-^-^-^-
	Since \(\|\bm{p}^k - \bm{p}^*\|_2^2 \geq 0\), we can discard this term from the inequality above, restating the inequality it as
%-^-^-^-^-^-^-^-^-^-^-^-^-^-^-^-^-^-^-^-^-^-^-^-^-^-^-^-^-^-^-
\[
	-\frac{\| \bm{p}^0 - \bm{p}^* \|_2^2}{2k\gamma} - \frac{\gamma L^2}{2} \leq \frac{1}{k} \sum_{j=0}^{k-1}\phi(\bm{p}^j) - \phi(\bm{p}^*).
\]
%-^-^-^-^-^-^-^-^-^-^-^-^-^-^-^-^-^-^-^-^-^-^-^-^-^-^-^-^-^-^-
	Finally, by Lemma \ref{lemma:ineq}, we have that \(\phi(\bm{p}^*)\geq \phi(\bm{p})\). Therefore,
%-^-^-^-^-^-^-^-^-^-^-^-^-^-^-^-^-^-^-^-^-^-^-^-^-^-^-^-^-^-^-
\[
	0\leq \phi(\bm{p}^*) - \frac{1}{k} \sum_{j=0}^{k-1}\phi(\bm{p}^j) \leq \frac{\| \bm{p}^0 - \bm{p}^* \|_2^2}{2k\gamma} + \frac{\gamma L^2}{2}.
\]
%-^-^-^-^-^-^-^-^-^-^-^-^-^-^-^-^-^-^-^-^-^-^-^-^-^-^-^-^-^-^-
Since \(\phi(\bm{p}) = f_{\text{RP}}\big(\bm{y}^{\text{RP}}(\bm{\bm{p}}), \bm{p}\big)\ \forall\ \bm{p}\in\mathcal{U}_{\bm{p}}\), we recover the desired result
%-^-^-^-^-^-^-^-^-^-^-^-^-^-^-^-^-^-^-^-^-^-^-^-^-^-^-^-^-^-^-
\[
		0 \leq f_{\text{RP}}\big(\bm{y}^{\text{RP}}(\bm{\bm{p}^*}), \bm{p}^*\big) - \frac{1}{k}\sum_{j=0}^{k-1}f_{\text{RP}}\big(\bm{y}^{\text{RP}}(\bm{p}^j), \bm{p}^j\big) \leq \frac{\| \bm{p}^0 - \bm{p}^* \|_2^2}{2k\gamma} +\frac{\gamma L^2}{2}.
\]
%-^-^-^-^-^-^-^-^-^-^-^-^-^-^-^-^-^-^-^-^-^-^-^-^-^-^-^-^-^-^-
\end{proof}

By Theorem \ref{thm:convergence}, we can see how the SCP--PGA algorithm converges towards the optimal solution \(f_{\text{RP}}\big(\bm{y}^{\text{RP}}(\bm{\bm{p}^*}), \bm{p}^*\big)\) at a rate of \(1/k\) with an error level \(\gamma L^2/2\). Convergence is determined by the number of iterations \(k\), the step size \(\gamma\) and the tightness of the bound given by the Lipschitz constant \(L\). 

Implementing an adaptive diminishing step size \(\gamma_k\) while repeatedly iterating (\(k\rightarrow \infty\)) means the algorithm should (slowly) converge. In our case, the maximization problem is quadratic concave over a compact convex set. Moreover, we retain the Barzilai--Borwein step size \(\gamma_k\) from \eqref{eq:gamma} with the GLL line search method.
}

{\color{rev}Enforcing the risk parity condition \(\bm{y}^k = \bm{y}^{\text{RP}}(\bm{p}^k)\) in every iteration \(k\)} increases the computational cost per iteration when compared against the PGDA algorithm. {\color{rev}However,} we note that convex optimization problems can be efficiently solved by modern optimization algorithms and software packages. Thus, as we will show numerically in Section \ref{sec:Exp}, the additional computational cost per iteration is almost negligible. Moreover, the SCP--PGA algorithm needs less iterations until convergence.

The SCP--PGA algorithm follows the same logic as the PGDA algorithm, except we are only concerned with the ascent step. Since our maximization problem is quadratic concave over a compact convex set, using an appropriate line search can speed up convergence. Specifically, the use of the GLL line search in our algorithm means that, by design, each iteration achieves a sufficient increase in {\color{rev}\(f_{\text{RP}}\big(\bm{y}^{\text{RP}}(\bm{p}),\bm{p}\big)\)} such that we converge to the global maximum. {\color{rev}We can subsequently find the corresponding DRRP optimal portfolio by projecting \(\bm{y}^{\text{RP}}(\bm{p}^*)\) onto \(\mathcal{X}\), or, equivalently, \(\bm{x}^* = \bm{x}^{\text{RP}}(\bm{p}^*)\)}. We complete this subsection by presenting the SCP--PGA algorithm in Algorithm \ref{algo:DRExact}. 

%-^-^-^-^-^-^-^-^-^-^-^-^-^-^-^-^-^-^-^-^-^-^-^-^-^-^-^-^-^-^-
\begin{algorithm}
	\KwIn{Data \(\hat{\bm{\xi}}\in\mathbb{R}^{n\times T}\); {\color{rev}Degree of robustness} \(\omega\in (0,1)\);  Distance measure \{JS, Hellinger, TV\}; Nominal distribution \(\bm{q}\in\mathcal{P}\); Risk parity constant \(\kappa > 0\); Initial step size \(\gamma_0 > 0\); Convergence tolerance \(\varepsilon_0\); Search control parameters \(\beta, \tau\in (0,1)\); GLL parameter \(m\geq 1\)}
	Find the distance limit \(d(\omega,T)\) as shown in either of (\ref{eq:limitJS}--\ref{eq:limitTV})\;
	Initialize the adversarial distribution \(\bm{p}^0 = \bm{q}\)\;
    Initialize the convergence measure: \(\varepsilon >> 1\)\;
	Initialize the counter: \(k = 0\)\;
	\While{\(\varepsilon > \varepsilon_0\)}{  
		\If{\(k \geq 1\)}{
            Update \(\gamma_k\) as shown in \eqref{eq:gamma}\;
        }
		\(\bm{y}^k = \bm{y}^{\text{RP}}(\bm{p}^k)\)\;
		\(\bm{h}^k = \text{\large $\Pi$}_{\mathcal{U}_{\bm{p}}}\Big(\bm{p}^{k} + \gamma_k \nabla_{\bm{p}} f_{\text{RP}}(\bm{y}^k,\bm{p}^k) \Big) - \bm{p}^k\) \label{DRExact:direction}\;      
        \(\eta_{\bm{p}} = 1\)\;
		\(\bar{\bm{p}} = \bm{p}^k + \eta_{\bm{p}} \bm{h}^k\)\;
		\While{\(f_{\text{RP}}(\bm{y}^k,\bar{\bm{p}}) < \min\limits_{j\in\mathcal{J}} f_{\text{RP}}(\bm{y}^{k-j}, \bm{p}^{k-j}) + \beta \eta_{\bm{p}} (\bm{h}^k)^\top \nabla_{\bm{p}} f_{\text{RP}}(\bm{y}^k,\bm{p}^k) \) \label{DRExact:GLL}}{
		\(\eta_{\bm{p}} = \eta_{\bm{p}}\tau \)\;
        \(\bar{\bm{p}} = \bm{p}^k + \eta_{\bm{p}} \bm{h}^k\)\;
        }
        \(\bm{p}^{k+1} = \bar{\bm{p}}\)\;
        
        \If{\(k \geq 1\)}{
            \(\color{rev}\displaystyle \varepsilon = \frac{\| \bm{p}^{k+1} - \bm{p}^k \|_2}{\| \bm{p}^k \|_2}\)\;
        }
        \(k = k+1\)\;
    }
    Find the optimal portfolio: \(\bm{x}^* = \bm{x}^{\text{RP}}(\bm{p}^k)\)\;
\KwOut{Optimal DRRP portfolio \(\bm{x}^*\)}
\caption{SCP--PGA for DRRP portfolio optimization}
\label{algo:DRExact}
\end{algorithm}
%-^-^-^-^-^-^-^-^-^-^-^-^-^-^-^-^-^-^-^-^-^-^-^-^-^-^-^-^-^-^-

%%%%%%%%%%%%%%%%%%%%%%%%%%%%%%%%%%%%%%%%%%%%%%%%%%%%%%%%%%%%%%%%%%%%%%%%%%%%%%%%%%%%%%%%%%%%%%
% 		NUMERICAL EXPERIMENTS
%%%%%%%%%%%%%%%%%%%%%%%%%%%%%%%%%%%%%%%%%%%%%%%%%%%%%%%%%%%%%%%%%%%%%%%%%%%%%%%%%%%%%%%%%%%%%%
\section{Numerical experiments} \label{sec:Exp}

This section consists of three separate experiments. {\color{rev}The first experiment evaluates the computational performance of the SCP--PGA algorithm (Algorithm \ref{algo:DRExact}) against both the robust counterpart in \eqref{eq:singleConvex} and the PGDA algorithm (Algorithm \ref{algo:DRORP}). This numerical performance experiment is conducted using synthetically-generated data, allowing us to test increasingly larger datasets.} 

The second experiment assesses the in-sample performance of the DRRP portfolio in a financial context and uses historical data. The third experiment tests the out-of-sample financial performance of the DRRP portfolio. The second and third experiments share the same data. The data consist of historical observations ranging from the start of 1998 until the end of 2016 for 30 industry portfolios. These industry portfolios serve as our financial assets and are akin to many popular exchange traded funds. The data were obtained from Kenneth R. French's data library \cite{French2020Data}. Table \ref{table:assets} lists the 30 industry portfolios.

\begin{table}[!ht]
\footnotesize
\caption{List of assets}
\centering
\begin{tabular}{l l l l} 
\hline 
\rule{0pt}{3ex}Food Products & Tobacco & Beer and Liquor & Recreation\\[1.5ex]
                Household Products & Apparel & Healthcare & Chemicals\\[1.5ex]
                Fabricated Products & Construction & Steel Works & Electrical Equip.\\[1.5ex]
                Aircraft, Ships, Rail Equip. & Mining & Coal & Oil and Gas\\[1.5ex]
                Communication & Services & Business Equip. & Paper\\[1.5ex]
                Restaurants and Hotels & Wholesale & Retail & Financials\\[1.5ex]
                Printing & Textiles & Automobiles & Utilities\\[1.5ex]
                Transportation & Other\\[1.5ex]
\hline
\end{tabular}
\label{table:assets}
\end{table}

All experiments were conducted on an Apple MacBook Pro computer (2.8 GHz Intel Core i7, 16 GB 2133 MHz DDR3 RAM) running macOS `Catalina'. The script was written using the Julia programming language (version 1.4.0) with the modelling language `JuMP' \cite{DunningHuchetteLubin2017} and with IPOPT (version 3.12.6) as the optimization solver.

%-------------------------------------------------------------------------------------------
%       Numerical performance and tractability
%-------------------------------------------------------------------------------------------
\subsection{Numerical performance and tractability}\label{sec:NumPerf}

{\color{rev}The numerical performance experiment is conducted in three parts. The first part compares the SCP--PGA algorithm against the robust counterpart. The second part compares the SCP--PGA algorithm against the PGDA algorithm. Finally, the third part provides an extended evaluation of the numerical performance and tractability of the SCP--PGA algorithm.

The first part of this experiment, which compares the SCP--PGA algorithm against the robust counterpart in \eqref{eq:singleConvex}, is conducted as follows. We randomly generate synthetic datasets with \(n=\) 100, 200, 400 assets and \(T= \) 100, 200, 400 scenarios, meaning there are a total of nine different datasets. For each dataset, we find the corresponding DRRP portfolio using both the SCP--PGA algorithm and the robust counterpart. By design, both optimal portfolios should be identical. Thus, we evaluate performance by comparing the runtime of both methods. The tests are repeated twice using two different degrees of robustness, \(\omega=\) 0.25, 0.5. For this part of the experiment, we only use the Hellinger distance to define the ambiguity set \(\mathcal{U}_{\bm{p}}\).}

The remainder of the user-defined parameters are the following. The risk parity constant is set to \(\kappa = 1\), while the convergence tolerance is set to {\color{rev}\(\varepsilon_0 = 10^{-4}\)}. As recommended in \cite{birgin2000nonmonotone}, we set \(m=10\). Moreover, we set the search parameters to \(\beta=10^{-6}\) and \(\tau=0.9\). The initial ascent step size is \(\gamma_0=0.1\). {\color{rev}The results are presented in Table \ref{table:RC}.}

\begin{table}[!ht]
\color{rev}
\footnotesize
\caption{Comparison of numerical performance between the SCP--PGA algorithm (shown as `SCP') from Algorithm \ref{algo:DRExact} and the robust counterpart (RC) for DRRP portfolios based on the Hellinger distance.}
\centering
\begin{tabular}{@{\extracolsep{-1.25pt}}l rr l rr l rr}
\hline 
\rule{0pt}{2.75ex} & \multicolumn{2}{c}{\(n=100\)} && \multicolumn{2}{c}{\(n=200\)} && \multicolumn{2}{c}{\(n=400\)}\\[0.5ex]\cline{2-3} \cline{5-6} \cline{8-9}
\rule{0pt}{2.75ex}& SCP & RC && SCP & RC && SCP & RC\\[0.5ex]
\hline
	\rule{0pt}{2.75ex}\(\omega=0.2\)  & \\[0.5ex]
\hline
	\rule{0pt}{3.25ex}\(\bm{T=100}\) & \\[0.5ex]
	Runtime (s) 	& 0.369 & 11.8 && 0.966 & 25.9 && 3.83 & 207\\[1ex]
\(\| \bm{x}^{(i)} - \bm{x}^{\text{Nom}}\|_2\) & 0.0104 & 0.0104 && 0.0142 & 0.0142 && 0.1 & 0.1\\[1ex]
\(\| \bm{x}^{\text{SCP}} - \bm{x}^{\text{RC}}\|_2\) & 3.64e-6 & && 1.48e-5 & && 9.18e-5\\[1.5ex]
	
	\rule{0pt}{3.25ex}\(\bm{T=200}\) & \\[0.5ex]
	Runtime (s) 	& 0.852 & 19.0 && 2.02 & 109 && 6.84 & 578\\[1ex]
\(\| \bm{x}^{(i)} - \bm{x}^{\text{Nom}}\|_2\) & 0.0187 & 0.0187 && 8.52e-3 & 8.52e-3 && 0.0278 & 0.0278\\[1ex]
\(\| \bm{x}^{\text{SCP}} - \bm{x}^{\text{RC}}\|_2\) & 1.5e-5 & && 8.29e-7 & && 1.17e-5\\[1.5ex]

	\rule{0pt}{3.25ex}\(\bm{T=400}\) & \\[0.5ex]
	Runtime (s) 	& 1.56 & 55.0 && 3.16 & 444 && 12.0 & 1,282\\[1ex]
\(\| \bm{x}^{(i)} - \bm{x}^{\text{Nom}}\|_2\) & 0.0203 & 0.0203 && 0.0404 & 0.0404 && 6.3e-3 & 6.3e-3\\[1ex]
\(\| \bm{x}^{\text{SCP}} - \bm{x}^{\text{RC}}\|_2\) & 3.45e-5 & && 5.54e-6 & && 8.72e-7\\[1.5ex]
	
\hline
	\rule{0pt}{2.75ex}\(\omega=0.4\)  & \\[0.5ex]
\hline
	\rule{0pt}{3.25ex}\(\bm{T=100}\) & \\[0.5ex]
	Runtime (s) 	& 0.866 & 26.9 && 1.47 & 151 && 5.4 & 229\\[1ex]
\(\| \bm{x}^{(i)} - \bm{x}^{\text{Nom}}\|_2\) & 0.0759 & 0.0759 && 0.0255 & 0.0255 && 0.0317 & 0.0317\\[1ex]
\(\| \bm{x}^{\text{SCP}} - \bm{x}^{\text{RC}}\|_2\) & 2.92e-5 & && 2.87e-6 & && 1.7e-5\\[1ex]
	
	\rule{0pt}{3.25ex}\(\bm{T=200}\) & \\[0.5ex]
	Runtime (s) 	& 1.28 & 21.8 && 3.47 & 107 && 14.3 & 217\\[1ex]
\(\| \bm{x}^{(i)} - \bm{x}^{\text{Nom}}\|_2\) & 0.0265 & 0.0265 && 0.0393 & 0.0393 && 0.0907 & 0.0907\\[1ex]
\(\| \bm{x}^{\text{SCP}} - \bm{x}^{\text{RC}}\|_2\) & 5.2e-5 & && 2.03e-5 & && 3.99e-5\\[1ex]

	\rule{0pt}{3.25ex}\(\bm{T=400}\) & \\[0.5ex]
	Runtime (s) 	& 3.75 & 19.5 && 6.16 & 236 && 20.3 & 556\\[1ex]
\(\| \bm{x}^{(i)} - \bm{x}^{\text{Nom}}\|_2\) & 0.0252 & 0.0251 && 0.0425 & 0.0424 && 0.0161 & 0.0161\\[1ex]
\(\| \bm{x}^{\text{SCP}} - \bm{x}^{\text{RC}}\|_2\) & 1.71e-4 & && 4.2e-4 & && 4.4e-5\\[1ex]
\hline
\end{tabular}
\label{table:RC}
\end{table}

{\color{rev}Table \ref{table:RC} presents the following elements for evaluation. The first element is the runtime (in seconds) of the two methods. The runtime is the sole performance indicator in this table and allows us to see the performance advantage of the SCP--PGA algorithm over the robust counterpart. The remaining elements in Table \ref{table:RC} are there to verify that the resulting portfolios are, in fact, virtually identical. The element denoted as \(\| \bm{x}^{(i)} - \bm{x}^{\text{Nom}}\|_2\) measures the \(\ell_2\)-norm distance between a DRRP portfolio (either the SCP--PGA or robust counterpart) and the nominal (non-robust) risk parity portfolio. Since the DRRP asset allocations are almost identical, we expect this measure to also be almost identical (i.e., the `distance' between the two DRRP portfolios and the nominal portfolio should be almost the same). Finally, the last element, denoted as \(\| \bm{x}^{\text{SCP}} - \bm{x}^{\text{RC}}\|_2\), measures the \(\ell_2\)-norm distance between the two DRRP portfolios. Once again, since the asset allocations are very similar, this distance should be close to zero.

Inspecting the distance metrics \(\| \bm{x}^{(i)} - \bm{x}^{\text{Nom}}\|_2\) and \(\| \bm{x}^{\text{SCP}} - \bm{x}^{\text{RC}}\|_2\) in Table \ref{table:RC}, we can see that the DRRP portfolios behave as intended, i.e., both methods find the same optimal portfolio. Looking at the runtime results, we can see that the SCP--PGA algorithm outperforms the robust counterpart in every single test. In fact, the runtime results suggest that the SCP--PGA algorithm is at least five times faster than the robust counterpart (see the trial with \(\omega=0.4\), \(n=100\) and \(T=400\)) and up to 100 times faster (see the trial with \(\omega=0.2\), \(n=400\) and \(T=400\)). These results highlight a significant numerical advantage of the SCP--PGA algorithm over the robust counterpart.}

The {\color{rev}second} part of the numerical performance experiment compares the SCP--PGA algorithm to the PGDA algorithm. {\color{rev}The experiment is conducted as follows. We randomly generate synthetic datasets with \(n=\) 50, 200 assets and\(T=\) 1,000, 5,000 scenarios, meaning there are a total of four different datasets. The largest dataset, with \(n=200\) and \(T=5,000\), simulates the conditions to create a portfolio with 200 constituent assets using approximately 20 years worth of daily observations. For each dataset, we find the corresponding DRRP portfolio using both the PGDA algorithm and the SCP--PGA algorithm. We optimize the portfolios using three different degrees of robustness, \(\omega=\) 0.15, 0.3, 0.45, and using the JS, Hellinger and TV distances.}

{\color{rev}The remainder of the user-defined parameters are the same as before.} In addition, we set the following values for the PGDA algorithm: \(\alpha_0=30\), {\color{rev}\(y_i^0 = 10\)} for \(i=1, \dots, n\). {\color{rev}The results are presented in Table \ref{table:algoComp}.}

\begin{table}[!ht]
\footnotesize
\caption{Comparison of numerical performance between the PGDA {\color{rev}from Algorithm \ref{algo:DRORP} and the SCP--PGA (shown as `SCP') from Algorithm \ref{algo:DRExact}}. The maximum number of iterations is limited to 1,000, after which the algorithms terminate.}
\centering
\begin{tabular}{@{\extracolsep{-6.5pt}}l rr l rr l rr l rr l rr l rr}
\hline 
\rule{0pt}{2.75ex} & \multicolumn{8}{c}{\(n=50\)} && \multicolumn{8}{c}{\(n=200\)}\\[0.5ex]\cline{2-9} \cline{11-18} 
\rule{0pt}{2.75ex} & \multicolumn{2}{c}{JS} && \multicolumn{2}{c}{Hellinger} && \multicolumn{2}{c}{TV} && \multicolumn{2}{c}{JS} && \multicolumn{2}{c}{Hellinger} && \multicolumn{2}{c}{TV}\\[0.5ex]\cline{2-3} \cline{5-6} \cline{8-9}\cline{11-12}\cline{14-15}\cline{17-18}
\rule{0pt}{2.75ex}& PGDA & SCP && PGDA & SCP && PGDA & SCP && PGDA & SCP && PGDA & SCP && PGDA & SCP\\[0.5ex]
\hline
	\rule{0pt}{2.75ex}\(\omega=0.15\)  & \\[0.5ex]
\hline
	\rule{0pt}{3.25ex}\(\bm{T=1,000}\) & \\[0.5ex]
	Runtime (s) 		& {\color{rev}2.74} & {\color{rev}0.91} && {\color{rev}2.98} & {\color{rev}1.73} && {\color{rev}317} & {\color{rev}5.21}&& {\color{rev}20.5} & {\color{rev}4.54} && {\color{rev}47.9} & {\color{rev}4.87} && {\color{rev}653} & {\color{rev}11.2}\\[0.25ex]
	Iterations 		& {\color{rev}10} & {\color{rev}7} &\ \ & {\color{rev}14} & {\color{rev}8} &\ \ &  {\color{rev}1,000} & {\color{rev}15} &\ \ & {\color{rev}40} & {\color{rev}8} &\ \ & {\color{rev}81} & {\color{rev}7} &\ \ & {\color{rev}1,000} & {\color{rev}16}\\[0.25ex]
	Var. (\(\times 10^4\)) 	& {\color{rev}6.61} & 6.61 && {\color{rev}7.13} & 7.13 && {\color{rev}9.38} & 9.38 && {\color{rev}6.17} & 6.17 && {\color{rev}6.70} & 6.70 && {\color{rev}9.13} & 9.13\\[0.25ex]

	\rule{0pt}{3.25ex}\(\bm{T=5,000}\) & \\[0.5ex]
	Runtime (s) 		& {\color{rev}12.2} & {\color{rev}9.10} && {\color{rev}22.8} & {\color{rev}18.9} && {\color{rev}51.57} & {\color{rev}26.0} && {\color{rev}191} & {\color{rev}24.7} && {\color{rev}159} & {\color{rev}34.4} && {\color{rev}4,385} & {\color{rev}71.5}\\[0.25ex]
	Iterations 		& {\color{rev}12} & {\color{rev}8} && {\color{rev}11} & {\color{rev}9} && {\color{rev}25} & {\color{rev}11} && {\color{rev}76} & {\color{rev}10} && {\color{rev}49} & {\color{rev}10} && {\color{rev}1,000} & {\color{rev}19}\\[0.25ex]
   	Var. (\(\times 10^4\)) 	& 5.64 & 5.64 && {\color{rev}6.03} & 6.03 && {\color{rev}9.89} & 9.89 && {\color{rev}5.23} & 5.23 && {\color{rev}5.68} & 5.68 && {\color{rev}8.16} & 8.17\\[1ex]

\hline
	\rule{0pt}{2.75ex}\(\omega=0.3\)  & \\[0.5ex]
\hline
	\rule{0pt}{3.25ex}\(\bm{T=1,000}\) & \\[0.5ex]
	Runtime (s) 		& {\color{rev}1.37} & {\color{rev}1.68} && {\color{rev}4.18} & {\color{rev}4.82} && {\color{rev}74.5} & {\color{rev}8.28} && {\color{rev}37.5} & {\color{rev}9.22} && {\color{rev}15.5} & {\color{rev}10.2} && {\color{rev}214} & {\color{rev}15.5}\\[0.25ex]
	Iterations 		& {\color{rev}14} & {\color{rev}16} && {\color{rev}16} & {\color{rev}17} && {\color{rev}259} & {\color{rev}25} && {\color{rev}91} & {\color{rev}17} && {\color{rev}26} & {\color{rev}17} && {\color{rev}360} & {\color{rev}24}\\[0.25ex]
	Var. (\(\times 10^4\)) 	& {\color{rev}10.8} & 10.8 && {\color{rev}11.9} & 11.9 && {\color{rev}14.4} & 14.4 && {\color{rev}10.5} & 10.5 && {\color{rev}11.7} & 11.7 && {\color{rev}14.2} & 14.2\\[0.25ex]

	\rule{0pt}{3.25ex}\(\bm{T=5,000}\) & \\[0.5ex]
	Runtime (s) 		& {\color{rev}12.4} & {\color{rev}12.2} && {\color{rev}25.4} & {\color{rev}22.8} && {\color{rev}45.2} & {\color{rev}33.6} && {\color{rev}259} & {\color{rev}35.3} && {\color{rev}491} & {\color{rev}49.2} && {\color{rev}870} & {\color{rev}101}\\[0.25ex]
	Iterations 		& {\color{rev}14} & {\color{rev}14} && {\color{rev}15} & {\color{rev}14} && {\color{rev}24} & {\color{rev}18} && {\color{rev}108} & {\color{rev}14} && {\color{rev}152} & {\color{rev}15} && {\color{rev}216} & {\color{rev}28}\\[0.25ex]
	Var. (\(\times 10^4\)) 	& {\color{rev}10.8} & 10.8 && 11.6 & 11.6 && {\color{rev}16.2} & 16.2 && {\color{rev}8.91} & 9.06 && {\color{rev}9.95} & 9.98 && {\color{rev}12.9} & 12.9\\[1ex]

\hline
	\rule{0pt}{2.75ex}\(\omega=0.45\)  & \\[0.5ex]
\hline
		\rule{0pt}{3.25ex}\(\bm{T=1,000}\) & \\[0.5ex]
 	Runtime (s) 		& {\color{rev}2.53} & {\color{rev}1.93} && {\color{rev}5.89} & {\color{rev}4.47} && {\color{rev}35.2} & {\color{rev}7.70} && {\color{rev}28.4} & {\color{rev}9.43} && {\color{rev}15.6} & {\color{rev}14.5} && {\color{rev}125} & {\color{rev}15.5}\\[0.25ex]
	Iterations 		& {\color{rev}27} & {\color{rev}17} && {\color{rev}26} & {\color{rev}19} && {\color{rev}119} & {\color{rev}27} && {\color{rev}74} & {\color{rev}22} && {\color{rev}29} & {\color{rev}26} && {\color{rev}217} & {\color{rev}24}\\[0.25ex]
 	Var. (\(\times 10^4\)) 	& 16.1 & 16.1 && {\color{rev}17.8} & 17.8 && {\color{rev}19.3} & 19.3 && {\color{rev}16.1} & 16.1 && {\color{rev}17.8} & 17.9 && {\color{rev}19.1} & 19.1\\[0.25ex]

	\rule{0pt}{3.25ex}\(\bm{T=5,000}\) & \\[0.5ex]
	Runtime (s) 		& {\color{rev}15.5} & {\color{rev}13.2} && {\color{rev}42.6} & {\color{rev}29.2} && {\color{rev}166} & {\color{rev}49.9} && {\color{rev}315} & {\color{rev}69.1} && {\color{rev}277} & {\color{rev}110} && {\color{rev}3,881} & {\color{rev}138}\\[0.25ex]
	Iterations 		& {\color{rev}21} & {\color{rev}16} && {\color{rev}26} & {\color{rev}17} && {\color{rev}93} & {\color{rev}27} && {\color{rev}135} & {\color{rev}29} && {\color{rev}84} & {\color{rev}31} && {\color{rev}1,000} & {\color{rev}38}\\[0.25ex]
 	Var. (\(\times 10^4\)) 	& 17.9 & 17.9 && {\color{rev}19.3} & 19.3 && {\color{rev}22.5} & 22.5 && {\color{rev}14.3} & 14.3 && {\color{rev}15.7} & 15.7 && {\color{rev}17.5} & 17.6\\[1ex]
\hline
\end{tabular}
\label{table:algoComp}
\end{table}

{\color{rev}Table \ref{table:algoComp} presents the following elements for evaluation. The first and most important element is the runtime (in seconds), which indicates the performance advantage of the SCP--PGA algorithm over the PGDA algorithm. The second element is the number of iterations until convergence. Finally, since part of the output of these two algorithms is the optimal adversarial probability distribution \(\bm{p}^*\), we can evaluate whether they achieve the same optimal DRRP portfolio by comparing the corresponding worst-case estimates of the portfolio variance. If the algorithms have the same portfolio variance, then they both converged to the same optimal DRRP portfolio. However, if the portfolio variances differ, a lower variance indicates convergence to a sub-optimal solution. Note that the algorithms were run with a limit of 1,000 iterations, after which the algorithms are forced to terminate.}

The results in Table \ref{table:algoComp} show that the SCP--PGA algorithm {\color{rev}attains} an equal or higher portfolio variance than the PGDA algorithm in every single instance, meaning the PGDA algorithm sometimes converges to a suboptimal solution. {\color{rev}Moreover, the SCP--PGA algorithm has a faster runtime in every single instance except when \(n=50\), \(T=1,000\), \(\omega=0.3\) with the JS and Hellinger distances (in these two instances the runtimes are almost identical). Crucially, the runtime of the SCP--PGA algorithm scales much better as the datasets increase in size, i.e., the runtime of the SCP--PGA algorithm is comparatively faster than the PGDA algorithm for high-dimensional problems.}

{\color{rev}The third part of the numerical performance experiment focuses solely on the SCP--PGA algorithm. The experiment is conducted as follows. We randomly generate large synthetic datasets with \(n=\) 200, 500, 1,000 assets and\(T=\) 1,000, 5,000, 7,500 scenarios, meaning there are a total of nine different datasets. The largest dataset, with \(n=1,000\) and \(T=7,500\), simulates the conditions to create a portfolio with 1,000 constituent assets using approximately 30 years worth of daily observations. For each dataset, we find the corresponding DRRP portfolio using the SCP--PGA algorithm. We optimize the portfolios using three different degrees of robustness, \(\omega=\) 0.15, 0.3, 0.45, and using the JS, Hellinger and TV distances. The remainder of the user-defined parameters are the same as before. The results are presented in Table \ref{table:performance}.}

\begin{table}[!ht]
\footnotesize
\color{rev}
\caption{Numerical performance of the SCP--PGA algorithm}
\centering
\begin{tabular}{@{\extracolsep{-4pt}}l rrrr l rrrr l rrrr}
\hline 
\rule{0pt}{2.5ex}   & \multicolumn{4}{c}{\(n=200\)} && \multicolumn{4}{c}{\(n=500\)} && \multicolumn{4}{c}{\(n=1,000\)}\\[0.25ex]\cline{2-5} \cline{7-10} \cline{12-15}
\rule{0pt}{3ex}& Nom. & JS & H & TV && Nom. & JS & H & TV && Nom. & JS & H & TV\\[0.5ex]
\hline
	\rule{0pt}{2.75ex}\(\omega=0.15\)  & \\[0.5ex]
\hline
	\rule{0pt}{3ex}\(\bm{T=1,000}\) & \\[0.5ex]
 	Runtime (s) 	& - & 4.74 & 5.32 & 11.8 && - & 21.9 & 22.7 & 44.1 && - & 56.7 & 57.6 & 118\\[0.25ex]
 	Iterations 	& - & 9 & 9 & 18 && - & 8 & 8 & 15 && - & 8 & 8 & 16\\[0.25ex]
 	Time/Iter. (s) & - & 0.53 & 0.59 & 0.66 && - & 0.13 & 0.11 & 0.1 && - & 0.56 & 0.57 & 0.61\\[0.25ex]
 	Var. (\(\times 10^4\))	& 3.12 & 5.08 & 5.47 & 7.00 && 3.33 & 5.35 & 5.76 & 7.04 && 3.26 & 5.66 & 6.15 & 8.58\\[0.25ex]

	\rule{0pt}{3ex}\(\bm{T=5,000}\) & \\[0.5ex]
 	Runtime (s) 	& - & 26.4 & 36.9 & 47.7 && - & 101 & 120 & 138 && - & 214 & 239 & 473\\[0.25ex]
 	Iterations 	& - & 9 & 10 & 12 && - & 9 & 10 & 12 && - & 9 & 10 & 19\\[0.25ex]
 	Time/Iter. (s) & - & 2.93 & 3.69 & 3.98 && - & 11.2 & 12.0 & 11.5 && - & 23.7 & 24.0 & 24.9\\[0.25ex]
 	Var. (\(\times 10^4\)) 	& 3.45 & 5.94 & 6.37 & 9.66 && 3.44 & 6.01 & 6.49 & 9.9 && 3.30 & 5.54 & 5.98 & 8.28\\[0.25ex]

	\rule{0pt}{3ex}\(\bm{T=7,500}\) & \\[0.5ex]
 	Runtime (s) 	& - & 57.8 & 61.2 & 103 && - & 133 & 187 & 290 && - & 324 & 410 & 707\\[0.25ex]
 	Iterations 	& - & 13 & 11 & 17 && - & 9 & 11 & 16 && - & 9 & 11 & 17\\[0.25ex]
 	Time/Iter. (s) & - & 4.45 & 5.56 & 6.09 && - & 14.8 & 17.0 & 18.2 && - & 36.0 & 37.3 & 41.6\\[0.25ex]
 	Var. (\(\times 10^4\)) 	& 2.76 & 4.64 & 5.04 & 7.34 && 3.25 & 5.69 & 6.17 & 9.37 && 3.15 & 5.37 & 5.80 & 7.83\\[0.25ex]
\hline
	\rule{0pt}{2.75ex}\(\omega=0.3\)  & \\[0.5ex]
\hline
	\rule{0pt}{3ex}\(\bm{T=1,000}\) & \\[0.5ex]
 	Runtime (s) 	& - & 8.60 & 10.2 & 12.9 && - & 41.3 & 44.4 & 74.7 && - & 103 & 100 & 146\\[0.25ex]
 	Iterations 	& - & 18 & 17 & 20 && - & 18 & 18 & 29 && - & 15 & 14 & 20\\[0.25ex]
 	Time/Iter. (s) & - & 0.48 & 0.6 & 0.65 && - & 2.3 & 2.47 & 2.58 && - & 6.83 & 7.11 & 7.32\\[0.25ex]
 	Var. (\(\times 10^4\)) 	& 3.12 & 8.03 & 8.89 & 10.17 && 3.33 & 8.31 & 9.22 & 10.44 && 3.26 & 9.77 & 10.82 & 13.49\\[0.25ex]

	\rule{0pt}{3ex}\(\bm{T=5,000}\) & \\[0.5ex]
 	Runtime (s) 	& - & 48.3 & 48.6 & 78.2 && - & 167 & 202 & 186 && - & 505 & 556 & 675\\[0.25ex]
 	Iterations 	& -  & 17 & 13 & 20 && - & 16 & 18 & 17 && -  & 21 & 22 & 27\\[0.25ex]
 	Time/Iter. (s) & - & 2.84 & 3.74 & 3.91 && - & 10.4 & 11.2 & 11.0 && - & 24.1 & 25.3 & 25.0\\[0.25ex]
 	Var. (\(\times 10^4\)) 	& 3.45 & 10.6 & 11.6 & 15.4 && 3.44 & 10.9 & 11.9 & 15.8 && 3.30 & 9.33 & 10.3 & 12.9\\[0.25ex]

	\rule{0pt}{3ex}\(\bm{T=7,500}\) & \\[0.5ex]
 	Runtime (s) 	& - & 73.1 & 122 & 176.4 && - & 289 & 353 & 458 && - & 413 & 749 & 1,146\\[0.25ex]
 	Iterations 	& - & 17 & 21 & 28 && - & 18 & 20 & 26 && - & 12 & 20 & 29\\[0.25ex]
 	Time/Iter. (s) & - & 4.30 & 5.81 & 6.3 && - & 16.1 & 17.7 & 17.6 && - & 34.4 & 37.5 & 39.5\\[0.25ex]
 	Var. (\(\times 10^4\)) 	& 2.76 & 8.10 & 8.86 & 11.7 && 3.25 & 10.4 & 11.4 & 15.3 && 3.15 & 8.67 & 9.65 & 12.2\\[0.25ex]
 	 	
 \hline
	\rule{0pt}{2.75ex}\(\omega=0.45\)  & \\[0.5ex]
\hline
	\rule{0pt}{3ex}\(\bm{T=1,000}\) & \\[0.5ex]
 	Runtime (s) 	& - & 13.2 & 15.4 & 14.8 && - & 58.6 & 80.5 & 85.5 && - & 102 & 112 & 105\\[0.25ex]
 	Iterations 	& - & 24 & 24 & 21 && - & 26 & 32 & 34 && - & 15 & 16 & 15\\[0.25ex]
 	Time/Iter. (s) & - & 0.55 & 0.64 & 0.70 && - & 2.25 & 2.52 & 2.51 && - & 6.82 & 6.99 & 6.99\\[0.25ex]
 	Var. (\(\times 10^4\)) & 3.12 & 11.6 & 12.9 & 13.5 && 3.33 & 11.9 & 13.3 & 13.7 && 3.26 & 15.0 & 16.6 & 18.1\\[0.25ex]
 	
	\rule{0pt}{3ex}\(\bm{T=5,000}\) & \\[0.5ex]
 	Runtime (s) 	& - & 73.6 & 74.0 & 112 && - & 211 & 228 & 362 && - & 403 & 547 & 1,064\\[0.25ex]
 	Iterations 	& - & 25 & 19 & 28 && - & 19 & 18 & 30 && - & 17 & 23 & 40\\[0.25ex]
 	Time/Iter. (s) & - & 2.94 & 3.90 & 3.99 && - & 11.1 & 12.7 & 12.1 && - & 23.7 & 23.8 & 26.6\\[0.25ex]
 	Var. (\(\times 10^4\)) 	& 3.45 & 16.9 & 18.5 & 21.0 && 3.44 & 17.4 & 19.1 & 21.5 && 3.30 & 14.3 & 15.8 & 17.4\\[0.25ex]

	\rule{0pt}{3ex}\(\bm{T=7,500}\) & \\[0.5ex]
 	Runtime (s) 	& - & 99.0 & 180 & 196 && - & 377 & 487 & 798 && - & 509 & 841 & 1,817\\[0.25ex]
 	Iterations 	& - & 23 & 30 & 31 && - & 24 & 27 & 42 && - & 15 & 22 & 44\\[0.25ex]
 	Time/Iter. (s) & - & 4.30 & 5.99 & 6.32 && - & 15.7 & 18.0 & 19.0 && - & 33.9 & 38.2 & 41.3\\[0.25ex]
 	Var. (\(\times 10^4\)) & 2.76 & 12.8 & 14.0 & 15.9 && 3.25 & 16.9 & 18.6 & 21.2 && 3.15 & 13.5 & 14.9 & 16.5\\[0.25ex]	
\hline
\end{tabular}
\label{table:performance}
\end{table}

{\color{rev}Table \ref{table:performance} summarizes the numerical performance of the SCP--PGA algorithm.} As before, the SCP--PGA algorithm is evaluated based on its runtime, the number of iterations until convergence, and the {\color{rev}resulting} portfolio variance. We also include the runtime per iteration. Finally, the results also show the variance of the nominal risk parity portfolio for the same dataset. The nominal portfolio variance serves as a benchmark. When looking at the DRRP portfolio variance, we must keep in mind that this corresponds to the worst-case estimate of the variance as defined by the ambiguity set \(\mathcal{U}_{\bm{p}}\). Therefore, we expect the DRRP portfolio variance to be larger than the nominal.

The results in Table \ref{table:performance} indicate that, overall, the SCP--PGA algorithm converges within reasonable time, even for the largest dataset tested. We note that the largest dataset, with {\color{rev}\(n=1,000\) and \(T=7,500\)}, exaggerates the number of scenarios \(T\) that we would normally consider for parameter estimation in a conventional environment.\footnote{This may exclude high frequency trading environments.} Most financial data service providers tend to use anywhere from 10 days to five years when calculating risk metrics such as the portfolio variance or the CAPM `beta' \cite{sharpe1964capital}, and rely on daily, weekly or monthly scenarios for these calculations. 

If we inspect the portfolio variance, we can see that the JS divergence is more restrictive than either the Hellinger distance or the TV distance. In other words, for the same {\color{rev}degree of robustness}, the JS divergence provides a smaller feasible region in our variance maximization step, leading to portfolios with lower variances. Conversely, the TV distance is the most permissive, consistently having the highest variance for all trials. This suggests that, although the three distance measures have been scaled proportionally, their intrinsic differences suggest some are fundamentally more permissive than others from a portfolio variance perspective.

As shown by the different runtimes in Table \ref{table:performance}, the SCP--PGA algorithm converged the fastest when we used the JS divergence to construct the ambiguity set ({\color{rev}with one exception where the Hellinger distance was marginally faster}). The runtime per iteration is relatively similar for all three distance measures. Thus, this suggests that having a faster convergence rate is mostly dependent on the number of iterations required until convergence. 

{\color{rev}Finally, Figure \ref{fig:convergence} presents the convergence plots corresponding to the test where \(\omega=0.3\) and \(T=5,000\). The plots indicate that the first few iterations are responsible for the majority of the improvement in our objective value. Thus, we may be able to improve the runtime through early stopping of the SCP--PGA algorithm. However, the analysis of an early stopping criterion is outside the scope of this manuscript. Some of these plots also shows the non-monotonic behaviour of the Barzilai--Borwein step size with GLL line search (e.g., see the second plot where the distance is JS and \(n=500\)).}

\begin{figure}[!htbp]
\centering
    \includegraphics[width=0.95\textwidth]{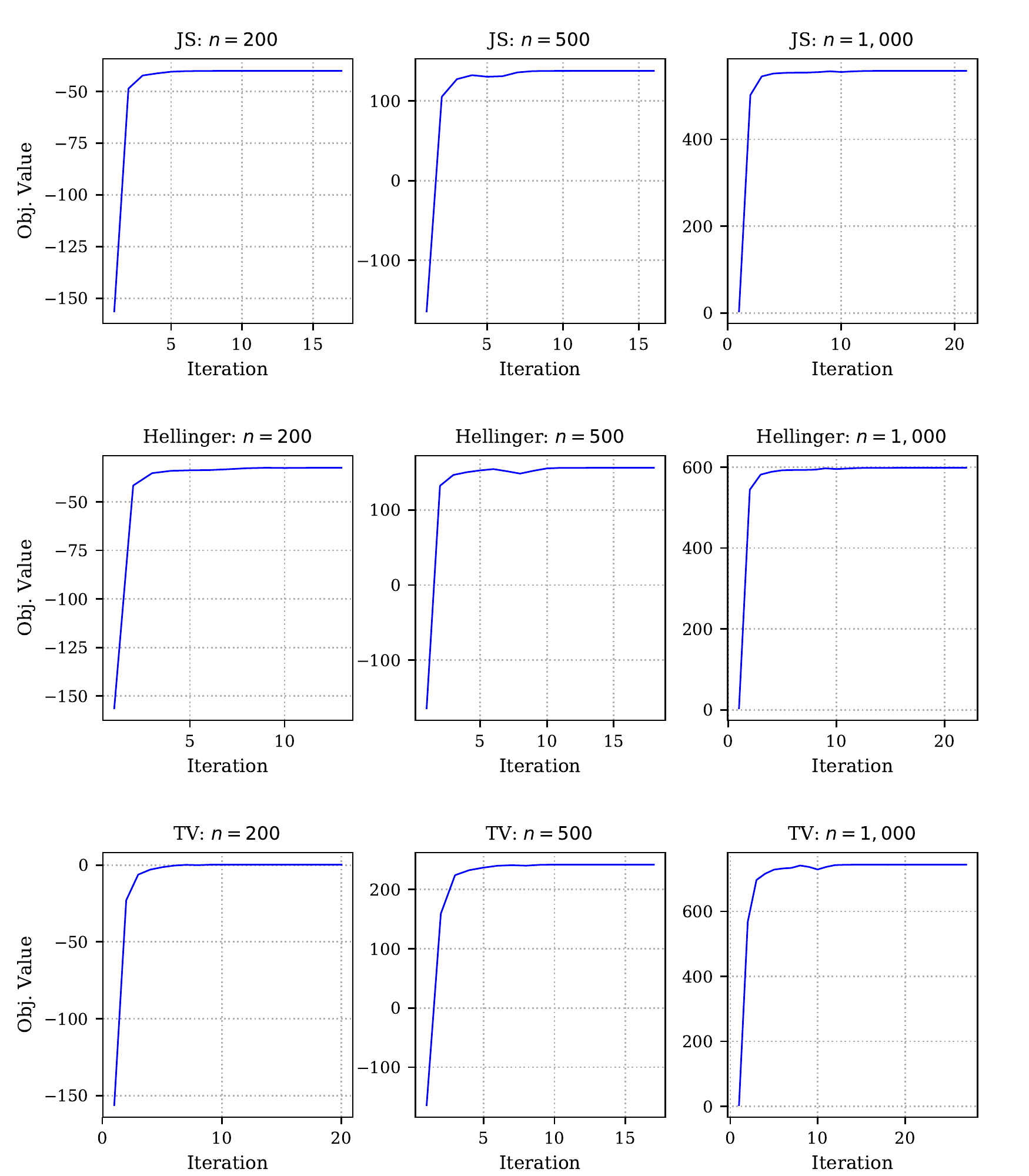}
    \caption{\color{rev} Convergence plots of the SCP--PGA algorithm (Algorithm \ref{algo:DRExact}) for DRRP portfolios with \(\omega=0.3\), \(T=5,000\), and varying values of \(n\).}
\label{fig:convergence}
\end{figure}

%-------------------------------------------------------------------------------------------
%       In-sample experiment
%-------------------------------------------------------------------------------------------
\subsection{In-sample experiment}\label{sec:InExp}

To better understand how the distributionally robust framework operates on our dataset, we present a set of in-sample trials over multiple {\color{rev}degrees of robustness}. The DRRP portfolio is, by design, a risk parity portfolio under the worst-case estimate of the risk measure (i.e., the portfolio variance). This means that the portfolio risk is perfectly diversified among the constituent assets with respect to a given estimate of the covariance matrix. However, {\color{rev}as shown in Corollary \ref{coll:RP},} it is paramount to understand that only one risk parity portfolio exists for a specific {\color{rev}estimate of the covariance matrix, provided this estimate is linearly independent of other estimates. Therefore, we should not expect the robust portfolios to satisfy the risk parity condition under the nominal estimate of the covariance matrix, and vice versa}. 

With that said, our goal is to evaluate how the asset allocations and risk contributions differ between the robust portfolios and the nominal portfolio. We use the 30 industry portfolios listed in Table \ref{table:assets} as our assets (\(n=30\)), and we use two years of weekly returns to estimate the covariance matrix, meaning we have 104 historical scenarios (\(T=104\)). Specifically, the data corresponds to the time period from 01--Jan--2008 to 31--Dec--2009.

We begin by inspecting the asset weights and risk contributions for robust portfolios built with a {\color{rev}degree of robustness} \(\omega=0.3\). The asset weights are shown in Figure \ref{fig:InSampleWeights}, and show that the DRRP portfolios exhibit a similar behaviour under all three statistical distance measures. Not only do the DRRP portfolios differ from the nominal portfolio, but the asset weights of all three DRRP portfolios also follow a similar pattern (i.e., the peaks and troughs in Figure \ref{fig:InSampleWeights} are similar for all four portfolios). We also note that the wealth allocation of the DRRP portfolios is less pronounced than that of the nominal portfolio, with the DRRP portfolios exhibiting a more even distribution of wealth.

\begin{figure}[!htbp]
\centering
    \includegraphics[width=0.95\textwidth]{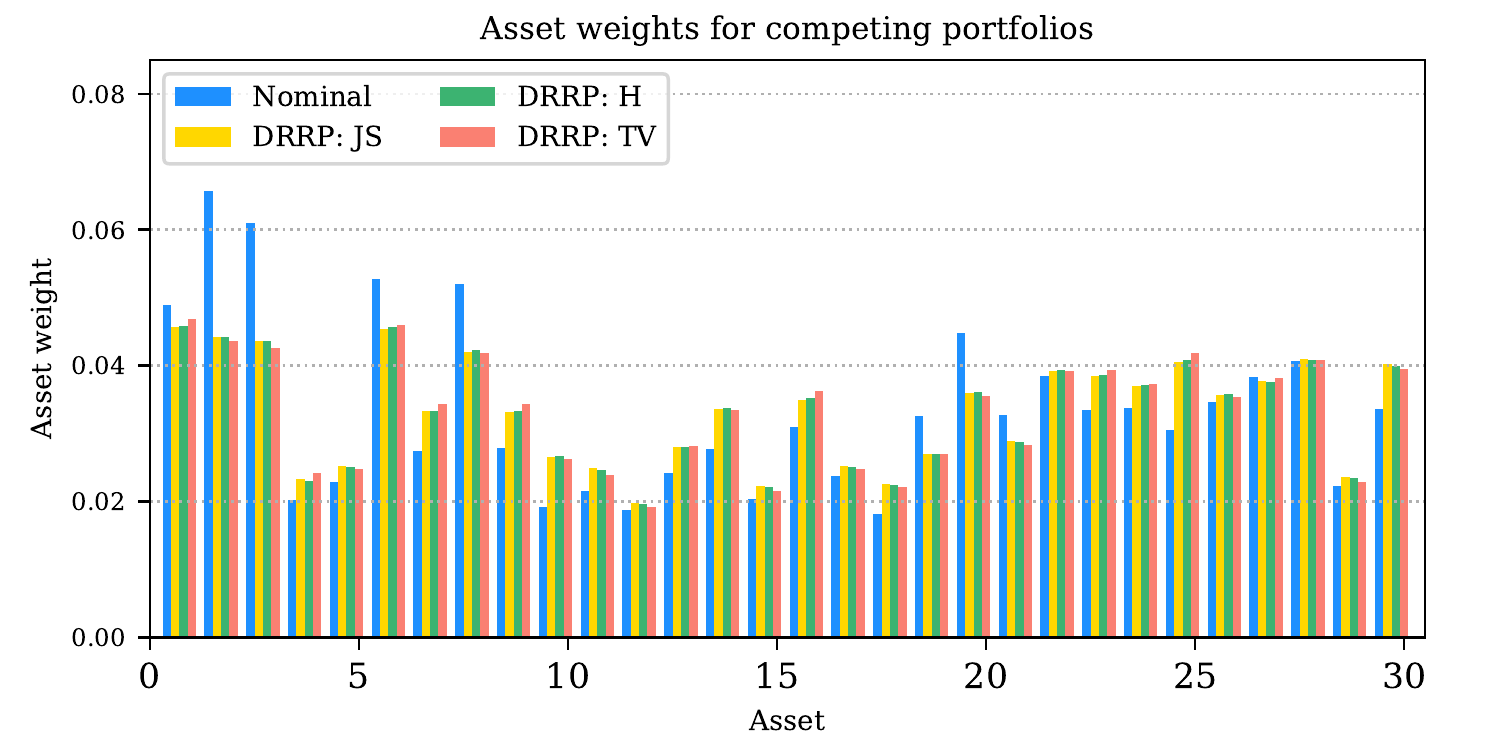}
    \caption{Asset weights of the nominal and DRRP portfolios with \(\omega=0.3\).}
\label{fig:InSampleWeights}
\end{figure}

Figure \ref{fig:InSampleRisk} presents a similar analysis, except we compare the risk contribution per asset with respect to some estimate of the covariance matrix. For convenience, we restate that the risk contribution per asset is defined as \(R_i \triangleq x_i [\hat{\bm{\Sigma}} \bm{x}]_i\) for some estimate \(\hat{\bm{\Sigma}}\). For a fair comparison, the top plot in Figure \ref{fig:InSampleRisk} shows the risk contribution per asset for all portfolios relative to the nominal estimate of the covariance matrix, \(\bm{\Sigma}^{\text{nom}} \triangleq \bm{\Sigma}(\bm{q})\). The remaining three plots compares the risk contributions of the nominal portfolio with respect to \(\bm{\Sigma}^{\text{JS}}\), \(\bm{\Sigma}^{\text{H}}\) and \(\bm{\Sigma}^{\text{TV}}\), respectively. These three matrices correspond to the estimated covariance matrix obtained after the convergence of Algorithm \ref{algo:DRExact} for the respective distance measure. 

The top plot in Figure \ref{fig:InSampleRisk} confirms the similarity between all three distributionally robust portfolios. For all risk contributions per asset, the three DRRP portfolios together are either lower or higher than the nominal portfolio (i.e., there is no asset where its risk contribution from a DRRP portfolio is higher than from the nominal portfolio while simultaneously lower from another DRRP portfolio). The DRRP portfolios choose the same assets to over- or under-contribute risk relative to the nominal portfolio, highlighting the structural similarity between the DRRP portfolios.

The three remaining plots in Figure \ref{fig:InSampleRisk} serve to show that all robust portfolios are true risk parity portfolios with respect to their corresponding estimate of the covariance matrix. As shown in the plots, the bars for the robust portfolios are of equal height. 

\begin{figure}[!htbp]
\centering
    \includegraphics[width=0.95\textwidth]{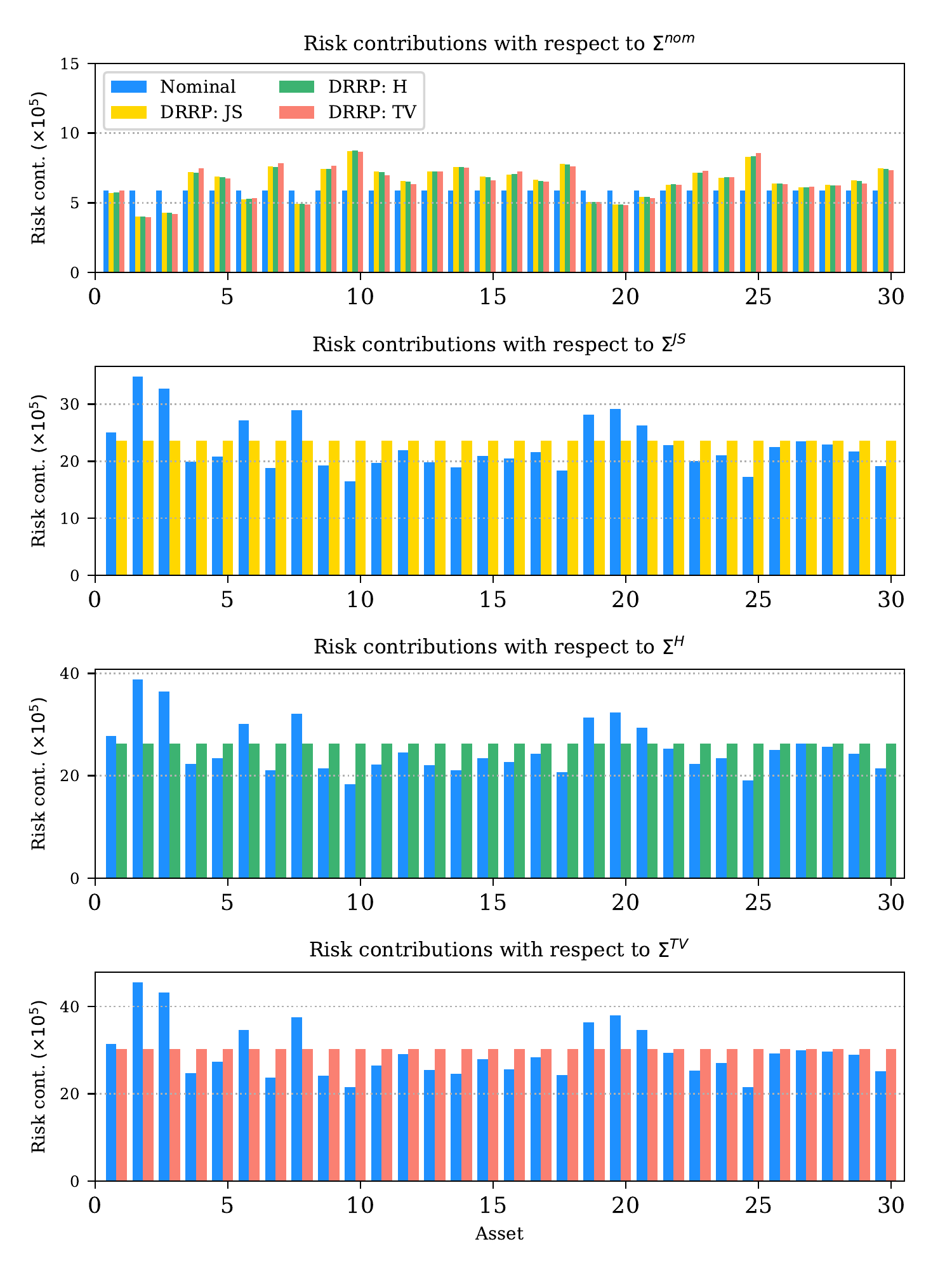}
    \caption{Risk contributions per asset of the nominal and DRRP portfolios with \(\omega=0.3\). The top plot shows the risk contributions with respect to the nominal estimate of the covariance matrix. The remaining plots show the risk contributions of the nominal and DRRP portfolios based on the robust estimates of the covariance matrix.}
\label{fig:InSampleRisk}
\end{figure}

The last component of the in-sample experiment replicates the same procedure, except we use varying {\color{rev}degrees of robustness} \(\omega\). For brevity, these results are summarized in Table \ref{table:inSample}. The table shows the total variance of the four competing portfolios with respect to the nominal estimate of the covariance matrix, as well as a pairwise comparison of the total variance of the robust portfolios against the nominal using the corresponding worst-case estimates of the covariance matrix. In addition, we report the level of risk concentration through the coefficient of variation (CV) of the risk contributions. The CV is calculated by taking the standard deviation of the risk contributions and dividing them by their average, i.e., 
%-^-^-^-^-^-^-^-^-^-^-^-^-^-^-^-^-^-^-^-^-^-^-
\begin{equation}
    \text{CV} = \frac{\text{SD}\big(\bm{x} \circ [\hat{\bm{\Sigma}} \bm{x}]\big)}{\frac{1}{n} \bm{x}^\top \bm{\Sigma} \bm{x}},
\label{eq:CV}
\end{equation}
%-^-^-^-^-^-^-^-^-^-^-^-^-^-^-^-^-^-^-^-^-^-^-
where `\(\circ\)' is the element-wise multiplication operator and \(\text{SD}(\cdot)\) computes the standard deviation of the corresponding vector. In theory, an optimal risk parity portfolio should have a CV of zero.

\begin{table}[!ht]
\footnotesize
\caption{Portfolio variance and CV based on the nominal and worst-case estimates of the asset covariance matrix}
\centering
\begin{tabular}{@{\extracolsep{-1pt}}l rrrr l rr l rr l rr}
\hline 
\rule{0pt}{3ex}   & \multicolumn{4}{c}{\(\bm{\Sigma}^{\text{nom}}\)} && \multicolumn{2}{c}{\(\bm{\Sigma}^{\text{JS}}\)} && \multicolumn{2}{c}{\(\bm{\Sigma}^{\text{H}}\)} && \multicolumn{2}{c}{\(\bm{\Sigma}^{\text{TV}}\)}\\[0.8ex]\cline{2-5} \cline{7-8} \cline{10-11} \cline{13-14}
\rule{0pt}{3ex} & \(\bm{x}^{\text{nom}}\) & \(\bm{x}^{\text{JS}}\) & \(\bm{x}^{\text{H}}\) & \(\bm{x}^{\text{TV}}\) && \(\bm{x}^{\text{nom}}\) & \(\bm{x}^{\text{JS}}\) && \(\bm{x}^{\text{nom}}\) & \(\bm{x}^{\text{H}}\) && \(\bm{x}^{\text{nom}}\) & \(\bm{x}^{\text{TV}}\) \\[1.5ex]
\hline 
\rule{0pt}{3ex} \(\bm{\omega=0.1}\)	& \\
	\rule{0pt}{3ex}Var. (\(\times 10^3\)) & 1.77 & 1.87 & 1.87 & 1.99 && 2.91 & 3.03 && 3.13 & 3.25 && 4.35 & 4.54\\[1.75ex]
	CV & 7e-16 & 0.10 & 0.10 & 0.19 && 0.10 & 6e-16 && 0.11 & 2e-16 && 0.22 & 2e-16\\[1.75ex]

\rule{0pt}{3ex} \(\bm{\omega=0.2}\)	& \\
	\rule{0pt}{3ex}Var. (\(\times 10^3\)) & 1.77 & 1.93 & 1.94 & 1.97 && 4.63 & 4.83 && 5.12 & 5.34 && 6.59 & 6.83\\[1.75ex]
	CV & 7e-16 & 0.15 & 0.15 & 0.19 && 0.17 & 3e-16 && 0.17 & 3e-16 && 0.21 & 4e-16\\[1.75ex]

\rule{0pt}{3ex} \(\bm{\omega=0.3}\)	& \\
	\rule{0pt}{3ex}Var. (\(\times 10^3\)) & 1.77 & 1.96 & 1.96 & 1.95 && 6.80 & 7.06 && 7.59 & 7.87 && 8.81 & 9.09\\[1.75ex]
	CV & 7e-16 & 0.18 & 0.18 & 0.188 && 0.20 & 6e-16 && 0.20 & 4e-16 && 0.20 & 3e-16\\[1.75ex]

\rule{0pt}{3ex} \(\bm{\omega=0.4}\)	& \\
	\rule{0pt}{3ex}Var. (\(\times 10^3\)) & 1.77 & 1.96 & 1.95 & 1.95 && 9.27 & 9.57 && 10.4 & 10.7 && 11.0 & 11.32\\[1.75ex]
	CV & 7e-16 & 0.18 & 0.18 & 0.18 && 0.20 & 3e-16 && 0.20 & 3e-16 && 0.20 & 2e-16\\[1.75ex]

\rule{0pt}{3ex} \(\bm{\omega=0.5}\)	& \\
	\rule{0pt}{3ex}Var. (\(\times 10^3\)) & 1.77 & 1.95 & 1.95 & 1.94 && 11.9 & 12.3 && 13.3 & 13.6 && 13.2 & 13.5\\[1.75ex]
	CV & 7e-16 & 0.19 & 0.19 & 0.18 && 0.20 & 3e-16 && 0.20 & 3e-16 && 0.20 & 2e-16\\[1.75ex]

\rule{0pt}{3ex} \(\bm{\omega=0.6}\)	& \\
	\rule{0pt}{3ex}Var. (\(\times 10^3\)) & 1.77 & 1.94 & 1.94 & 1.94 && 14.6 & 15.0 && 16.1 & 16.5 && 15.3 & 15.67\\[1.75ex]
	CV & 7e-16 & 0.19 & 0.19 & 0.19 && 0.20 & 3e-16 && 0.20 & 6e-16 && 0.20 & 5e-16\\[1.75ex]

\hline
\end{tabular}
\label{table:inSample}
\end{table}

The results in Table \ref{table:inSample} show that all portfolios have perfect risk diversification with respect to their corresponding estimates of the covariance matrix (i.e., the CV of the portfolios is approximately zero with respect to their corresponding instance of \(\hat{\bm{\Sigma}}\)). An interesting observation from Table \ref{table:inSample} is that the nominal portfolio has the lowest total variance when compared against the robust portfolios for all instances of the covariance matrix. We note that this observation does not fundamentally conflict with our objective, as our robust portfolios aim to diversify risk, and not minimize it. Nevertheless, the results suggest that these robust portfolios incur more ex ante risk when compared to the nominal portfolio.

%-------------------------------------------------------------------------------------------
%       Out-of-sample experiment
%-------------------------------------------------------------------------------------------
\subsection{Out-of-sample experiment}\label{sec:OutExp}

An overview of the out-of-sample experimental setup follows. Our portfolio constituents are the 30 assets listed in Table \ref{table:assets} (\(n=30\)). The dataset consists of weekly historical returns from 01--Jan--1998 to 31--Dec--2016, with the data obtained from \cite{French2020Data}. This is a rolling window experiment, where we use two years of weekly scenarios to calibrate our portfolios (\(T=104\)) and we hold these portfolios for six month before rebalancing them. All estimated parameters and weights are recalibrated every time we rebalance our portfolios. To exemplify our approach, consider the first investment period. We use the data from 01--Jan-1998 to 31--Dec-1999 to calibrate our initial portfolios, and then we hold and observe the out-of-sample performance from 01--Jan--2000 to 30--Jun--2000. Afterwards, we roll the calibration window forward and recalibrate and rebalance our portfolios using the preceding two-year period (01--Jul--1998 to 30--Jun--2000). We then observe the out-of-sample performance from 01--Jul--2000 to 31--Dec--2000. We repeat these steps until the end of the investment horizon. Our out-of-sample experiment runs from 01--Jan--2000 until 31--Dec--2016, meaning we have a total of 34 six-month out-of-sample investment periods. We record the wealth evolution of the portfolios over the entire horizon. Finally, we note that this experiment is non-exhaustive since the portfolio performance is highly dependent on our choice of assets and historical time period. However, having a diverse basket of assets representative of major U.S. industries and a 17-year out-of-sample investment period should suffice for our analysis. 

The first set of results, shown in Table \ref{table:outSample} and Figure \ref{fig:outSample}, correspond to risk parity portfolios with {\color{rev}degree of robustness} \(\omega = 0.15, 0.3, 0.45\). The top plot in Figure \ref{fig:outSample} shows the total wealth evolution of the nominal portfolio. The remaining three plots show the relative wealth of the robust portfolios. The `relative wealth' is defined as a percentage, \((W^t_i/W^t_{\text{nom}}-1)\times 100\), where \(W^t_i\) is the wealth of portfolio \(i\) at each weekly time step \(t\), while \(W^t_{\text{nom}}\) is the nominal portfolio's wealth. 

The robust portfolios exhibit a drop in their relative wealth over the bear market periods of 2000--2003 and 2008--2009. However, we note that the risk parity portfolios are not designed to minimize a portfolio's risk, but rather to be fully risk diverse. In turn, the results suggest that the robust portfolios are in a better position to take advantage of the subsequent bull market periods, where we can see sustained growth relative to the nominal. Moreover, the ex post portfolio performance aligns with our findings from the in-sample experiment, where we saw that the robust portfolios had a somewhat higher risk appetite given that they had a higher ex ante variance when compared to the nominal portfolio. 

\begin{figure}[!htbp]
\centering
    \includegraphics[width=0.95\textwidth]{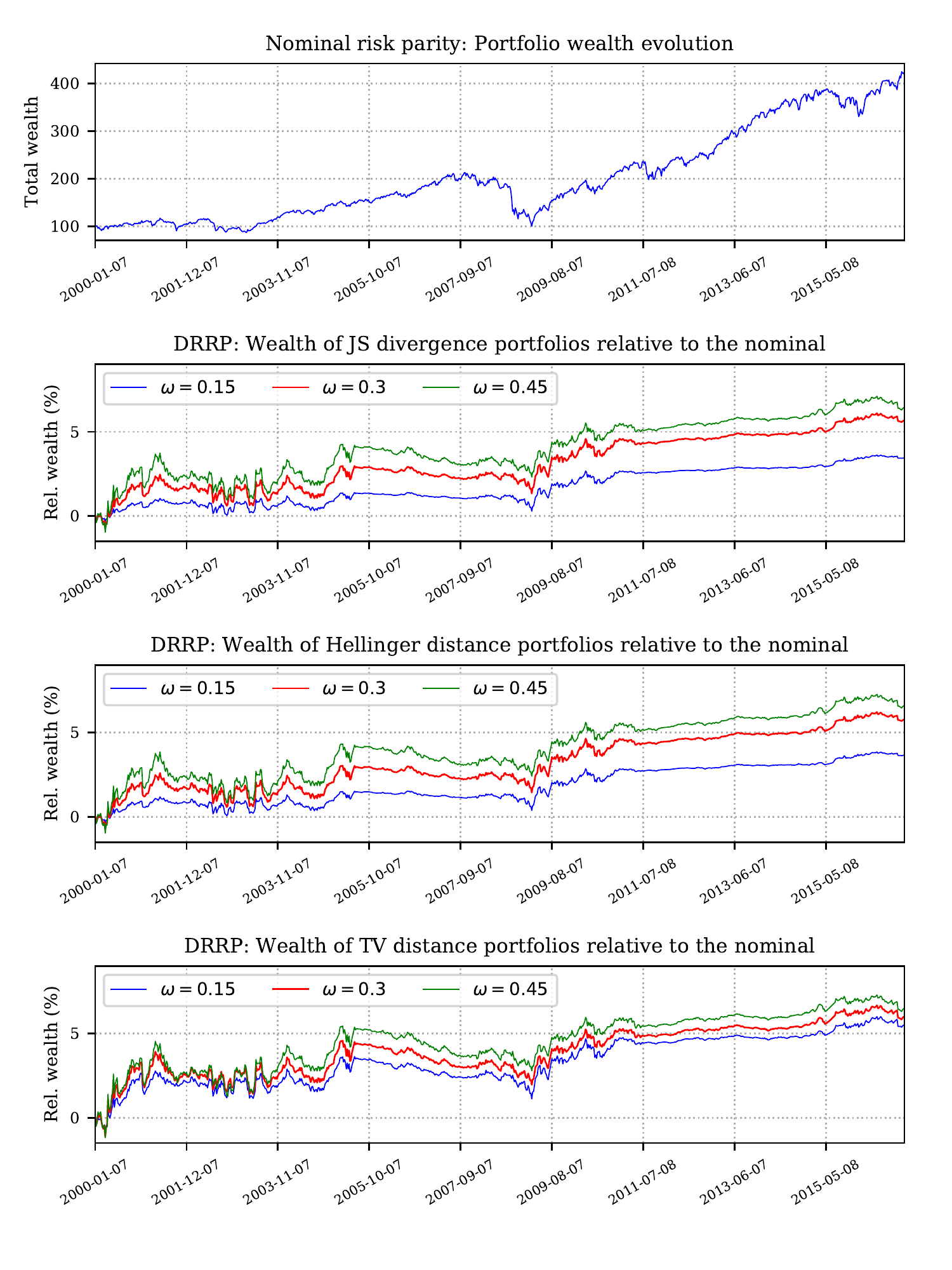}
    \caption{Wealth evolution for DRRP portfolios. The top plot shows the total wealth evolution of the nominal portfolio. The remaining three plots present the relative wealth evolution of the DRRP portfolios with respect to the nominal portfolio for varying {\color{rev}degrees of robustness}.}
\label{fig:outSample}
\end{figure}

We summarize the ex post performance in Table \ref{table:outSample}, where we show the annualized average excess return, annualized volatility,\footnote{The portfolio volatility is the square root of the ex post portfolio variance.} Sharpe ratio \cite{sharpe1994sharpe} and average turnover rate over the entire investment horizon (2000--2016). We also provide the subset of results corresponding to a bear market period and its subsequent recovery (2007--2011). The results consistently show that the robust portfolios are able to attain a higher average excess return while maintaining a similar level of volatility, leading to higher Sharpe ratios. However, as the Sharpe ratio increases so does the average turnover rate, which serves as a proxy of transaction costs. With that said, transaction costs are becoming increasingly negligible in modern financial markets. Moreover, we note that the turnover rates of risk parity portfolios are typically very low when compared against other asset allocation strategies such as MVO (e.g., see \cite{costa2020regime}), and our results in {\color{rev}Table} \ref{table:outSample} are no exception. Thus, the increased transaction costs incurred by the DRRP portfolios are somewhat negligible. Finally, we note that our observations are consistent over the 2007--2011, with the robust portfolios having a higher Sharpe ratio over this time period when compared to the nominal. 

\begin{table}[!ht]
\footnotesize
\caption{Summary of financial performance of the risk parity portfolios over the periods 2000--2016 and 2007--2011}
\centering
\begin{tabular}{@{\extracolsep{-1pt}}l rrrr l rrr l rrr}
\hline 
\rule{0pt}{2.75ex}   & Nom.\ \ & \multicolumn{3}{c}{JS} && \multicolumn{3}{c}{Hellinger} && \multicolumn{3}{c}{TV}\\[0.8ex]\cline{3-5} \cline{7-9} \cline{11-13}
\rule{0pt}{2.75ex} \qquad\qquad\(\omega=\) &  & 0.15 & 0.3 & 0.45 && 0.15 & 0.3 & 0.45 && 0.15 & 0.3 & 0.45\\[1ex]
\hline 
\rule{0pt}{3ex}\textbf{2000 -- 2016}	& \\
\rule{0pt}{3ex}Return (\%)	& 6.64  & 6.84  & 6.96  & 7.01  && 6.85 & 6.97 & 7.02 && 6.95  & 6.98  & 7.01\\[1.5ex]
{\color{rev}Volatility} (\%)					& 17.0  & 17.2  & 17.2  & 17.2  && 17.2  & 17.2  & 17.2  && 17.2  & 17.2  & 17.2\\[1.5ex]
Sharpe Ratio 				& 0.390 & 0.398 & 0.405 & 0.407 && 0.399 & 0.405 & 0.408 && 0.404 & 0.406 & 0.407\\[1.5ex]
Turnover 					& 0.100 & 0.126 & 0.151 & 0.167 && 0.128 & 0.152 & 0.168 && 0.160 & 0.172 & 0.180\\[1.5ex]
\rule{0pt}{3ex}\textbf{2007 -- 2011}	& \\
\rule{0pt}{3ex}Return (\%)	& 2.39  & 2.67  & 2.77  & 2.73  && 2.69  & 2.77  & 2.73  && 2.71  & 2.66  & 2.62\\[1.5ex]
{\color{rev}Volatility} (\%)					& 23.3  & 23.7  & 23.8  & 23.8  && 23.7  & 23.8  & 23.8  && 23.9  & 23.8  & 23.8\\[1.5ex]
Sharpe Ratio 				& 0.102 & 0.113 & 0.116 & 0.115 && 0.113 & 0.116 & 0.115 && 0.114 & 0.112 & 0.110\\[1.5ex]
Turnover 					& 0.098 & 0.123 & 0.150 & 0.166 && 0.125 & 0.151 & 0.166 && 0.162 & 0.172 & 0.177\\[1ex]
\hline
\end{tabular}
\label{table:outSample}
\end{table}

%%%%%%%%%%%%%%%%%%%%%%%%%%%%%%%%%%%%%%%%%%%%%%%%%%%%%%%%%%%%%%%%%%%%%%%%%%%%%%%%%%%%%%%%%%%%%%
% 		CONCLUSION
%%%%%%%%%%%%%%%%%%%%%%%%%%%%%%%%%%%%%%%%%%%%%%%%%%%%%%%%%%%%%%%%%%%%%%%%%%%%%%%%%%%%%%%%%%%%%%
\section{Conclusion} \label{sec:conclusion}

This paper introduced a DRO problem specifically designed for risk parity portfolios. Distributional robustness is introduced through a discrete probability distribution that allows us to break away from the assumption that all scenarios in a data-driven parameter estimation process are equally likely. Instead, we can model the probability attached to each  {\color{rev}scenario} as a decision variable, which in turn allows us to formulate a minimax problem that seeks risk parity while simultaneously seeking the most adversarial instance of the discrete distribution such that the portfolio variance is maximized. Our modelling framework allows us to define the probability ambiguity set using any convex function {\color{rev}as a  measure of statistical distance between the adversarial probability distribution and the `equally likely' nominal assumption}. We exemplify this by implementing three alternative statistical distances: JS, Hellinger, and TV.

The DRRP problem is a constrained convex--concave minimax problem over convex sets. {\color{rev}Using the dual problem of the maximization step, the problem can be recast as a straightforward convex robust counterpart. However, in practice, the complexity of the robust counterpart means it is difficult to solve numerically. Instead, we} apply projected gradient methods to iterate over the DRRP minimax problem in both descent and ascent directions. The projections ensure that we retain feasibility after each iteration. However, iteratively moving in both directions may slow down convergence. 

{\color{rev}Therefore}, we propose a novel algorithmic framework to solve the DRRP problem. The proposed SCP--PGA algorithm exploits the strict convexity of the risk parity problem, which guarantees that we have a unique risk parity portfolio for each instance of the adversarial probability distribution. Thus, we aim to iteratively ascend in the probability space through PGA while solving the corresponding convex risk parity problem after every iteration. The SCP--PGA algorithm dramatically improves computational runtime and, by design, retains the global convergence properties of general projected gradient methods. Our numerical results show that the SCP--PGA algorithm is computationally tractable and scalable. From a financial perspective, our experiments show that a DRRP portfolio is able to attain a higher risk-adjusted return when compared to the nominal portfolio. 

Finally, we note that the DRRP problem can be adapted to solve other portfolio selection problems that may benefit from distributional robustness. Moreover, the general design of the SCP--PGA algorithm should allow it to solve other types of constrained convex--concave minimax problems, including problems in other disciplines. These topics are the subject of future research. 

%%%%%%%%%%%%%%%%%%%%%%%%%%%%%%%%%%%%%%%%%%%%%%%%%%%%%%%%%%%%%%%%%%%%%%%%%%%%%%%%%%%%%%%%%%%%%%%
%% 		DECLARATIONS OF INTEREST
%%%%%%%%%%%%%%%%%%%%%%%%%%%%%%%%%%%%%%%%%%%%%%%%%%%%%%%%%%%%%%%%%%%%%%%%%%%%%%%%%%%%%%%%%%%%%%%
%\section*{Declarations of Interest}
%
%The authors report no conflicts of interest.

%%%%%%%%%%%%%%%%%%%%%%%%%%%%%%%%%%%%%%%%%%%%%%%%%%%%%%%%%%%%%%%%%%%%%%%%%%%%%%%%%%%%%%%%%%%%%%
% 		REFERENCES
%%%%%%%%%%%%%%%%%%%%%%%%%%%%%%%%%%%%%%%%%%%%%%%%%%%%%%%%%%%%%%%%%%%%%%%%%%%%%%%%%%%%%%%%%%%%%%%

\bibliographystyle{apa}
\bibliography{main.bib}

%%%%%%%%%%%%%%%%%%%%%%%%%%%%%%%%%%%%%%%%%%%%%%%%%%%%%%%%%%%%%%%%%%%%%%%%%%%%%%%%%%%%%%%%%%%%%%
% 		APPENDICES
%%%%%%%%%%%%%%%%%%%%%%%%%%%%%%%%%%%%%%%%%%%%%%%%%%%%%%%%%%%%%%%%%%%%%%%%%%%%%%%%%%%%%%%%%%%%%%
\begin{appendices}

\section{Numerical implementation of statistical distances} \label{app:distance}

Here we describe how to numerically implement the squared Hellinger distance in \eqref{eq:delH} and the TV distance in \eqref{eq:delTV}. We use either of these two distance measures to define the ambiguity set \(\mathcal{U}_{\bm{p}}\), and then use the set to construct the corresponding Euclidean projection optimization problem in \eqref{eq:maxRPproj}. However, in their current form, most optimization solvers will reject them.

If we wish to use the squared Hellinger distance in \eqref{eq:delH}, then the projection optimization problem can be implemented as follows.
%-^-^-^-^-^-^-^-^-^-^-^-^-^-^-^-^-^-^-^-^-^-^-^-^-^-^-^-^-^-^-
\[
\begin{aligned}
&\begin{aligned}
    & \min_{\bm{p},\bm{r}} & \| \bm{u} - \bm{p}\|_2^2\end{aligned}\\
&\begin{aligned}
&\ \text{s.t.} &\bm{1}^T \bm{p} 	&= 1,\\
		& & \frac{1}{2}\sum_{t=1}^T p_t - 2 r_t \sqrt{q_t} + q_t	&\leq d_{\text{H}}, \\
		& & p_t &\geq r_t^2\quad \text{for } t = 1,\dots,T, \\
		& &\bm{p},\ \bm{r}					&\geq 0,
\end{aligned}
\end{aligned}
\]
%-^-^-^-^-^-^-^-^-^-^-^-^-^-^-^-^-^-^-^-^-^-^-^-^-^-^-^-^-^-^-
where \(\bm{u}\in\mathbb{R}^T\) is some arbitrary vector that we wish to project onto the set \(\mathcal{U}_{\bm{p}}\), while \(\bm{r}\in\mathbb{R}^T\) is an auxiliary variable that serves as a placeholder for the square root of each element of \(\bm{p}\). As before, \(\bm{q}\in\mathcal{P}\) is the nominal probability distribution, while \(d_{\text{H}}\) is the maximum permissible distance in \eqref{eq:limitH} and is defined by the number of scenarios \(T\) and the investor's {\color{rev}desired degree of robustness} \(\omega\).

On the other hand, if we wish to use the TV distance in \eqref{eq:delTV}, then the projection optimization problem can be implemented as follows.
%-^-^-^-^-^-^-^-^-^-^-^-^-^-^-^-^-^-^-^-^-^-^-^-^-^-^-^-^-^-^-
\[
\begin{aligned}
&\begin{aligned}
    & \min_{\bm{p},\bm{\zeta}} & \| \bm{u} - \bm{p}\|_2^2\end{aligned}\\
&\begin{aligned}
    &\ \text{s.t.} &\bm{1}^T \bm{p} 	&= 1,\\
    & & \frac{1}{2}\sum_{t=1}^T \zeta_t	&\leq d_{\text{TV}}, \\
		& & \zeta_t &\geq p_t - q_t\quad \text{for } t = 1,\dots,T, \\
		& & \zeta_t &\geq q_t - p_t\quad \text{for } t = 1,\dots,T, \\
		& &\bm{p}					&\geq 0,
\end{aligned}
\end{aligned}
\]
%-^-^-^-^-^-^-^-^-^-^-^-^-^-^-^-^-^-^-^-^-^-^-^-^-^-^-^-^-^-^-
where \(\bm{\zeta}\in\mathbb{R}^T\) is an auxiliary variable that represents the absolute value of the difference between the elements of \(\bm{p}\) and \(\bm{q}\). The maximum permissible distance \(d_{\text{TV}}\) is defined by the number of scenarios \(T\) and the investor's {\color{rev}desired degree of robustness} \(\omega\) as shown in \eqref{eq:limitTV}.

\end{appendices}

\end{document}